\documentclass{article}

\usepackage{graphicx}
\usepackage{float}
\usepackage{bm}
\usepackage{bbm}
\usepackage{amsthm}
\usepackage{amsmath}
\usepackage{amssymb}
\usepackage{mathrsfs}
\usepackage{tikz-cd}
\usepackage{stmaryrd}
\usepackage{enumitem}
\usepackage{tabularx}
\usepackage{hyperref}
\usepackage{cleveref}

\crefformat{footnote}{#2\footnotemark[#1]#3}

\theoremstyle{definition}
\newtheorem{Definition}{Definition}[subsection]

\theoremstyle{plain}
\newtheorem{Theorem}[Definition]{Theorem}

\newcounter{mainthm}

\newtheorem{maintheorem}[mainthm]{Theorem}

\newtheorem{maincorollary}[mainthm]{Corollary}

\theoremstyle{plain}
\newtheorem{Proposition}[Definition]{Proposition}

\theoremstyle{plain}
\newtheorem{Lemma}[Definition]{Lemma}

\theoremstyle{plain}
\newtheorem{Corollary}[Definition]{Corollary}

\theoremstyle{definition}
\newtheorem{Example}[Definition]{Example}

\theoremstyle{definition}
\newtheorem{Conjecture}[Definition]{Conjecture}

\theoremstyle{definition}

\theoremstyle{definition}

\theoremstyle{remark}
\newtheorem{Remark}[Definition]{Remark}

\title{Relative Invertibility and Full Dualizability of Finite Braided Tensor Categories}
\author{Thibault D. D\'ecoppet}
\date{June 2025}

\begin{document}

\bibliographystyle{alpha}

\maketitle

\begin{center}
\textit{D'apr\`es des vieux cahiers.}
\end{center}

    \hspace{1cm}
    \begin{abstract}
        Fix a finite symmetric tensor category $\mathcal{E}$ over an algebraically closed field. We derive an $\mathcal{E}$-enriched version of Shimizu's characterizations of non-degeneracy for finite braided tensor categories.
        In order to do so, we consider, associated to any $\mathcal{E}$-enriched finite braided tensor category $\mathcal{A}$ satisfying a mild technical assumption, a Hopf algebra $\mathbb{F}_{\mathcal{E}/\mathcal{A}}$ in $\mathcal{A}$.
        This is a generalization of Lyubashenko's universal Hopf algebra $\mathbb{F}_{\mathcal{A}}$ in $\mathcal{A}$. In fact, we show that there is a short exact sequence $\mathbb{F}_{\mathcal{E}}\rightarrow\mathbb{F}_{\mathcal{A}}\rightarrow\mathbb{F}_{\mathcal{E}/\mathcal{A}}$ of Hopf algebras in $\mathcal{A}$, and that the canonical pairing on $\mathbb{F}_{\mathcal{A}}$ descends to a pairing $\omega_{\mathcal{E}/\mathcal{A}}$ on $\mathbb{F}_{\mathcal{E}/\mathcal{A}}$.
        We prove that $\mathcal{A}$ is $\mathcal{E}$-non-degenerate, i.e.\ its symmetric center is exactly $\mathcal{E}$, if and only if the pairing $\omega_{\mathcal{E}/\mathcal{A}}$ is non-degenerate.
        
        We then use the above characterization to show that an $\mathcal{E}$-enriched finite braided tensor category is invertible in the Morita 4-category of $\mathcal{E}$-enriched pre-tensor categories if and only if it is $\mathcal{E}$-non-degenerate.
        As an application of our relative invertibility criterion, we extend the full dualizability result of Brochier-Jordan-Snyder by showing that a finite braided tensor category is fully dualizable as an object of the Morita 4-category of braided pre-tensor categories if its symmetric center is separable.
    \end{abstract}

\tableofcontents

\section*{Introduction}
\addcontentsline{toc}{section}{Introduction}

We continue the study of the dualizability of the higher Morita category of braided pre-tensor categories initiated in \cite{BJS}. A symmetric monoidal $n$-category is fully dualizable if its objects have duals and its $k$-morphisms have adjoints for all $k<n$. More generally, an object of such a category is called fully dualizable if it lies in a fully dualizable monoidal subcategory. Our interest in this notion stems from the cobordism hypothesis \cite{BD,L}, which proposes an equivalence between fully extended framed $n$-dimensional topological field theories valued in a target symmetric monoidal $n$-category and fully dualizable objects therein. Here fully extended refers to topological field theories that can be evaluated on all $k$-manifolds, possibly with corners, for $k\leq n$. The cobordism hypothesis asserts more precisely that such a topological field theory is determined by its value on the point, a fully dualizable object. This motivates the quest for interesting examples of fully dualizable objects.

Of particular interest are the higher Morita categories of $E_m$-algebras in a symmetric monoidal $n$-category constructed in \cite{JFS, Hau}. In fact, it is already instructive to investigate the dualizability properties of the classical Morita 2-category of associative algebras, bimodules, and bimodule morphisms. It is a folklore result that an algebra $A$ is fully dualizable in the aforementioned Morita 2-category if and only if it is \textit{separable}, i.e.\ $A$ is projective as an $A$-$A$-bimodule. For our purposes, it is illuminating to express this last condition differently. Namely, every separable algebra is in particular finite semisimple. Moreover, a finite semisimple algebra $A$ is separable if and only if its center $Z(A)$ is separable. Altogether, we find that an algebra $A$ is fully dualizable in the Morita 2-category of algebras if and only if it is finite semisimple and its center is separable.

The Morita 3-category of finite tensor categories in the sense of \cite{EGNO} was studied in \cite{DSPS:dualizable}. In particular, it was established that a finite tensor category is fully dualizable if and only if it satisfies an algebraic condition called separability:\ A finite tensor category $\mathcal{C}$ is said to be \textit{separable} if it is semisimple, and its Drinfeld center $\mathcal{Z}(\mathcal{C})$ is also semisimple. Over an algebraically closed field, this is equivalent to requiring that $\mathcal{C}$ has non-zero global, also called categorical or quantum, dimension. Strikingly, over fields of characteristic zero, every finite semisimple tensor category has non-zero global dimension \cite{ENO1}, and is therefore separable. This explains why this condition is often omitted. We emphasize that this does not hold over fields of positive characteristic.

We will be mainly interested in the Morita 4-category of braided pre-tensor categories introduced in \cite{BJS}. For technical reasons, (finite) braided tensor categories cannot be assembled into a Morita 4-category, which explains the need for a more general notion. They show that every separable braided tensor category is a fully dualizable object of the Morita 4-category of pre-tensor categories. Notably, separability is not a necessary condition for full dualizability. Namely, following \cite{Bru, Mu}, we say that a finite braided tensor category $\mathcal{A}$ is non-degenerate if its symmetric center $\mathcal{Z}_{(2)}(\mathcal{A})$ is trivial, i.e.\ equivalent to the category vector spaces -- and therefore separable. It was observed in \cite{BJSS} that every non-degenerate finite braided tensor categories is an invertible, and a fortiori fully dualizable, object in the Morita 4-category of braided pre-tensor categories.

The invertibility result of \cite{BJSS} relies heavily on the characterizations of non-degeneracy for finite braided tensor categories obtained in \cite{Shi:nondeg}, which we now review. Let $\mathcal{A}$ be a finite braided tensor category. We say that $\mathcal{A}$ is factorizable if the canonical braided tensor functor $\mathcal{A}\boxtimes\mathcal{A}^{\mathrm{rev}}\rightarrow\mathcal{Z}(\mathcal{A})$ into the Drinfeld center of $\mathcal{A}$ is an equivalence. This notion of factorizability for finite braided tensor categories was introduced in \cite{ENO0} as a generalization of a factorizability criterion for quasi-triangular Hopf algebras. In a different direction, recall from \cite{Lyu:squared}, based on the earlier work \cite{Maj, Lyu:ribbon}, that, associated to $\mathcal{A}$, there is a universal Hopf algebra $\mathbb{F}_{\mathcal{A}}$ in $\mathcal{A}$. Moreover, it was shown in \cite{Lyu:modular} that the Hopf algebra $\mathbb{F}_{\mathcal{A}}$ carries a canonical pairing $\omega_{\mathcal{A}}$, whose non-degeneracy plays a key role in applications \cite{Lyu:invariant, KL}. Remarkably, it was proven in \cite{Shi:nondeg} that $\mathcal{A}$ is non-degenerate if and only if it is factorizable if and only if the canonical pairing $\omega_{\mathcal{A}}$ is non-degenerate.

We improve on the dualizability results of \cite{BJS, BJSS} by showing that any finite braided tensor category whose symmetric center is separable is a fully dualizable object in the Morita 4-category of braided pre-tensor categories. In order to do so, fixing a finite symmetric tensor category, we consider the Morita 4-category of $\mathcal{E}$-enriched pre-tensor categories, and study invertible objects therein. This is achieved by developing an enriched or relative version of the characterizations of non-degeneracy established in \cite{Shi:nondeg}.

\subsection*{Results}
\addcontentsline{toc}{subsection}{Results}

Throughout, we work over a fixed algebraically closed field $\mathbbm{k}$. Our finite tensor categories are the $\mathrm{ind}$-completions of the finite multitensor categories of \cite{EGNO}.

\subsubsection*{The Relative Universal Hopf Algebra}

Let us now fix a finite symmetric tensor category $\mathcal{E}$. An $\mathcal{E}$-enriched finite braided tensor category is a finite braided tensor category $\mathcal{A}$ equipped with a symmetric tensor functor $F:\mathcal{E}\rightarrow\mathcal{Z}_{(2)}(\mathcal{A})$ to the symmetric center of $\mathcal{A}$. This notion can be traced back to \cite{DNO}. The relative Deligne tensor product $\mathcal{A}\boxtimes_{\mathcal{E}}\mathcal{A}^{\mathrm{mop}}$ is always a finite pre-tensor category, but arbitrary compact objects typically fail to have duals. In order to remedy this issue, we introduce the following notion:\ We say that an $\mathcal{E}$-enriched braided tensor category $\mathcal{A}$ is \textit{faithfully flat} if $F:\mathcal{E}\hookrightarrow\mathcal{Z}_{(2)}(\mathcal{A})$ is fully faithful. This ensures that $\mathcal{A}\boxtimes_{\mathcal{E}}\mathcal{A}^{\mathrm{mop}}$ is a finite tensor category.

Let $\mathcal{A}$ be a faithfully flat $\mathcal{E}$-enriched finite braided tensor category. It follows from Tannakian reconstruction, see for instance \cite{LM}, that there exists a Hopf algebra $\mathbb{F}_{\mathcal{E}/\mathcal{A}}$ in $\mathcal{A}$ such that $$\mathrm{Comod}_{\mathcal{A}}(\mathbb{F}_{\mathcal{E}/\mathcal{A}})\simeq\mathcal{A}\boxtimes_{\mathcal{E}}\mathcal{A}^{\mathrm{mop}}$$ as finite tensor categories. We will refer to $\mathbb{F}_{\mathcal{E}/\mathcal{A}}$ as the relative universal Hopf algebra of $\mathcal{A}$. In the special case $\mathcal{E} = \mathrm{Vec}$, this recovers precisely the universal Hopf algebra $\mathbb{F}_{\mathcal{A}}$ of \cite{Lyu:squared}. Additionally, recall from \cite{Lyu:modular} that $\mathbb{F}_{\mathcal{A}}$ carries a canonical pairing $\omega_{\mathcal{A}}$, which plays a key role in applications \cite{KL, Shi:nondeg}. It is natural to wonder about the precise relation between $\mathbb{F}_{\mathcal{A}}$ and $\mathbb{F}_{\mathcal{E}/\mathcal{A}}$, and whether the latter Hopf algebra sports a pairing.

\begin{maintheorem}
Let $\mathcal{A}$ be a faithfully flat $\mathcal{E}$-enriched braided tensor category. There is an exact sequence of Hopf algebras in $\mathcal{A}$ $$\mathbb{F}_{\mathcal{E}}\hookrightarrow\mathbb{F}_{\mathcal{A}}\twoheadrightarrow\mathbb{F}_{\mathcal{E}/\mathcal{A}}.$$ Moreover, the canonical pairing $\omega_{\mathcal{A}}$ descends to a pairing $\omega_{\mathcal{E}/\mathcal{A}}$ on $\mathbb{F}_{\mathcal{E}/\mathcal{A}}$.
\end{maintheorem}

\subsubsection*{Relative Non-Degeneracy and Relative Invertibility}

Using the universal relative Hopf algebra $\mathbb{F}_{\mathcal{E}/\mathcal{A}}$, and more specifically its pairing $\omega_{\mathcal{E}/\mathcal{A}}$, we obtain the following generalization of \cite[Theorem 1.1]{Shi:nondeg} and \cite[Theorem 1.6]{BJSS}.

\begin{maintheorem}\label{thm:mainnondegeneracy}
Let $\mathcal{A}$ be an $\mathcal{E}$-enriched finite braided tensor category such that $F:\mathcal{E}\rightarrow\mathcal{Z}_{(2)}(\mathcal{A})$ is faithful. The following conditions on $\mathcal{A}$ are equivalent:
\begin{enumerate}
    \item It is $\mathcal{E}$-non-degenerate, i.e.\ the natural functor $\mathcal{E}\rightarrow \mathcal{Z}_{(2)}(\mathcal{A})$ is an equivalence.
    \item It is $\mathcal{E}$-factorizable, i.e.\ the natural functor $\mathcal{A}\boxtimes_{\mathcal{E}}\mathcal{A}^{\mathrm{rev}}\rightarrow \mathcal{Z}(\mathcal{A},\mathcal{E})$ is an equivalence.
    \item It is $\mathcal{E}$-cofactorizable, i.e.\ the natural functor $\mathrm{HC}_{\mathcal{E}}(\mathcal{A})\rightarrow\mathrm{End}_{\mathcal{E}}(\mathcal{A})$ is an equivalence.
    \item It is faithfully flat and the canonical pairing $\omega_{\mathcal{E}/\mathcal{A}}$ on $\mathbb{F}_{\mathcal{E}/\mathcal{A}}$ is non-degenerate.
\end{enumerate}
\end{maintheorem}

\noindent We note that the faithfulness condition on $F$ is a very mild one. In fact, it is only necessary to exclude pathological examples. For instance, it is well-known that $F$ is automatically faithful if $\mathcal{E}$ has simple monoidal unit.

Thanks to a general criterion of \cite{BJSS}, the above theorem admits a higher categorical interpretation. More precisely, recall that we have fixed $\mathcal{E}$, a finite symmetric tensor category. We can then consider the Morita 4-category $\mathrm{Mor}^{\mathrm{pre}}_2(\mathbf{Pr}_{\mathcal{E}})$ of $\mathcal{E}$-enriched braided pre-tensor categories. This is a variant of the Morita 4-category $\mathrm{Mor}^{\mathrm{pre}}_2(\mathbf{Pr})$ of braided pre-tensor categories introduced in \cite{BJS}. In particular, objects of $\mathrm{Mor}^{\mathrm{pre}}_2(\mathbf{Pr}_{\mathcal{E}})$ are $\mathcal{E}$-enriched braided pre-tensor categories, and 1-morphisms are $\mathcal{E}$-enriched central pre-tensor categories. Thanks to the general invertibility criterion of \cite{BJSS}, we obtain the following result.

\begin{maincorollary}\label{cor:maininvertibility}
Let $\mathcal{A}$ be an $\mathcal{E}$-enriched finite braided tensor category. Then, $\mathcal{A}$ is invertible as an object of $\mathrm{Mor_2^{pre}}(\mathbf{Pr}_{\mathcal{E}})$ if and only if $\mathcal{A}$ is $\mathcal{E}$-non-degenerate.
\end{maincorollary}

A conceptually related but largely orthogonal result was obtained in \cite{Kin}. Over a field of characteristic zero, it was proven that the (non-finite) braided tensor categories of representations of Lusztig's divided power quantum group at certain roots of unity are invertible when enriched over their (non-finite) symmetric centers. We suspect that our last theorem above can be used to remove the restrictions on the root of unity.

In light of the above results, it is natural to wonder about the structure of the Picard group of the symmetric monoidal 4-category $\mathrm{Mor_2^{pre}}(\mathbf{Pr}_{\mathcal{E}})$, that is, its group of equivalence classes of invertible objects. In the case $\mathcal{E} = \mathrm{Vec}$, the Picard group of $\mathrm{Mor_2^{pre}}(\mathbf{Pr})$ was put forward in \cite{BJSS} as a generalization of the (quantum) Witt group of non-degenerate separable braided tensor categories introduced in \cite{DMNO}. This Witt group has been the subject of much interest because it offers an alternative to the essentially inaccessible problem of classifying all non-degenerate separable braided tensor categories. More generally, associated to any separable symmetric tensor category $\mathcal{E}$, a relative Witt group of $\mathcal{E}$-non-degenerate separable braided tensor categories was constructed in \cite{DNO}. The structure of these groups offers some insight into the classification of all separable braided tensor categories. We view the Picard group of $\mathrm{Mor_2^{pre}}(\mathbf{Pr}_{\mathcal{E}})$, for any finite symmetric tensor category $\mathcal{E}$, as a generalization of these relative Witt groups. These Picard groups are particularly intriguing in positive characteristic owing to the existence of exotic finite symmetric tensor categories \cite{BEO}.

\subsubsection*{Full Dualizability}

Recall that a finite tensor category is called separable if it is semisimple and its Drinfeld center is also semisimple. This condition played a key role in the study of the dualizability of finite tensor categories undertaken in \cite{DSPS:dualizable}. Namely, they showed that separable tensor categories are fully dualizable objects in the Morita 3-category of finite tensor categories. Subsequently, it was shown in \cite[Theorem 1.10]{BJS} that every separable finite braided tensor category is a fully dualizable object of the Morita 4-category $\mathrm{Mor_2^{pre}}(\mathbf{Pr})$ of braided pre-tensor categories. In a different direction, it was also shown in \cite[Theorem 1.1]{BJSS} that every non-degenerate finite braided tensor category is an invertible, and a fortiori fully dualizable, object of $\mathrm{Mor_2^{pre}}(\mathbf{Pr})$. As an application of Corollary \ref{cor:maininvertibility}, we obtain a common generalization of these last two results pertaining to the dualizability of $\mathrm{Mor_2^{pre}}(\mathbf{Pr})$, thereby answering a question raised in \cite[Remark 3.11]{BJSS}.

\begin{maintheorem}\label{main:fulldualizability}
A finite braided tensor category $\mathcal{A}$ is fully dualizable as an object of $\mathrm{Mor_2^{pre}}(\mathbf{Pr})$ if its symmetric center $\mathcal{Z}_{(2)}(\mathcal{A})$ is separable.
\end{maintheorem}

\noindent Our proof can roughly be summarized as follows:\ If $\mathcal{A}$ is a finite braided tensor category with separable symmetric center $\mathcal{E}:=\mathcal{Z}_{(2)}(\mathcal{A})$, then we have seen that it defines an invertible object of $\mathrm{Mor_2^{pre}}(\mathbf{Pr}_{\mathcal{E}})$. There is a (non-monoidal) functor $\mathrm{Mor_2^{pre}}(\mathbf{Pr}_{\mathcal{E}})\rightarrow \mathrm{Mor_2^{pre}}(\mathbf{Pr})$ forgetting the enrichment. As $\mathcal{A}$ is invertible as an $\mathcal{E}$-enriched tensor category, its invertibility data is preserved by this functor. The key observation is that the full dualizability data for the finite braided tensor category $\mathcal{A}$ can be expressed using this invertibility data together with the full dualizability data for $\mathcal{E}$, which exists thanks to \cite{BJS} as $\mathcal{E}=\mathcal{Z}_{(2)}(\mathcal{A})$ is separable by assumption.

In characteristic zero, we also show that the converse of the above theorem holds, i.e.\ we prove that a finite braided tensor category is fully dualizable in $\mathrm{Mor_2^{pre}}(\mathbf{Pr})$ if and only if its symmetric center is separable. We expect that this holds irrespective of the characteristic. In a different direction, let us also point out that a specific case in which our theorem applies, but neither \cite{BJS} nor \cite{BJSS} does, is when $\mathcal{A}$ is a slightly degenerate finite braided tensor category, that is, we have $\mathcal{Z}_{(2)}(\mathcal{A})=\mathrm{sVec}$. In characteristic zero, concrete examples are provided by the finite braided tensor categories constructed in \cite{Neg:arbitrary} for specific pairs of simple algebraic group and root of unity.

\subsubsection*{Relation to Topological Field Theories}

By way of the cobordism hypothesis \cite{BD,L}, our last theorem shows that, associated to any finite braided tensor category $\mathcal{A}$ with separable symmetric center, there is a fully extended framed 4-dimensional topological field theory. As in \cite{BJS,BJSS}, we view these topological field theories as non-separable and framed analogues of the Crane-Yetter-Kauffman theories \cite{CKY}, which are associated to any separable braided tensor category. In fact, the cobordism hypothesis also affords a canonical action of the group $\mathrm{SO}(4)$ on the space of fully dualizable objects in $\mathrm{Mor_2^{pre}}(\mathbf{Pr})$. Then, in order for the finite braided tensor category $\mathcal{A}$ to bestow a fully extended oriented 4-dimensional topological field theory, it is enough to provide $\mathcal{A}$ with the data of fixed point for this action by $\mathrm{SO}(4)$. Following \cite{BJS}, we speculate that any ribbon structure on $\mathcal{A}$ yields such a fixed point structure. This point is also discussed in more detail in \cite{H}. We therefore expect that any finite ribbon tensor category $\mathcal{A}$ with separable symmetric center yields a fully extended oriented 4-dimensional topological field theory. When $\mathcal{A}$ is separable, we believe that this recovers the (non-extended) Crane-Yetter-Kauffman theory associated to $\mathcal{A}$.

In a different direction, skein theory was used in \cite{CGHPM} to associate to any finite ribbon tensor category $\mathcal{A}$ a partially defined (non-extended) oriented 4-dimensional topological field theory. This construction is expected to recover the invariant of 4-dimensional 2-handlebodies obtained in \cite{BDR} using the universal Hopf algebra of \cite{Lyu:invariant,KL}. Now, in \cite{CGHPM}, a skein theoretic condition on $\mathcal{A}$ called chromatic compactness was shown to guarantee that the corresponding (non-extended) oriented 4-dimensional topological field theory is completely defined. The discussion of the previous paragraph suggests that if $\mathcal{A}$ has separable symmetric center, then it must be chromatic compact. This observation is relevant to the question of when the invariant of oriented 4-manifolds constructed in \cite{CGHPM} from the finite ribbon tensor category $\mathcal{A}$ may be able to detect exotic smooth structures. Namely, when $\mathcal{A}$ has separable symmetric center, we surmise that the corresponding topological field theory is not only completely defined but is also fully extended and therefore cannot detect exotic smooth structure by \cite{Reu}. Heuristically, this suggests that finite ribbon tensor categories with exotic symmetric centers should be used as the input for the above invariants of oriented 4-manifolds.

\subsection*{Acknowledgments}

I would like to thank both Patrick Kinnear and Maksymilian Manko for providing feedback on a preliminary version of this manuscript, as well as Noah Snyder for answering various questions on the categorical setup of \cite{BJS}. This work was supported in part by the Simons Collaboration on Global Categorical Symmetries.

\section{Preliminaries}\label{sec:prelim}

\subsection{Locally Finitely Presentable Categories}

We recall various facts and definitions from \cite[Section 2.1]{BJS} to which we refer the reader for additional details.

Let us fix a field $\mathbbm{k}$. We say that a $\mathbbm{k}$-linear category is locally finitely presentable if it has all small colimits and is generated under filtered colimits by a small subcategory of compact objects. A $\mathbbm{k}$-linear category has enough compact projectives if it has all small colimits and is generated under small colimits by a small subcategory of compact-projective objects. If the subcategory of compact-projective objects can be chosen to consists of a single object, we say that the $\mathbbm{k}$-linear category has a projective generator. A locally finitely presentable $\mathbbm{k}$-linear category is called artinian, if it is abelian, $\mathrm{Hom}$-spaces between compact objects are finite dimensional, and compact objects have finite length. Finally, a locally finitely presentable $\mathbbm{k}$-linear category is finite if it is artinian and has a projective generator.

We write $\mathbf{Pr}$ for the 2-category of locally finitely presentable $\mathbbm{k}$-linear categories and cocontinous $\mathbbm{k}$-linear functors. In particular, it follows from the special adjoint functor theorem that every cocontinuous functor between locally finitely presentable $\mathbbm{k}$-linear categories has a (not necessarily cocontinuous) right adjoint. The 2-category $\mathbf{Pr}$ admits a symmetric monoidal structure given by the Deligne-Kelly tensor product $\boxtimes$. Let us also recall that the Deligne-Kelly tensor product preserves locally finitely presentable $\mathbbm{k}$-linear categories that have enough compact projectives, resp.\ are artinian, resp.\ are finite  \cite{Del:Tannakian, Lor}.

It follows from the definitions that there is a strong relationship between a locally finitely presentable $\mathbbm{k}$-linear category and its subcategory of compact objects. More precisely, it was shown in \cite[Sections 3.1]{BZBJ} that there is an equivalence \begin{equation}\label{eq:equivalencePrcRex}(-)^{\mathbf{c}}:\mathbf{Pr^c}\leftrightarrow \mathbf{Rex}:\mathrm{ind}\end{equation} between the 2-category $\mathbf{Pr^c}$ of locally finitely presentable $\mathbbm{k}$-linear categories and compact-preserving cocontinuous $\mathbbm{k}$-linear functors and the 2-category $\mathbf{Rex}$ of small finitely cocomplete $\mathbbm{k}$-linear categories and right exact $\mathbbm{k}$-linear functors. Explicitly, the equivalence associates to a locally finitely presentable category $\mathcal{C}$ its full subcategory $\mathcal{C}^{\mathbf{c}}$ on the compact objects, and to a small finitely cocomplete category $\mathcal{D}$ its ind-completion $\mathrm{ind}(\mathcal{D})$. Furthermore, this equivalence is compatible with the Deligne-Kelly tensor product \cite[Sections 3.2]{BZBJ}. We can, and will therefore almost always, conflate an artinian, respectively finite, $\mathbbm{k}$-linear category with its category of compact objects, which is artinian, respectively finite, in the sense of \cite[Section 1.8]{EGNO}.

\subsection{(Pre-)Tensor Categories}\label{sub:pretensor}

We now review various definitions and results from \cite[Section 2.2]{BJS} to which we refer the reader for additional details. We emphasize that our terminology differs from theirs.

We will be interested in $E_1$-algebras in $\mathbf{Pr}$. Explicitly, an $E_1$-algebra in $\mathbf{Pr}$ is a locally finitely presentable category $\mathcal{C}$ equipped with a monoidal structure whose monoidal product $\otimes:\mathcal{C}\times\mathcal{C}\rightarrow \mathcal{C}$ is $\mathbbm{k}$-bilinear and cocontinuous in each variable. These conditions ensure that $\otimes:\mathcal{C}\times\mathcal{C}\rightarrow \mathcal{C}$ factors as $T_{\mathcal{C}}:\mathcal{C}\boxtimes\mathcal{C}\rightarrow \mathcal{C}$ through the Deligne-Kelly tensor product. If $\mathcal{C}$ is an $E_1$-algebra as above, then we write $\mathcal{C}^{\mathrm{mop}}$ for the $E_1$-algebra whose underlying category is $\mathcal{C}$ and equipped with the opposite monoidal structure $\otimes^{\mathrm{mop}}$, i.e.\ $C\otimes^{\mathrm{mop}} D := D\otimes C$ for every objects $C$ and $D$ of $\mathcal{C}$.

A pre-tensor category is an $E_1$-algebra in $\mathbf{Pr}$ whose underlying category has enough compact-projective objects and such that all of its compact-projective objects have right and left duals. Such monoidal categories are called ``cp-rigid tensor categories'' in \cite{BJS}. We have preferred to reserve the term tensor category to a more restrictive notion, which will be introduced shortly. Thanks to \cite[Definition-Proposition 4.1]{BJS}, an $E_1$-algebra $\mathcal{C}$ in $\mathbf{Pr}$ whose underlying category enough compact projectives is a pre-tensor category if and only if $T_{\mathcal{C}}$ has a cocontinuous right adjoint $T^R_{\mathcal{C}}$, and the canonical lax $\mathcal{C}$-$\mathcal{C}$-bimodule structure on $T^R_{\mathcal{C}}$ is strong. Given $\mathcal{C}$ and $\mathcal{D}$ two pre-tensor categories, a tensor functor $F:\mathcal{C}\rightarrow\mathcal{D}$ is a map of $E_1$-algebras in $\mathbf{Pr}$, i.e.\ it is a monoidal cocontinuous $\mathbbm{k}$-linear functor.

We will focus our attention on the more restrictive class of tensor categories.

\begin{Definition}
A tensor category is an artinian category with enough compact-projective objects that is equipped with a $\mathbbm{k}$-bilinear monoidal structure for which every compact object has a right and a left dual.
\end{Definition}

\noindent Every tensor category is in particular an artinian pre-tensor category. Our tensor categories are also ``compact-rigid tensor categories'' in the sense of \cite{BJS}. It follows from our definitions that, under the equivalence of equation \eqref{eq:equivalencePrcRex}, tensor categories correspond to the ``multitensor categories'' of \cite{EGNO} that have enough projectives. In particular, our finite tensor categories are exactly the ind-completion of the finite multitensor categories of \cite{EGNO}. This is the main motivation behind our choice of nomenclature. For use below, we also record the following classical observation (see for instance \cite[Section 1.3]{BN}).

\begin{Lemma}\label{lem:adjointtensor}
Any tensor functor between finite tensor categories is left-exact and compact preserving. Consequently, it has cocontinuous and compact preserving left and right adjoints.
\end{Lemma}

On one hand, pre-tensor categories are closed under the Deligne-Kelly tensor product by \cite[Proposition 4.6]{BJS}. More precisely, the proof does not use any assumption on the field $\mathbbm{k}$. On the other hand, tensor categories are not closed under the Deligne-Kelly tensor product (see the discussion before \cite[Theorem 2.2.18]{DSPS:dualizable}). In order for this to hold, it is enough to assume that $\mathbbm{k}$ is a perfect field by \cite[Proposition 5.17]{Del:Tannakian}. From now on, we shall therefore always assume that $\mathbbm{k}$ is perfect.

\subsection{Braided Tensor Categories and Centers}

Many of the (pre-)tensor categories that we will consider below are in fact $E_2$-algebras in $\mathbf{Pr}$. Explicitly, an $E_2$-algebra in $\mathbf{Pr}$ is an $E_1$-algebra $\mathcal{A}$ in $\mathbf{Pr}$ equipped a braiding $\beta$, that is, coherent natural isomorphisms $\beta_{A,B}:A\otimes B\cong B\otimes A$ for every pair of objects $A$ and $B$ in $\mathcal{A}$. Given an $E_2$-algebra $\mathcal{A}$, we write $\mathcal{A}^{\mathrm{rev}}$ for the $E_2$-algebra whose underlying $E_1$-algebra is $\mathcal{A}$ and whose braiding is given by $\beta^{\mathrm{rev}}_{A,B} := \beta_{B,A}^{-1}:A\otimes B\cong B\otimes A$.

Let $\mathcal{A}$ be an $E_2$-algebra in $\mathbf{Pr}$ with braiding $\beta$. The most fundamental invariant of $\mathcal{A}$ is its symmetric, also called M\"uger, center $\mathcal{Z}_{(2)}(\mathcal{A})$, which is the full subcategory of $\mathcal{A}$ on those objects $A$ such that $\beta_{B,A}\circ \beta_{A,B} = \mathrm{id}_{A\otimes B}$ for every object $B$ of $\mathcal{A}$. It follows from the definition that $\mathcal{Z}_{(2)}(\mathcal{A})$ is a symmetric monoidal category, so that it yields an $E_{\infty}$-algebra in $\mathbf{Pr}$. If $\mathcal{A}$ is a braided pre-tensor category, then $\mathcal{Z}_{(2)}(\mathcal{A})$ is a symmetric pre-tensor category by \cite[Theorem 5.13]{BJS}. If $\mathcal{A}$ is a (finite) braided tensor category, it follows easily from the definitions that $\mathcal{Z}_{(2)}(\mathcal{A})$ is a (finite) symmetric tensor category.

To any $E_1$-algebra $\mathcal{C}$ in $\mathbf{Pr}$, one can associate an $E_2$-algebra $\mathcal{Z}(\mathcal{C})$, called the Drinfeld center. The objects of $\mathcal{Z}(\mathcal{C})$ are pairs consisting of an object $Z$ of $\mathcal{C}$ together with a half-braiding $\gamma_{Z}$, that is suitably coherent natural isomorphisms $\gamma_{Z,C}:Z\otimes C\cong C\otimes Z$ for every object $C$ of $\mathcal{C}$. The morphisms of $\mathcal{Z}(\mathcal{C})$ are the morphisms of $\mathcal{C}$ that commute with the half-braidings. The Drinfeld center of a pre-tensor category is a pre-tensor category by \cite[Theorem 5.13]{BJS}. The Drinfeld center of a finite tensor category is a finite tensor category by \cite[Proposition 7.13.8]{EGNO}.

\subsection{The Relative Framework}

Let us now fix $\mathcal{E}$ an $E_{\infty}$-algebra in $\mathbf{Pr}$. In particular, $\mathcal{E}$ is a symmetric monoidal category. We write $\mathbf{Pr}_{\mathcal{E}}$ for the 2-category of $\mathcal{E}$-modules in $\mathbf{Pr}$. Explicitly, an $\mathcal{E}$-module in $\mathbf{Pr}$ is a locally finitely presentable $\mathbbm{k}$-linear category $\mathcal{M}$ equipped with a coherent (left) $\mathcal{E}$-module structure (see, for instance, \cite[Definition 7.1.1]{EGNO} for the coherence axioms), whose action functor $\otimes:\mathcal{E}\times\mathcal{M}\rightarrow \mathcal{M}$ is $\mathbbm{k}$-bilinear and cocontinuous in both variables. Morally, we think of $\mathbf{Pr}_{\mathcal{E}}$ as the 2-category of $\mathcal{E}$-enriched locally finitely presentable categories.

As $\mathcal{E}$ is symmetric, we will make little distinction between left and right $\mathcal{E}$-module structures. In particular, we can consider the relative Deligne-Kelly tensor product $\boxtimes_{\mathcal{E}}$ of any two $\mathcal{E}$-modules in $\mathbf{Pr}$. Its existence at the present level of generality follows from the discussion in \cite[Definition 3.13 \& Remark 3.14]{BZBJ}. Moreover, the relative Deligne-Kelly tensor product $\boxtimes_{\mathcal{E}}$ endows $\mathbf{Pr}_{\mathcal{E}}$ with a symmetric monoidal structure. When $\mathcal{E}$ is a finite symmetric tensor category, the relative Deligne tensor product of finite $\mathcal{E}$-module categories was studied in detail in \cite{DSPS:dualizable}.

An $E_1$-algebra in $\mathbf{Pr}_{\mathcal{E}}$ is an $E_1$-algebra $\mathcal{C}$ in $\mathbf{Pr}$ equipped with an $\mathcal{E}$-central structure, that is, a braided monoidal cocontinuous $\mathbbm{k}$-linear functor $\mathcal{E}\rightarrow \mathcal{Z}(\mathcal{C})$. This follows by inspection (see also \cite[Definition-Proposition 3.2]{BJS}). Provided that $\mathcal{E}$ is a (pre-)tensor category, an $\mathcal{E}$-enriched (pre-)tensor category is an $E_1$-algebra in $\mathbf{Pr}_{\mathcal{E}}$ whose underlying $E_1$-algebra in $\mathbf{Pr}$ is a (pre-)tensor category.\footnote{\label{ft:enriched}Our choice of terminology is motivated by our desire to avoid having to chose between calling such structures either algebras over $\mathcal{E}$, following the representation theory oriented literature, or algebras under $\mathcal{E}$, following the category theory oriented literature. The soundness of our nomenclature is justified further by \cite[Theorem 1.2]{MPP}.}

An $E_2$-algebra in $\mathbf{Pr}_{\mathcal{E}}$ is an $E_2$-algebra $\mathcal{A}$ in $\mathbf{Pr}$ equipped with a symmetric monoidal cocontinuous $\mathbbm{k}$-linear functor $\mathcal{E}\rightarrow \mathcal{Z}_{(2)}(\mathcal{A})$. This follows by inspection (see also \cite[Proposition 2.40]{Kin}). Let us now assume that $\mathcal{E}$ is a symmetric pre-tensor category.

\begin{Definition}
An $\mathcal{E}$-enriched braided (pre-)tensor category is a braided (pre-)tensor category $\mathcal{A}$ equipped with a symmetric tensor functor $F:\mathcal{E}\rightarrow \mathcal{Z}_{(2)}(\mathcal{A})$.
\end{Definition}

\noindent By definition, $\mathcal{E}$-enriched braided (pre-)tensor categories are $E_2$-algebras in $\mathbf{Pr}_{\mathcal{E}}$ whose underlying $E_1$-algebra in $\mathbf{Pr}$ is a (pre-)tensor category.\cref{ft:enriched}

\subsection{Higher Morita Categories}

The symmetric monoidal 2-category $\mathbf{Pr}$ has geometric realizations and the Deligne-Kelly tensor product commutes with them by \cite[Proposition 3.5]{BZBJ}. These properties are therefore also inherited by $\mathbf{Pr}_{\mathcal{E}}$, for any fixed $E_{\infty}$-algebra $\mathcal{E}$ in $\mathbf{Pr}$. We may therefore appeal to the construction of \cite{JFS, Hau}, so as to consider the Morita 4-category $\mathrm{Mor_2}(\mathbf{Pr}_{\mathcal{E}})$ of $E_2$-algebras in $\mathbf{Pr}_{\mathcal{E}}$.

Morphisms in the 4-category $\mathrm{Mor_2}(\mathbf{Pr})$ are described in detail in \cite[Section 3]{BJS}. Morphisms in the 4-category $\mathrm{Mor_2}(\mathbf{Pr}_{\mathcal{E}})$ admit analogous descriptions. Compendiously, objects of $\mathrm{Mor_2}(\mathbf{Pr}_{\mathcal{E}})$ are $\mathcal{E}$-enriched $E_2$-algebras $\mathcal{A}$ in $\mathbf{Pr}$, that is, $E_2$-algebras equipped with a symmetric monoidal cocontinuous $\mathbbm{k}$-linear functor $\mathcal{E}\rightarrow \mathcal{Z}_{(2)}(\mathcal{A})$. Given two $\mathcal{E}$-enriched $E_2$-algebras $\mathcal{A}$ and $\mathcal{B}$ in $\mathbf{Pr}$, a 1-morphism $\mathcal{C}:\mathcal{A}\nrightarrow \mathcal{B}$ is an $\mathcal{E}$-enriched $\mathcal{A}$-$\mathcal{B}$-central $E_1$-algebra, i.e.\ the $E_1$-algebra $\mathcal{C}$ comes equipped with a braided monoidal cocontinuous $\mathbbm{k}$-linear functor $F_{\mathcal{C}}:\mathcal{A}\boxtimes_{\mathcal{E}}\mathcal{B}^{\mathrm{rev}}\rightarrow\mathcal{Z}(\mathcal{C})$. Observe that $\mathcal{C}$ is in particular an $\mathcal{E}$-central $E_1$-algebra via the canonical braided monoidal functor $\mathcal{E}\rightarrow \mathcal{A}\boxtimes_{\mathcal{E}}\mathcal{B}^{\mathrm{rev}}\rightarrow\mathcal{Z}(\mathcal{C})$. Given $\mathcal{C}$ and $\mathcal{D}$ two $\mathcal{E}$-enriched $\mathcal{A}$-$\mathcal{B}$-central $E_1$-algebra, a 2-morphism $\mathcal{M}:\mathcal{C}\nRightarrow\mathcal{D}$ is an $\mathcal{E}$-enriched $\mathcal{A}$-$\mathcal{B}$-centered $\mathcal{C}$-$\mathcal{D}$-bimodule $\mathcal{M}$ in $\mathbf{Pr}$, that is, the $\mathcal{C}$-$\mathcal{D}$-bimodule $\mathcal{M}$ comes equipped with natural isomorphisms $F_{\mathcal{C}}(X)\otimes M \cong M\otimes F_{\mathcal{D}}(X)$ for every $X$ in $\mathcal{A}\boxtimes_{\mathcal{E}}\mathcal{B}^{\mathrm{rev}}$ and $M$ in $\mathcal{M}$ satisfying the obvious compatibility conditions. Then, 3-morphisms are bimodule functors compatible with the centered structures, and 4-morphisms are bimodule natural transformations.

Recall from \cite[Definition-Proposition 4.7]{BJS} that there is a symmetric monoidal sub-4-category $\mathrm{Mor_2^{pre}}(\mathbf{Pr})$ of $\mathrm{Mor_2}(\mathbf{Pr})$ whose objects are braided pre-tensor categories, 1-morphisms are central pre-tensor categories, 2-morphisms are centered bimodules with enough compact projective objects, and both 3- and 4-morphisms are as in $\mathrm{Mor_2}(\mathbf{Pr})$. There can be no such 4-category whose 1-morphisms are all (finite) tensor categories as the next example shows.

\begin{Example}\label{ex:pretensorcounterexample}
Let $\mathbbm{k}$ be a field of characteristic $p>0$. Write $\mathrm{Vec}(\mathbb{Z}/p)$ for the finite tensor category of $\mathbb{Z}/p$-graded vector spaces over $\mathbbm{k}$. Then, there is an equivalence of pre-tensor categories $$\mathrm{Vec}\boxtimes_{\mathrm{Vec}(\mathbb{Z}/p)}\mathrm{Vec}\simeq \mathrm{Mod}(\mathbbm{k}[\mathbb{Z}/p]),$$ where the monoidal structure on the right hand-side arises from the commutative algebra structure on $\mathbbm{k}[\mathbb{Z}/p]$. The compact projective generator $\mathbbm{k}[\mathbb{Z}/p]$ of $\mathrm{Mod}(\mathbbm{k}[\mathbb{Z}/p])$ has duals by virtue of the fact that it is the monoidal unit. However, is easy to check that $\mathbbm{k}$, the trivial $\mathbbm{k}[\mathbb{Z}/p]$-module, does not.
\end{Example}

\subsection{Dualizability in Higher Morita Categories}

Recall from \cite{L} that a symmetric monoidal $k$-category is fully dualizable if every object has a dual and every $i$-morphism for $1\leq i < k$ has a right and a left adjoint. An object of a symmetric monoidal $n$-category is $k$-dualizable if it lies in a symmetric monoidal sub-$k$-category that is fully dualizable.

For instance, it was shown in \cite{BCJF} that every locally finitely presentable category that has enough compact projectives is 1-dualizable in $\mathbf{Pr}$. This explains the appearance of this condition. Moreover, it is known that a locally finite presnetable category is fully dualizable in $\mathbf{Pr}$ if it is the $\mathrm{ind}$-completion of a finite semisimple category \cite{Til,BDSPV}.

The dualizability of the symmetric monoidal Morita 4-category $\mathrm{Mor_2}(\mathbf{Pr})$ was extensively studied in \cite{BJS}. In particular, they show that the symmetric monoidal sub-4-category $\mathrm{Mor_2^{pre}}(\mathbf{Pr})$ is 3-dualizable. We record an enriched version of this statement. Let $\mathcal{E}$ be a symmetric pre-tensor category. We can consider the symmetric monoidal sub-4-category $\mathrm{Mor_2^{pre}}(\mathbf{Pr}_{\mathcal{E}})$ of $\mathrm{Mor_2}(\mathbf{Pr}_{\mathcal{E}})$ whose objects are $\mathcal{E}$-enriched braided pre-tensor categories, 1-morphisms are $\mathcal{E}$-enriched central pre-tensor categories, 2-morphisms are $\mathcal{E}$-enriched centered bimodule categories with enough compact projectives, and both 3- and 4-morphisms are as in $\mathrm{Mor_2}(\mathbf{Pr}_{\mathcal{E}})$. That $\mathrm{Mor_2^{pre}}(\mathbf{Pr}_{\mathcal{E}})$ is indeed a symmetric monoidal sub-4-category follows from \cite[Definition-Proposition 4.7]{BJS}.

\begin{Proposition}\label{prop:3dualizability}
Every object of $\mathrm{Mor_2^{pre}}(\mathbf{Pr}_{\mathcal{E}})$ has left and right duals. Every 1-morphism of $\mathrm{Mor_2^{pre}}(\mathbf{Pr}_{\mathcal{E}})$ has left and right adjoints. Every 2-morphism of $\mathrm{Mor_2^{pre}}(\mathbf{Pr}_{\mathcal{E}})$ has left and right adjoints. In particular, every object of $\mathrm{Mor_2^{pre}}(\mathbf{Pr}_{\mathcal{E}})$ is 3-dualizable.
\end{Proposition}
\begin{proof}
This follows from the proof of \cite[Theorem 5.16]{BJS}. More precisely, the assertions about objects and 1-morphisms are consequences of general properties of Morita categories of $E_2$-algebras \cite{GS}. The statement about 2-morphisms unpacks as follows:\ One has to show that every $\mathcal{E}$-enriched $\mathcal{A}$-$\mathcal{B}$-centered $\mathcal{C}$-$\mathcal{D}$ bimodule $\mathcal{M}$, whose underlying category has enough compact projectives, has right and left adjoints. It is enough to show that every $\mathcal{A}\boxtimes_{\mathcal{E}}\mathcal{B}^{\mathrm{rev}}$-centered $\mathcal{C}$-$\mathcal{D}$ bimodule $\mathcal{M}$, whose underlying category has enough compact projectives, has right and left adjoints. This follows from \cite[Theorem 5.16]{BJS}.
\end{proof}

We will also be specifically interested in a fully dualizable symmetric monoidal sub-4-category of $\mathrm{Mor_2^{pre}}(\mathbf{Pr}_{\mathcal{E}})$. It is based on the following definition.

\begin{Definition}
A finite tensor category $\mathcal{C}$ is separable if it is semisimple and its Drinfeld center $\mathcal{Z}(\mathcal{C})$ is semisimple.
\end{Definition}

\noindent This notion played a key role in the study of the dualizability of the Morita 3-category of finite tensor categories undergone in \cite{DSPS:dualizable}. More precisely, they showed that a finite tensor category is fully dualizbale if and only if it is separable. It is therefore not surprising that this condition appears in the analysis of the dualizability of the Morita 4-category of braided pre-tensor categories. Let us also recall that separability admits the following equivalent characterization:\ Over an algebraically closed field, a finite semisimple tensor category is separable if and only if its quantum, or categorical, dimension is non-zero in $\mathbbm{k}$. In particular, it follows from the main theorem of \cite{ENO1} that, over a field of characteristic zero, every finite semisimple tensor category is separable.

Let $\mathcal{E}$ be a separable symmetric tensor category. We write $\mathrm{Mor}_2^{\mathrm{sep}}(\mathbf{Pr}_{\mathcal{E}})$ for the symmetric monoidal sub-4-category of $\mathrm{Mor}_2^{\mathrm{pre}}(\mathbf{Pr}_{\mathcal{E}})$ whose objects are $\mathcal{E}$-enriched separable braided tensor categories, 1-morphisms are $\mathcal{E}$-enriched central separable tensor categories, 2-morphisms are $\mathcal{E}$-enriched centered finite semisimple bimodule categories, 3-morphisms are compact preserving centered bimodule functors, and 4-morphisms are bimodule natural transformations. That this indeed defines a symmetric monoidal sub-4-category follows from \cite[Definition-Proposition 4.8 \& Remark 4.9]{BJS} (see also \cite[Section 2.5]{DSPS:dualizable}).

\begin{Theorem}\label{thm:enrichedfullydualizbale}
Let $\mathcal{E}$ be a separable symmetric tensor category. We have that $\mathrm{Mor}_2^{\mathrm{sep}}(\mathbf{Pr}_{\mathcal{E}})$ is a fully dualizable symmetric monoidal sub-4-category of $\mathrm{Mor_2^{pre}}(\mathbf{Pr}_{\mathcal{E}})$.
\end{Theorem}
\begin{proof}
That $\mathrm{Mor}_2^{\mathrm{sep}}(\mathbf{Pr}_{\mathcal{E}})$ is 3-dualizable follows via the argument used in the proof of Proposition \ref{prop:3dualizability}. It only remains to check that every 3-morphisms has left and right adjoints. In order to do so, it suffices to show that every compact preserving centered bimodule functor between central finite semisimple tensor categories has adjoints as a centered bimodule functor. This is exactly the content of \cite[Theorem 5.2.1]{BJS} (see also \cite[Section 3.4]{DSPS:dualizable}).
\end{proof}

\section{The Relative Universal Hopf Algebra}

Throughout, we work over a fixed perfect field $\mathbbm{k}$. Recall that a finite braided tensor category $\mathcal{A}$ is non-degenerate if its symmetric center is trivial, i.e.\ $\mathcal{Z}_{(2)}(\mathcal{A})\simeq\mathrm{Vec}$. It was shown in \cite{BJSS} that every non-degenerate finite braided tensor category $\mathcal{A}$ is an invertible object of $\mathrm{Mor_2^{pre}}(\mathbf{Pr})$. This gives a higher categorical perspective on the characterizations of the non-degeneracy of $\mathcal{A}$ obtained in \cite{Shi:nondeg}. These results rely crucially on the universal Hopf algebra $\mathbb{F}_{\mathcal{A}}$ in $\mathcal{A}$ introduced in \cite{Maj, Lyu:ribbon}, which is uniquely characterized by the equivalence of tensor categories $$\mathrm{Comod}_{\mathcal{A}}(\mathbb{F}_{\mathcal{A}})\simeq \mathcal{A}\boxtimes\mathcal{A}^{\mathrm{mop}}.$$ A predominant role is played by the canonical pairing $\omega_{\mathcal{A}}$ on $\mathbb{F}_{\mathcal{A}}$ constructed in \cite{Lyu:modular}. More precisely, the key observation is that $\mathcal{A}$ is non-degenerate if and only if the pairing $\omega_{\mathcal{A}}$ is non-degenerate. In order to obtain an enriched, or relative, version of the results of \cite{Shi:nondeg, BJSS}, we will consider a relative universal Hopf algebra.

Fix a finite symmetric tensor category $\mathcal{E}$ as well as an $\mathcal{E}$-enriched finite braided tensor category $\mathcal{A}$ whose enrichment is provided by the symmetric tensor functor $F:\mathcal{E}\rightarrow\mathcal{Z}_{(2)}(\mathcal{A})$. We begin by giving a sufficient condition on $F$ for $\mathcal{A}\boxtimes_{\mathcal{E}}\mathcal{A}^{\mathrm{mop}}$ to be a finite tensor category. We then study the Hopf algebra $\mathbb{F}_{\mathcal{E}/\mathcal{A}}$ in $\mathcal{A}$ corresponding to $\mathcal{A}\boxtimes_{\mathcal{E}}\mathcal{A}^{\mathrm{mop}}$ under Tannaka reconstruction. In particular, we compare $\mathbb{F}_{\mathcal{E}/\mathcal{A}}$ with the universal Hopf algebra $\mathbb{F}_{\mathcal{A}}$ in $\mathcal{A}$, and show that the canonical pairing $\omega_{\mathcal{A}}$ on $\mathbb{F}_{\mathcal{A}}$ descends to a pairing $\omega_{\mathcal{E}/\mathcal{A}}$ on $\mathbb{F}_{\mathcal{E}/\mathcal{A}}$.

\subsection{Faithfully Flat Enriched Finite Tensor Categories}

Let $\mathcal{A}$ be an $\mathcal{E}$-enriched finite braided tensor category. It follows from \cite[Definition-Proposition 4.7]{BJS} that $\mathcal{A}\boxtimes_{\mathcal{E}}\mathcal{A}^{\mathrm{mop}}$ is a pre-tensor category. It also follows from the property of the relative Deligne tensor product that it is also finite (see for instance \cite[Theorem 3.3]{DSPS:balanced}). There is however not guarantee that $\mathcal{A}\boxtimes_{\mathcal{E}}\mathcal{A}^{\mathrm{mop}}$ is a finite tensor category (see Example \ref{ex:pretensorcounterexample}). As a consequence, there is in general no hope to identify $\mathcal{A}\boxtimes_{\mathcal{E}}\mathcal{A}^{\mathrm{mop}}$ as the tensor category of comodules over a Hopf algebra in $\mathcal{A}$. In order to remedy this issue, we make the following definition.

\begin{Definition}
An $\mathcal{E}$-enriched finite braided tensor category $\mathcal{A}$ is faithfully flat if $F:\mathcal{E}\rightarrow\mathcal{Z}_{(2)}(\mathcal{A})$ is fully faithful.
\end{Definition}

\noindent We will see shortly that faithful flatness of $\mathcal{A}$ is sufficient to guarantee that $\mathcal{A}\boxtimes_{\mathcal{E}}\mathcal{A}^{\mathrm{mop}}$ is a finite tensor category. In order to do so, we will use an alternative description of the finite pre-tensor category $\mathcal{A}\boxtimes_{\mathcal{E}}\mathcal{A}^{\mathrm{mop}}$.

The functor $T_{\mathcal{E}}:\mathcal{E}\boxtimes\mathcal{E}^{\mathrm{mop}}\rightarrow\mathcal{E}$ has a cocontinuous right adjoint $T_{\mathcal{E}}^R$ as an $\mathcal{E}\boxtimes\mathcal{E}^{\mathrm{mop}}$-module functor. In fact, as $\mathcal{E}$ is symmetric monoidal, $T_{\mathcal{E}}^R$ is lax symmetric monoidal, so that $T_{\mathcal{E}}^R(\mathbbm{1})$ is a commutative algebra in $\mathcal{E}\boxtimes\mathcal{E}^{\mathrm{mop}}$, whose underlying object is compact. It follows for instance from \cite[Corollary 6.7]{SY:MTC} that there is an equivalence $$\mathrm{Mod}_{\mathcal{E}\boxtimes \mathcal{E}^{\mathrm{mop}}}(T_{\mathcal{E}}^R(\mathbbm{1}))\simeq\mathcal{E}$$ of finite symmetric tensor categories. Moreover, under this equivalence the functor $T_{\mathcal{E}}$ is identified with the free $T_{\mathcal{E}}^R(\mathbbm{1})$-module functor.

\begin{Lemma}\label{lem:relativetensorasmodules}
There is an equivalence $$\mathrm{Mod}_{\mathcal{A}\boxtimes \mathcal{A}^{\mathrm{mop}}}(T_{\mathcal{E}}^R(\mathbbm{1}))\simeq  \mathcal{A}\boxtimes_{\mathcal{E}} \mathcal{A}^{\mathrm{mop}}$$ that is compatible both with the $\mathcal{A}\boxtimes\mathcal{A}^{\mathrm{mop}}$-module structures and the monoidal structures. Under this equivalence, the canonical tensor functor $\mathcal{A}\boxtimes\mathcal{A}^{\mathrm{mop}}\rightarrow\mathcal{A}\boxtimes_{\mathcal{E}}\mathcal{A}^{\mathrm{mop}}$ is identified with the free $T_{\mathcal{E}}^R(\mathbbm{1})$-module functor.
\end{Lemma}
\begin{proof}
Using the observation recalled above, we obtain equivalences of $\mathcal{A}\boxtimes\mathcal{A}^{\mathrm{mop}}$-module categories \begin{align*}\mathcal{A}\boxtimes_{\mathcal{E}}\mathcal{A}^{\mathrm{mop}} &\simeq (\mathcal{A}\boxtimes\mathcal{A}^{\mathrm{mop}})\boxtimes_{\mathcal{E}\boxtimes\mathcal{E}^{\mathrm{mop}}}\mathcal{E}\\ & \simeq (\mathcal{A}\boxtimes\mathcal{A}^{\mathrm{mop}})\boxtimes_{\mathcal{E}\boxtimes\mathcal{E}^{\mathrm{mop}}}\mathrm{Mod}_{\mathcal{E}\boxtimes\mathcal{E}^{\mathrm{mop}}}(T_{\mathcal{E}}^R(\mathbbm{1}))\\ &\simeq \mathrm{Mod}_{\mathcal{A}\boxtimes \mathcal{A}^{\mathrm{mop}}}(T_{\mathcal{E}}^R(\mathbbm{1})).\end{align*} The first and the third equivalences follow from the universal properties of the relative Deligne tensor product \cite{DSPS:balanced}. In particular, we view $T_{\mathcal{E}}^R(\mathbbm{1})$ as a commutative algebra in $\mathcal{A}\boxtimes \mathcal{A}^{\mathrm{mop}}$ via $F\boxtimes F^{\mathrm{mop}}$. Tracing through these equivalences, we find that the canonical functor $\mathcal{A}\boxtimes\mathcal{A}^{\mathrm{mop}}\rightarrow\mathcal{A}\boxtimes_{\mathcal{E}}\mathcal{A}^{\mathrm{mop}}$ can be identified with the free $T_{\mathcal{E}}^R(\mathbbm{1})$-module functor $\mathcal{A}\boxtimes_{\mathcal{E}}\mathcal{A}^{\mathrm{mop}}\rightarrow\mathrm{Mod}_{\mathcal{A}\boxtimes \mathcal{A}^{\mathrm{mop}}}(T_{\mathcal{E}}^R(\mathbbm{1}))$. It readily follows that the monoidal structure are identified as claimed, which concludes the proof.
\end{proof}

We now prove our main technical result about faithfully flat enriched finite tensor categories.

\begin{Proposition}\label{prop:relativefinite}
If $\mathcal{A}$ is a faithfully flat $\mathcal{E}$-enriched finite braided tensor category, then $\mathcal{A}\boxtimes_{\mathcal{E}} \mathcal{A}^{\mathrm{mop}}$ is a finite tensor category.
\end{Proposition}
\begin{proof}
It is clear that $\mathcal{A}\boxtimes_{\mathcal{E}} \mathcal{A}^{\mathrm{mop}}$ is a finite pre-tensor category. Now, as $\mathcal{E}$ is a finite symmetric tensor category, the underlying object of the commutative algebra $T_{\mathcal{E}}^R(\mathbbm{1})$ is compact. Moreover, the algebra $T_{\mathcal{E}}^R(\mathbbm{1})$ is exact in the sense of \cite[Definition 7.8.20]{EGNO} by inspection (see also \cite[Corollary 6.7]{SY:MTC}). It therefore follows from \cite[Theorem 7.1]{CSZ} that $T_{\mathcal{E}}^R(\mathbbm{1})$ is finite semisimple, i.e.\ it decomposes as a finite direct sum of algebras whose underlying objects are compact and that have no non-trivial two-sided ideals. But finite semisimple algebras are preserved by fully faithful tensor functors between finite tensor categories. Namely, such functors preserve simple objects by \cite[Proposition 6.3.1]{EGNO}, thence induce bijections between the lattices of subobjects. It follows that $T_{\mathcal{E}}^R(\mathbbm{1})$ is an exact commutative algebra in $\mathcal{A}\boxtimes \mathcal{A}^{\mathrm{mop}}$. Appealing to \cite[Proposition 5.19]{SY:MTC}, we find that $\mathcal{A}\boxtimes_{\mathcal{E}} \mathcal{A}^{\mathrm{mop}}$ is a finite tensor category as claimed.
\end{proof}

The tensor functor $T_{\mathcal{A}}$ induces a tensor functor $T_{\mathcal{E}/\mathcal{A}}:\mathcal{A}\boxtimes_{\mathcal{E}} \mathcal{A}^{\mathrm{mop}}\rightarrow\mathcal{A}$ . For use below, we recall from \cite[Proposition 3.13]{Kin} how to describe $T_{\mathcal{E}/\mathcal{A}}$ under the equivalence of Lemma \ref{lem:relativetensorasmodules}. In order to do so, we will use $\mathbb{F}_{\mathcal{E}}:=T_{\mathcal{E}}T^R_{\mathcal{E}}(\mathbbm{1})$, a commutative algebra in $\mathcal{E}$.

\begin{Lemma}[\cite{Kin}]\label{lem:technicalrelativetensorfunctortensorproduct}
The functor $T_{\mathcal{E}/\mathcal{A}}$ corresponds to the functor
$$\begin{tabular}{c c c}
$\mathrm{Mod}_{\mathcal{A}\boxtimes \mathcal{A}^{\mathrm{mop}}}(T_{\mathcal{E}}^R(\mathbbm{1}))$ & $\rightarrow$ & $\mathcal{A}$.\\
$M$ & $\mapsto$ & $T_{\mathcal{A}}(M)\otimes_{\mathbb{F}_{\mathcal{E}}}\mathbbm{1}$
\end{tabular}$$
\end{Lemma}
\begin{proof}
Observe that $T_{\mathcal{A}}:\mathcal{A}\boxtimes\mathcal{A}^{\mathrm{mop}}\rightarrow\mathcal{A}\boxtimes_{\mathcal{E}}\mathcal{A}^{\mathrm{mop}}\rightarrow\mathcal{A}$ is identified with $$\mathcal{A}\boxtimes\mathcal{A}^{\mathrm{mop}}\rightarrow\mathrm{Mod}_{\mathcal{A}\boxtimes\mathcal{A}^{\mathrm{mop}}}(T^R_{\mathcal{E}}(\mathbbm{1}))\rightarrow\mathcal{A}.$$ It therefore follows that this composite acts by $$X\mapsto X\otimes T^R_{\mathcal{E}}(\mathbbm{1})\mapsto T_{\mathcal{A}}(X) = T_{\mathcal{A}}(X\otimes T^R_{\mathcal{E}}(\mathbbm{1}))\otimes_{\mathbb{F}_{\mathcal{E}}}\mathbbm{1}.$$ But the functor $\mathrm{Mod}_{\mathcal{A}\boxtimes\mathcal{A}^{\mathrm{mop}}}(T^R_{\mathcal{E}}(\mathbbm{1}))\rightarrow\mathcal{A}$ is both cocontinuous and compatible with the left $\mathcal{A}\boxtimes\mathcal{A}^{\mathrm{mop}}$-module structures. It therefore follows that it must correspond to $M\mapsto T_{\mathcal{A}}(M)\otimes_{\mathbb{F}_{\mathcal{E}}}\mathbbm{1}$ as claimed. Additionally, it is straightforward to check that this correspondence preserves the monoidal structures.
\end{proof}

\subsection{Comonadic Reconstruction}\label{sub:comonadic}

We now express the category $\mathcal{A}\boxtimes_{\mathcal{E}}\mathcal{A}^{\mathrm{mop}}$ as a category of right comodules in $\mathcal{A}$. In order to do so, we will use the canonical tensor functor $$T_{\mathcal{E}/\mathcal{A}}:\mathcal{A}\boxtimes_{\mathcal{E}}\mathcal{A}^{\mathrm{mop}}\rightarrow\mathcal{A},$$ arising from the $\mathcal{E}$-enriched tensor structure on $\mathcal{A}$. Given our finiteness assumptions, this reconstruction could be carried out using \cite[Theorem 6.7]{LM}. We give a different, more direct proof, which avoids completely the theory of ``relative coends''. In the special case $\mathcal{E} = \mathrm{Vec}$, the next result is well-known \cite[Section 2.7]{Lyu:squared}.

\begin{Proposition}\label{prop:comonadicreconstruction}
Let $\mathcal{A}$ be a faithfully flat $\mathcal{E}$-enriched finite braided tensor category. Then, there is a coalgebra structure on $\mathbb{F}_{\mathcal{E}/\mathcal{A}}=T_{\mathcal{E}/\mathcal{A}} T_{\mathcal{E}/\mathcal{A}}^R(\mathbbm{1})$ for which we have an equivalence of left $\mathcal{A}$-module categories $$\mathrm{Comod}_{\mathcal{A}}(\mathbb{F}_{\mathcal{E}/\mathcal{A}})\simeq\mathcal{A}\boxtimes_{\mathcal{E}}\mathcal{A}^{\mathrm{mop}}.$$ Under this equivalence, the functor $T_{\mathcal{E}/\mathcal{A}}$ is identified with the forgetful functor.
\end{Proposition}
\begin{proof}
This is an instance of the comonadic version of the Barr-Beck theorem. Namely, the functor $T_{\mathcal{E}/\mathcal{A}}:\mathcal{A}\boxtimes_{\mathcal{E}}\mathcal{A}^{\mathrm{mop}}\rightarrow\mathcal{A}$ has a right adjoint. Moreover, $T_{\mathcal{E}/\mathcal{A}}$ is exact as $\mathcal{A}\boxtimes_{\mathcal{E}}\mathcal{A}^{\mathrm{mop}}$ is a finite tensor category by Proposition \ref{prop:relativefinite}.

In order to appeal to the Barr-Beck theorem, it therefore only remains to argue that $T_{\mathcal{E}/\mathcal{A}}$ is conservative. By \cite[Lemma 4.3]{BZBJ}, as $\mathcal{A}\boxtimes_{\mathcal{E}}\mathcal{A}^{\mathrm{mop}}$ is abelian, this is equivalent to showing that $T_{\mathcal{E}/\mathcal{A}}$ is faithful. But as $\mathcal{A}\boxtimes_{\mathcal{E}}\mathcal{A}^{\mathrm{mop}}$ is a tensor category, it is enough to check that $$T_{\mathcal{E}/\mathcal{A}}:\mathrm{End}_{\mathcal{A}\boxtimes_{\mathcal{E}}\mathcal{A}^{\mathrm{mop}}}(\mathbbm{1})\rightarrow\mathrm{End}_{\mathcal{A}}(\mathbbm{1})$$ is faithful by \cite[Theorem 4.3.8]{EGNO}. In order to see that this holds, first observe that $\mathcal{E}\hookrightarrow \mathcal{A}\boxtimes_{\mathcal{E}}\mathcal{A}^{\mathrm{mop}}$ is fully faithful. Namely, it is identified with the fully faithful functor $\mathrm{Mod}_{\mathcal{E}\boxtimes\mathcal{E}^{\mathrm{mop}}}(T^R_{\mathcal{E}}(\mathbbm{1}))\rightarrow \mathrm{Mod}_{\mathcal{A}\boxtimes\mathcal{A}^{\mathrm{mop}}}(T^R_{\mathcal{E}}(\mathbbm{1}))$ induced by the fully faithful functor $F\boxtimes F^{\mathrm{mop}}$. Then, note that the composite $$\mathcal{E}\hookrightarrow \mathcal{A}\boxtimes_{\mathcal{E}}\mathcal{A}^{\mathrm{mop}}\rightarrow\mathcal{A}$$ is canonically identified with $F$, which is fully faithful by hypothesis. This concludes the proof of the conservativity of $T_{\mathcal{E}/\mathcal{A}}$.

Altogether, we have that $\mathcal{A}\boxtimes_{\mathcal{E}}\mathcal{A}^{\mathrm{mop}}$ is equivalent to the category of comodules over the comonad $T_{\mathcal{E}/\mathcal{A}}\circ T^R_{\mathcal{E}/\mathcal{A}}$ on $\mathcal{A}$. But the adjunction between $T_{\mathcal{E}/\mathcal{A}}$ and $T^R_{\mathcal{E}/\mathcal{A}}$ is compatible with the left $\mathcal{A}$-module structures. Namely, the tensor functor $T_{\mathcal{E}/\mathcal{A}}$ preserves duals, and therefore also compact objects, so that the claim follows from \cite[Corollary 4.3]{BJS}. It follows that the comonad $T_{\mathcal{E}/\mathcal{A}}\circ T_{\mathcal{E}/\mathcal{A}}^R$ is compatible with the left $\mathcal{A}$-module structures, so that that there is an isomorphism $T_{\mathcal{E}/\mathcal{A}} \circ T_{\mathcal{E}/\mathcal{A}}^R(-)\simeq (-)\otimes T_{\mathcal{E}/\mathcal{A}} T_{\mathcal{E}/\mathcal{A}}^R(\mathbbm{1})$ of comonads. This concludes the proof.
\end{proof}

\begin{Remark}
It is not enough to assume that the functor $F$ be merely faithful in the statement of the previous proposition. Namely, taking $\mathcal{E}=\mathrm{Vec}_{\mathbb{R}}$ and $\mathcal{A}=\mathrm{Vec}_{\mathbb{C}}$, we find that $\mathcal{A}\boxtimes_{\mathcal{E}}\mathcal{A}^{\mathrm{mop}}\rightarrow \mathcal{A}$ is not faithful as it induces the algebra map $\mathbb{C}\otimes_{\mathbb{R}}\mathbb{C}\rightarrow\mathbb{C}$ on the monoidal unit.
\end{Remark}

\subsection{The Relative Universal Hopf Algebra}

We have shown in Proposition \ref{prop:comonadicreconstruction} above that if $\mathcal{A}$ is a faithfully flat $\mathcal{E}$-enriched finite braided tensor category, then $\mathbb{F}_{\mathcal{E}/\mathcal{A}}$ is naturally a coalgebra. Given that $\mathcal{A}\boxtimes_{\mathcal{E}}\mathcal{A}^{\mathrm{mop}}$ is a finite tensor category, it is natural to expect that $\mathbb{F}_{\mathcal{E}/\mathcal{A}}$ carries a Hopf algebra structure. When $\mathcal{E} = \mathrm{Vec}$, this is well-known \cite[Section 2.7]{Lyu:squared}. For a general finite symmetric tensor category $\mathcal{E}$, this follows from \cite[Theorems 6.11 \& 6.17]{LM} (see also \cite[Theorem 3.14]{BV:Hopf}, or more precisely its dual, for a related observation). We sketch a proof for completeness.

\begin{Proposition}\label{prop:relativeuniversalHopf}
Let $\mathcal{A}$ be a faithfully flat $\mathcal{E}$-enriched finite braided tensor category. Then, $\mathbb{F}_{\mathcal{E}/\mathcal{A}}$ is a Hopf algebra in $\mathcal{A}$. Furthermore, there is an equivalence of tensor categories $$\mathrm{Comod}_{\mathcal{A}}(\mathbb{F}_{\mathcal{E}/\mathcal{A}})\simeq\mathcal{A}\boxtimes_{\mathcal{E}}\mathcal{A}^{\mathrm{mop}}.$$
\end{Proposition}
\begin{proof}
We begin by the observation that there is an equivalence of categories 
$$\mathrm{Comod}_{\mathcal{A}}(\mathbb{F}_{\mathcal{E}/\mathcal{A}}\otimes \mathbb{F}_{\mathcal{E}/\mathcal{A}})\simeq \big(\mathcal{A}\boxtimes_{\mathcal{E}}\mathcal{A}^{\mathrm{mop}}\big)\boxtimes_{\mathcal{A}}\big(\mathcal{A}\boxtimes_{\mathcal{E}}\mathcal{A}^{\mathrm{mop}}\big) = \mathcal{A}\boxtimes_{\mathcal{E}}\mathcal{A}^{\mathrm{mop}}\boxtimes_{\mathcal{E}}\mathcal{A}^{\mathrm{mop}},$$
where the coalgebra structure on $\mathbb{F}_{\mathcal{E}/\mathcal{A}}\otimes \mathbb{F}_{\mathcal{E}/\mathcal{A}}$ is given by $(\mathrm{id}\otimes\beta\otimes\mathrm{id})\circ (\Delta\otimes\Delta)$. Namely, as $T_{\mathcal{A}}$ is not only cocontinuous, but also left-exact, there is a left-exact cocontinuous functor 
$$\begin{tabular}{c c c}
$\mathrm{Comod}_{\mathcal{A}}(\mathbb{F}_{\mathcal{E}/\mathcal{A}})\boxtimes\mathrm{Comod}_{\mathcal{A}}(\mathbb{F}_{\mathcal{E}/\mathcal{A}})$ & $\rightarrow$ & $\mathrm{Comod}_{\mathcal{A}}(\mathbb{F}_{\mathcal{E}/\mathcal{A}}\otimes \mathbb{F}_{\mathcal{E}/\mathcal{A}})$\\
$(X,\rho_X)\boxtimes (Y,\rho_Y)$ & $\mapsto$ & $(X\otimes Y,\rho_{X\otimes Y}),$
\end{tabular}$$
where $\rho_{X\otimes Y}:=(\mathrm{id}\otimes \beta\otimes \mathrm{id})\circ(\rho_X\otimes \rho_Y)$.
In fact, the above functor is easily seen to be compatible with the left $\mathcal{A}$-module structures, so that we obtain a left exact and cocontinuous functor $$\mathrm{Comod}_{\mathcal{A}}(\mathbb{F}_{\mathcal{E}/\mathcal{A}})\boxtimes_{\mathcal{A}}\mathrm{Comod}_{\mathcal{A}}(\mathbb{F}_{\mathcal{E}/\mathcal{A}})\rightarrow \mathrm{Comod}_{\mathcal{A}}(\mathbb{F}_{\mathcal{E}/\mathcal{A}}\otimes \mathbb{F}_{\mathcal{E}/\mathcal{A}}).$$ In order to check that it is an equivalence, one can therefore use the fact that every right $\mathbb{F}_{\mathcal{E}/\mathcal{A}}$-comodule $X$ can be expressed as an equalizer:
$$\begin{tikzcd}
X \arrow[r, "\rho"] & X\otimes\mathbb{F}_{\mathcal{E}/\mathcal{A}} \arrow[r, "\rho\otimes\mathrm{id}", shift left] \arrow[r, "\mathrm{id}\otimes\Delta"', shift right] & X\otimes\mathbb{F}_{\mathcal{E}/\mathcal{A}}\otimes\mathbb{F}_{\mathcal{E}/\mathcal{A}}.
\end{tikzcd}$$
The claim then follows readily by a standard argument.

As a consequence of the discussion of the preceding paragraph, we find that the canonical functor $$T:\big(\mathcal{A}\boxtimes_{\mathcal{E}}\mathcal{A}^{\mathrm{mop}}\big)\boxtimes_{\mathcal{A}}\big(\mathcal{A}\boxtimes_{\mathcal{E}}\mathcal{A}^{\mathrm{mop}}\big) = \mathcal{A}\boxtimes_{\mathcal{E}}\mathcal{A}^{\mathrm{mop}}\boxtimes_{\mathcal{E}}\mathcal{A}^{\mathrm{mop}}\rightarrow \mathcal{A}$$ is identified with the forgetful functor $\mathrm{Comod}_{\mathcal{A}}(\mathbb{F}_{\mathcal{E}/\mathcal{A}}\otimes \mathbb{F}_{\mathcal{E}/\mathcal{A}})\rightarrow \mathcal{A}$. In fact, this identifies the comonad $T\circ T^R(-)$ with $(-)\otimes (\mathbb{F}_{\mathcal{E}/\mathcal{A}}\otimes \mathbb{F}_{\mathcal{E}/\mathcal{A}})$. Let us now write $$\widetilde{T}:\big(\mathcal{A}\boxtimes_{\mathcal{E}}\mathcal{A}^{\mathrm{mop}}\big)\boxtimes_{\mathcal{A}}\big(\mathcal{A}\boxtimes_{\mathcal{E}}\mathcal{A}^{\mathrm{mop}}\big)\rightarrow\mathcal{A}\boxtimes_{\mathcal{E}}\mathcal{A}^{\mathrm{mop}}$$ for the functor induced by the tensor product on $\mathcal{A}\boxtimes_{\mathcal{E}}\mathcal{A}^{\mathrm{mop}}$. By construction, the following diagram is commutative:
$$\begin{tikzcd}[sep=small]
\mathcal{A}\boxtimes_{\mathcal{E}}\mathcal{A}^{\mathrm{mop}}\boxtimes_{\mathcal{E}}\mathcal{A}^{\mathrm{mop}} \arrow[rr, "\widetilde{T}"] \arrow[rd, "T"'] &             & \mathcal{A}\boxtimes_{\mathcal{E}}\mathcal{A}^{\mathrm{mop}} \arrow[ld, "T_{\mathcal{E}/\mathcal{A}}"] \\
& \mathcal{A}. &
\end{tikzcd}$$
The functor $\widetilde{T}$ is cocontinuous and preserves compact objects, so that it admits a right adjoint compatible with the left $\mathcal{A}$-module structures by \cite[Corollary 4.3]{BJS}. In particular, there is an induced morphism of comonads $T\circ T^R \rightarrow T_{\mathcal{E}/\mathcal{A}}\circ T_{\mathcal{E}/\mathcal{A}}^R$ compatible with the  left $\mathcal{A}$-module structures. Upon evaluating this morphism at the monoidal unit $\mathbbm{1}$, we get a map of coalgebras $m:\mathbb{F}_{\mathcal{E}/\mathcal{A}}\otimes \mathbb{F}_{\mathcal{E}/\mathcal{A}}\rightarrow \mathbb{F}_{\mathcal{E}/\mathcal{A}}$. That $m$ is associative follows from the fact that $T_{\mathcal{E}/\mathcal{A}}$ is associative using a similar argument.

Finally, we refer the reader to \cite[Section 6.5]{LM} for the construction of the antipode on $\mathbb{F}_{\mathcal{E}/\mathcal{A}}$ and its inverse.
\end{proof}

\begin{Remark}\label{rem:Shimizucomparison}
In the case $\mathcal{E} = \mathrm{Vec}$, the Hopf algebra $\mathbb{F}_{\mathcal{A}}$ in $\mathcal{A}$ coincides with the Hopf algebra $\mathbb{H}_{\mathcal{A}}$ introduced in \cite{Maj, Lyu:modular}, and employed in \cite{Shi:nondeg}. Namely, it was shown in \cite{Lyu:squared} that $\mathrm{Comod}_{\mathcal{A}}(\mathbb{H}_{\mathcal{A}})\simeq \mathcal{A}\boxtimes\mathcal{A}^{\mathrm{mop}}$ as tensor categories, so that it follows immediately from \cite[Lemma 6.3 and Proposition 6.14]{LM} that $\mathbb{F}_{\mathcal{A}}\cong \mathbb{H}_{\mathcal{A}}$ as Hopf algebras. It is well-known that the Hopf algebra structure on $\mathbb{F}_{\mathcal{A}}$ can be described very explicitly. This is achieved using the fact, proven for instance in \cite[Section 3.2]{BJSS}, that the right adjoint to $T_{\mathcal{A}}$ can expressed via the coend $$T_{\mathcal{A}}^R(\mathbbm{1}) \cong \int^{A\in \mathcal{A}^{\mathbf{c}}}A^*\boxtimes A \in \mathcal{A}\boxtimes\mathcal{A}^{\mathrm{mop}},$$ where $A^*$ denotes the right dual of the compact object $A$ in $\mathcal{A}^{\mathbf{c}}$. In particular, it follows that $$\mathbb{F}_{\mathcal{A}} \cong \int^{A\in \mathcal{A}^{\mathbf{c}}}A^*\otimes A$$ as coalgebras in $\mathcal{A}$. More generally, the Hopf algebra structure on $\mathbb{F}_{\mathcal{A}}$ can be described very explicitly using this description as a coend. We refer the reader to \cite[Section 2]{Lyu:modular} for the details. For use below, we will nonetheless recall how the canonical pairing $\omega_{\mathcal{A}}:\mathbb{F}_{\mathcal{A}}\otimes \mathbb{F}_{\mathcal{A}}\rightarrow\mathbbm{1}$ is defined. Given a compact object $A$ in $\mathcal{A}$, we write $\iota_{A}:A^*\otimes A\rightarrow\mathbb{F}_{\mathcal{A}}$ for the canonical map coming from the definition of $\mathbb{F}_{\mathcal{A}}$ as a coend. The pairing $\omega_{\mathcal{A}}$ from \cite{Lyu:modular,Shi:nondeg} is then uniquely specified through the universal property of the coend by the following equality
$$\omega_{\mathcal{A}} \circ (\iota_{A}\otimes\iota_{B}) = (\mathrm{ev}_{A}\otimes \mathrm{ev}_{B}) \circ (A^*\otimes\beta^2_{A,B^*}\otimes B)\circ (\iota_{A}\otimes\iota_{B}),$$ for every compact objects $A$ and $B$ in $\mathcal{A}$. Above, we have used $\mathrm{ev}_A:A^*\otimes A\rightarrow \mathbbm{1}$ to denote the evaluation morphism witnessing that $A^*$ is right dual to $A$, and similarly for $\mathrm{ev}_B$.
\end{Remark}

\subsection{An Exact Sequence of Hopf Algebras}

We now explain the nature of the relation between the Hopf algebras $\mathbb{F}_{\mathcal{E}/\mathcal{A}}$ and $\mathbb{F}_{\mathcal{A}}$ in $\mathcal{A}$. We use a generalization of the notion of a (strictly) exact sequence of Hopf algebras in $\mathrm{Vec}$ introduced in \cite{Sch} (see also \cite[Section 2.2]{BN}). Given a Hopf algebra $H$ in a braided tensor category $\mathcal{A}$, we write $H^+$ for its augmentation ideal, that is, the kernel of its counit $\epsilon:H\rightarrow \mathbbm{1}$.

\begin{Definition}
Let $\mathcal{A}$ be a braided tensor category. A sequence $$K\xrightarrow{i} H\xrightarrow{p} Q$$ of maps of Hopf algebras in $\mathcal{A}$ is called a (strictly) exact sequence if the following three conditions hold:
\begin{enumerate}[label=(\alph*)]
    \item $K$ is a normal Hopf subalgebra of $H$, that is, $K$ is stable under the adjoint action of $H$ on itself,
    \item $H$ is faithfully flat over $K$,
    \item $p$ is a categorical cokernel of $i$, that is, $p$ is isomorphic to the quotient map $H\rightarrow H/HK^+$.
\end{enumerate}
\end{Definition}

When working with finite dimensional Hopf algebras in $\mathrm{Vec}$, it follows from the Nichols-Zoeller theorem that $H$ is free as a $K$-module. In particular, the second condition above is automatic. This is a general phenomenon.

\begin{Lemma}\label{lem:faithfulflatness}
Let $i:K\hookrightarrow H$ be an inclusion of Hopf algebras in $\mathcal{A}$ such that the underlying objects of $K$ and $H$ are both compact. Then, $H$ is faithfully flat over $K$.
\end{Lemma}
\begin{proof}
We have that $i^*:\mathrm{Mod}_{\mathcal{A}}(H)\rightarrow \mathrm{Mod}_{\mathcal{A}}(K)$ is a tensor functor between finite tensor categories. In particular, its left adjoint $$(-)\otimes_KH:\mathrm{Mod}_{\mathcal{A}}(K)\rightarrow \mathrm{Mod}_{\mathcal{A}}(H)$$ is also left-exact. Namely, the right adjoint of $i^*$ can be expressed using $(-)\otimes_KH$ and duality functors of $\mathcal{A}$ (see for instance \cite[Section 1.3]{BN}). It follows that $H$ is flat over $K$. Moreover, as $i$ is a monomorphism, then $(-)\otimes_KH$ is faithful. To see this, note that the map $\mathrm{Hom}_K(M,N)\rightarrow \mathrm{Hom}_H(M\otimes_KH,N\otimes_KH)$ can be factored as $$\mathrm{Hom}_K(M,N)\hookrightarrow \mathrm{Hom}_K(M,N\otimes_KH)\cong \mathrm{Hom}_H(M\otimes_KH,N\otimes_KH).$$ We therefore find that $H$ is faithfully flat over $K$.
\end{proof}

\begin{Theorem}\label{thm:exactsequenceHopf}
Let $\mathcal{A}$ be a faithfully flat $\mathcal{E}$-enriched finite braided tensor category. There is an exact sequence of Hopf algebras in $\mathcal{A}$ $$\mathbb{F}_{\mathcal{E}}\hookrightarrow\mathbb{F}_{\mathcal{A}}\twoheadrightarrow\mathbb{F}_{\mathcal{E}/\mathcal{A}}.$$
\end{Theorem}
\begin{proof}
We begin by describing how the map of Hopf algebras $i:\mathbb{F}_{\mathcal{E}}\rightarrow\mathbb{F}_{\mathcal{A}}$ arises. We claim that there is an equivalence of tensor categories $$\mathrm{Comod}_{\mathcal{A}}(\mathbb{F}_{\mathcal{E}})\simeq\mathcal{A}\boxtimes\mathcal{E}^{\mathrm{mop}}.$$ Namely, using cofree copresentations, it is easy to show that the canonical tensor functor $$\mathcal{A}\boxtimes\mathcal{E}^{\mathrm{mop}}\simeq \mathcal{A}\boxtimes_{\mathcal{E}}\mathrm{Comod}_{\mathcal{E}}(\mathbb{F}_{\mathcal{E}})\rightarrow\mathrm{Comod}_{\mathcal{A}}(\mathbb{F}_{\mathcal{E}})$$ is an equivalence. Now, the inclusion of finite tensor categories $\mathrm{Id}\boxtimes F^{\mathrm{mop}}:\mathcal{A}\boxtimes\mathcal{E}^{\mathrm{mop}}\hookrightarrow  \mathcal{A}\boxtimes\mathcal{A}^{\mathrm{mop}}$ has a right adjoint. It therefore induces a map of comonads, and, a fortiori, a map of Hopf algebras $i:\mathbb{F}_{\mathcal{E}}\rightarrow\mathbb{F}_{\mathcal{A}}$. Moreover, $F$ is a fully faithful tensor functor between finite tensor categories, whence so is $\mathrm{Id}\boxtimes F^{\mathrm{mop}}$. Consequently, these two functors correspond to the inclusion of a topologizing subcategory by \cite[Proposition 6.3.1]{EGNO}. In particular, the right adjoint $G$ to $\mathrm{Id}\boxtimes F^{\mathrm{mop}}:\mathcal{A}\boxtimes\mathcal{E}^{\mathrm{mop}}\hookrightarrow \mathcal{A}\boxtimes\mathcal{A}^{\mathrm{mop}}$ is given by sending an object in $\mathcal{A}\boxtimes\mathcal{A}^{\mathrm{mop}}$ to its largest subobject in $\mathcal{A}\boxtimes\mathcal{E}^{\mathrm{mop}}$ (see for instance \cite[Section 4.2]{Shi:nondeg}). On the other hand, it follows by direct inspection that $G(\mathbb{F}_{\mathcal{A}}) = \mathbb{F}_{\mathcal{E}}$. Altogether, we find that $\mathbb{F}_{\mathcal{E}}$ is a subobject of $\mathbb{F}_{\mathcal{A}}$ as desired.

We show that the subalgebra $\mathbb{F}_{\mathcal{E}}\hookrightarrow \mathbb{F}_{\mathcal{A}}$ is a normal Hopf subalgebra, that is, we show that the composite 
\begin{equation}\label{eq:normality}
\mathbb{F}_{\mathcal{A}}\otimes\mathbb{F}_{\mathcal{E}}\xrightarrow{\Delta\otimes i}\mathbb{F}_{\mathcal{A}}\otimes\mathbb{F}_{\mathcal{A}}\otimes\mathbb{F}_{\mathcal{A}}\xrightarrow{\mathrm{id}\otimes\beta}\mathbb{F}_{\mathcal{A}}\otimes\mathbb{F}_{\mathcal{A}}\otimes\mathbb{F}_{\mathcal{A}}\xrightarrow{m\otimes S}\mathbb{F}_{\mathcal{A}}\otimes\mathbb{F}_{\mathcal{A}}\xrightarrow{m}\mathbb{F}_{\mathcal{A}}
\end{equation}
has image in $\mathbb{F}_{\mathcal{E}}$. In order to do so, we will use the descriptions of $\mathbb{F}_{\mathcal{A}}$ and $\mathbb{F}_{\mathcal{E}}$ as a coends recalled in Remark \ref{rem:Shimizucomparison}. In particular, it is enough to consider the image of the map in \eqref{eq:normality} precomposed with the canonical maps $$\iota_A\otimes\iota_E:(E^*\otimes E)\otimes (A^*\otimes A)\rightarrow \mathbb{F}_{\mathcal{E}}\otimes\mathbb{F}_{\mathcal{A}}$$ for every pair of compact objects $A$ in $\mathcal{A}$ and $E$ in $\mathcal{E}$. Using the explicit description of the Hopf algebra structure on $\mathbb{F}_{\mathcal{A}}$ given in \cite[Section 2]{Lyu:modular}, and recalling that $E$ lies in $\mathcal{Z}_{(2)}(\mathcal{A})$, we find that the overall composite is given by 
\begin{align*}(A^*\otimes A)\otimes (E^*\otimes E)&\xrightarrow{\mathrm{id}\otimes\beta\otimes\mathrm{id}}(A^*\otimes E^*\otimes A)\otimes E\\
&\xrightarrow{\mathrm{id}\otimes\mathrm{coev}_A\otimes\mathrm{id}}(A^*\otimes E^*\otimes A)\otimes (A^*\otimes A\otimes E)\\
&\xrightarrow{\mathrm{id}\otimes\beta}(A^*\otimes E^*\otimes A)\otimes (A^*\otimes E\otimes A)\\
&\xrightarrow{\iota_{A^*\otimes E\otimes A}}\mathbb{F}_{\mathcal{A}}.
\end{align*}
But, this last map factors through the coend $\int^{A\in\mathcal{A}^{\mathbf{c}}}A^*\otimes F^R(A)$ as $$E\xrightarrow{\mathrm{coev}_A\otimes\mathrm{id}}A^*\otimes A\otimes E\xrightarrow{\mathrm{id}\otimes\beta}A^*\otimes E\otimes A$$ factors through $F^R(A^*\otimes E\otimes A)$, which is the largest subobject of $A^*\otimes E\otimes A$ in $\mathcal{E}$ by the previous paragraph. Finally, the isomorphism of coends $$\int^{A\in\mathcal{A}^{\mathbf{c}}}A^*\otimes F^R(A)\cong \int^{E\in\mathcal{E}^{\mathbf{c}}}E^*\otimes E = \mathbb{F}_{\mathcal{E}}$$ provided by \cite[Lemma 2.2]{Shi:nondeg} completes the proof of normality.

That the inclusion $\mathbb{F}_{\mathcal{E}}\hookrightarrow \mathbb{F}_{\mathcal{A}}$ is faithfully flat follows from Lemma \ref{lem:faithfulflatness}.

It remains to identify the quotient Hopf algebra $\pi:\mathbb{F}_{\mathcal{A}}\twoheadrightarrow\mathbb{F}_{\mathcal{A}}/\mathbb{F}_{\mathcal{A}}\mathbb{F}_{\mathcal{E}}^+$. Note that we have $$\mathrm{Mod}_{\mathrm{Comod}_{\mathcal{A}}(\mathbb{F}_{\mathcal{A}})}(\mathbb{F}_{\mathcal{E}})\simeq \mathcal{A}\boxtimes_{\mathcal{E}}\mathcal{A}^{\mathrm{mop}}$$ as finite tensor categories by Lemma \ref{lem:relativetensorasmodules}. Namely, under the equivalence $\mathcal{E}\boxtimes\mathcal{E}^{\mathrm{mop}}\simeq \mathrm{Comod}_{\mathcal{E}}(\mathbb{F}_{\mathcal{E}})$, the algebra $T^R_{\mathcal{E}}(\mathbbm{1})$ corresponds to $\mathbb{F}_{\mathcal{E}}$. In particular, the tensor structure on $\mathrm{Mod}_{\mathrm{Comod}_{\mathcal{A}}(\mathbb{F}_{\mathcal{A}})}(\mathbb{F}_{\mathcal{E}})$ arises from the commutative algebra structure on $\mathbb{F}_{\mathcal{E}}$. To prove that $\mathbb{F}_{\mathcal{E}/\mathcal{A}}\cong \mathbb{F}_{\mathcal{A}}/\mathbb{F}_{\mathcal{A}}\mathbb{F}^+_{\mathcal{E}}$ as Hopf algebras, it will therefore be enough to show that there is an equivalence of tensor categories
$$\mathrm{Mod}_{\mathrm{Comod}_{\mathcal{A}}(\mathbb{F}_{\mathcal{A}})}(\mathbb{F}_{\mathcal{E}})\xrightarrow{\simeq} \mathrm{Comod}_{\mathcal{A}}(\mathbb{F}_{\mathcal{A}}/\mathbb{F}_{\mathcal{A}}\mathbb{F}_{\mathcal{E}}^+).$$ In order to do so, we will use a generalization of \cite[Theorem 1]{Tak}.

To define the desired functor observe that an object of $\mathrm{Mod}_{\mathrm{Comod}_{\mathcal{A}}(\mathbb{F}_{\mathcal{A}})}(\mathbb{F}_{\mathcal{E}})$ is an object $X$ of $\mathcal{A}$ equipped with a right action $n:X\otimes\mathbb{F}_{\mathcal{E}}\rightarrow X$ and a right coaction $\rho:X\rightarrow X\otimes\mathbb{F}_{\mathcal{A}}$ such that the following diagram commutes:
$$\begin{tikzcd}
X\otimes\mathbb{F}_{\mathcal{E}} \arrow[rr, "n"] \arrow[d, "\rho\otimes\Delta"']                                                                       &                                                                                                                        & X \arrow[d, "\rho"]              \\
X\otimes\mathbb{F}_{\mathcal{A}}\otimes\mathbb{F}_{\mathcal{E}}\otimes\mathbb{F}_{\mathcal{E}} \arrow[r, "\mathrm{id}\otimes\beta\otimes\mathrm{id}"'] & X\otimes\mathbb{F}_{\mathcal{E}}\otimes\mathbb{F}_{\mathcal{A}}\otimes\mathbb{F}_{\mathcal{E}} \arrow[r, "n\otimes m"'] & X\otimes\mathbb{F}_{\mathcal{A}}.
\end{tikzcd}$$ 
Postcomposing with the projection $\pi:\mathbb{F}_{\mathcal{A}}\twoheadrightarrow\mathbb{F}_{\mathcal{A}}/\mathbb{F}_{\mathcal{A}}\mathbb{F}_{\mathcal{E}}^+$, it follows that $$(\mathrm{id}\otimes\pi)\circ (n\otimes \mathrm{id})\circ(\mathrm{id}\otimes\beta\otimes\mathrm{id})\circ(\rho\otimes\mathrm{id}) = (\mathrm{id}\otimes\pi)\circ\rho \circ n,$$ so that $\rho$ induces a right $\mathbb{F}_{\mathcal{A}}/\mathbb{F}_{\mathcal{A}}\mathbb{F}_{\mathcal{E}}^+$-comodule structure on $X\otimes_{\mathbb{F}_{\mathcal{E}}}\mathbbm{1}$. This assignment is manifestly functorial. Additionally, it is straightforward to check that it is compatible with the monoidal structures. We therefore have a tensor functor
$$(-)\otimes_{\mathbb{F}_{\mathcal{E}}}\mathbbm{1}:\mathrm{Mod}_{\mathrm{Comod}_{\mathcal{A}}(\mathbb{F}_{\mathcal{A}})}(\mathbb{F}_{\mathcal{E}})\rightarrow \mathrm{Comod}_{\mathcal{A}}(\mathbb{F}_{\mathcal{A}}/\mathbb{F}_{\mathcal{A}}\mathbb{F}_{\mathcal{E}}^+).$$
This functor is left-exact as both its source and target are finite tensor categories. We claim that $(-)\otimes_{\mathbb{F}_{\mathcal{E}}}\mathbbm{1}$ is also faithful. Namely, we have seen in Lemma \ref{lem:technicalrelativetensorfunctortensorproduct} above that the composite functor $$\mathrm{Mod}_{\mathrm{Comod}_{\mathcal{A}}(\mathbb{F}_{\mathcal{A}})}(\mathbb{F}_{\mathcal{E}})\xrightarrow{(-)\otimes_{\mathbb{F}_{\mathcal{E}}}\mathbbm{1}}\mathrm{Comod}_{\mathcal{A}}(\mathbb{F}_{\mathcal{A}}/\mathbb{F}_{\mathcal{A}}\mathbb{F}_{\mathcal{E}}^+)\rightarrow\mathcal{A}$$ is identified with the faithful tensor functor $T_{\mathcal{E}/\mathcal{A}}:\mathcal{A}\boxtimes_{\mathcal{E}}\mathcal{A}^{\mathrm{rev}}\rightarrow\mathcal{A}$. It therefore follows from \cite[Lemma 4.3]{BZBJ} that $(-)\otimes_{\mathbb{F}_{\mathcal{E}}}\mathbbm{1}$ is conservative.

We only have left to prove that $(-)\otimes_{\mathbb{F}_{\mathcal{E}}}\mathbbm{1}$ is an equivalence. In order to do so, we will use the explicit form of its right adjoint provided by the cotensor product functor $(-)\Box^{\pi}\mathbb{F}_{\mathcal{A}}$. More precisely, for any right $\mathbb{F}_{\mathcal{A}}/\mathbb{F}_{\mathcal{A}}\mathbb{F}_{\mathcal{E}}^+$-comodule $Y$ in $\mathcal{A}$, we can consider the equalizer
$$\begin{tikzcd}[sep=small]
Y\Box^{\pi}\mathbb{F}_{\mathcal{A}} \arrow[r, dashed] & Y\otimes\mathbb{F}_{\mathcal{A}} \arrow[r, shift left] \arrow[r, shift right] & Y\otimes\mathbb{F}_{\mathcal{A}}/\mathbb{F}_{\mathcal{A}}\mathbb{F}_{\mathcal{E}}^+\otimes\mathbb{F}_{\mathcal{A}}.
\end{tikzcd}$$
The object $Y\Box^{\pi}\mathbb{F}_{\mathcal{A}}$ therefore carries a natural right $\mathbb{F}_{\mathcal{A}}$-comodule structure. It also carries a compatible right $\mathbb{F}_{\mathcal{E}}$-module structure, and thereby yields an object of $\mathrm{Mod}_{\mathrm{Comod}_{\mathcal{A}}(\mathbb{F}_{\mathcal{A}})}(\mathbb{F}_{\mathcal{E}})$. 
The unit of the adjunction between $(-)\otimes_{\mathbb{F}_{\mathcal{E}}}\mathbbm{1}$ and $ (-)\Box^{\pi}\mathbb{F}_{\mathcal{A}}$ is the map $$\xi_X: X\rightarrow (X\otimes_{\mathbb{F}_{\mathcal{E}}}\mathbbm{1})\Box^{\pi}\mathbb{F}_{\mathcal{A}}$$ induced by the $\mathbb{F}_{\mathcal{A}}$-comodule structure on $X$. The counit of this adjunction is the map $$\upsilon_Y:\big(Y\Box^{\pi}\mathbb{F}_{\mathcal{A}}\big)\otimes_{\mathbb{F}_{\mathcal{E}}}\mathbbm{1}\rightarrow Y$$ induced by $\pi:\mathbb{F}_{\mathcal{A}}\rightarrow \mathbb{F}_{\mathcal{A}}/\mathbb{F}_{\mathcal{A}}\mathbb{F}_{\mathcal{E}}^+$. It is easy to check that $\xi$ and $\upsilon$ satisfy the triangle identities. Now, the functor $(-)\otimes_{\mathbb{F}_{\mathcal{E}}}\mathbbm{1}$ is left-exact, so that the counit $\upsilon$ is an isomorphism. Finally, an adjunction for which both the counit is an isomorphism and the left adjoint is conservative is necessarily an equivalence.
\end{proof}

\begin{Remark}\label{rem:exactsequencecomod}
The notion of an exact sequence of tensor categories from \cite{BN} generalizes that of an exact sequence of Hopf algebras. In particular, an exact sequences of Hopf algebras in $\mathrm{Vec}$ leads to exact sequence of tensor categories by taking (co)modules. Conversely, in the presence of a fiber functor to $\mathrm{Vec}$, an exact sequences of tensor categories yields an exact sequence of Hopf algebras.

Our last theorem above may be interpreted as a version of this correspondence relative to the finite braided tensor category $\mathcal{A}$. We will not make this notion precise here because it is orthogonal to our purposes. We will nonetheless point out that, under the hypotheses of Theorem \ref{thm:exactsequenceHopf}, the exact sequence $\mathbb{F}_{\mathcal{E}}\hookrightarrow\mathbb{F}_{\mathcal{A}}\twoheadrightarrow\mathbb{F}_{\mathcal{E}/\mathcal{A}}$ of Hopf algebras in $\mathcal{A}$ corresponds to the exact sequence of tensor categories relative to $\mathcal{A}$ given by $$\begin{tikzcd}[sep=small]
\mathrm{Comod}_{\mathcal{A}}(\mathbb{F}_{\mathcal{E}}) \arrow[r] \arrow[d, "\rotatebox{-90}{$\simeq$}"] & \mathrm{Comod}_{\mathcal{A}}(\mathbb{F}_{\mathcal{A}}) \arrow[r] \arrow[d, "\rotatebox{-90}{$\simeq$}"] & \mathrm{Comod}_{\mathcal{A}}(\mathbb{F}_{\mathcal{E}/\mathcal{A}}) \arrow[d, "\rotatebox{-90}{$\simeq$}"] \\
\mathcal{A}\boxtimes\mathcal{E}^{\mathrm{mop}} \arrow[r]             & \mathcal{A}\boxtimes\mathcal{A}^{\mathrm{mop}} \arrow[r]             & \mathcal{A}\boxtimes_{\mathcal{E}}\mathcal{A}^{\mathrm{mop}}.
\end{tikzcd}$$
\end{Remark}

\subsection{The Relative Canonical Pairing}

We now show that the canonical pairing $\omega_{\mathcal{A}}$ on the universal Hopf algebra $\mathbb{F}_{\mathcal{A}}$ constructed in \cite{Lyu:modular} and recalled above in Remark \ref{rem:Shimizucomparison} descends to a pairing $\omega_{\mathcal{E}/\mathcal{A}}$ on $\mathbb{F}_{\mathcal{E}/\mathcal{A}}$. The pairing $\omega_{\mathcal{A}}$ plays a crucial role in \cite{Shi:nondeg, BJSS}. It is therefore not surprising that the pairing $\omega_{\mathcal{E}/\mathcal{A}}$ will feature prominently in our subsequent discussion of the relative non-degeneracy conditions for $\mathcal{A}$.

\begin{Proposition}\label{prop:inducedpairing}
The canonical pairing $\omega_{\mathcal{A}}$ on the universal Hopf algebra $\mathbb{F}_{\mathcal{A}}$ descends to a pairing $\omega_{\mathcal{E}/\mathcal{A}}$ on $\mathbb{F}_{\mathcal{E}/\mathcal{A}}$.
\end{Proposition}
\begin{proof}
Thanks to Theorem \ref{thm:exactsequenceHopf}, it will be enough to show that $\mathbb{F}_{\mathcal{A}}\mathbb{F}_{\mathcal{E}}^+$ is contained in the left and right kernels of the pairing $\omega_{\mathcal{A}}$. To see that $\mathbb{F}_{\mathcal{A}}\mathbb{F}_{\mathcal{E}}^+$ is contained in the left kernel, it suffices to prove that the composite $$\mathbb{F}_{\mathcal{A}}\otimes\mathbb{F}_{\mathcal{A}}\otimes\mathbb{F}_{\mathcal{E}} \xrightarrow{\mathrm{id}\otimes m}\mathbb{F}_{\mathcal{A}}\otimes\mathbb{F}_{\mathcal{A}}\xrightarrow{\omega_{\mathcal{A}}}\mathbbm{1}$$ coincides with $\omega_{\mathcal{A}}\otimes\epsilon$. In order to see this, let $E$ be a compact object of $\mathcal{E}$, and let $A$, $B$ be compact objects of $\mathcal{A}$. Precomposing the first map above with $\iota_A\otimes\iota_B\otimes\iota_E$, we obtain the map in $\mathcal{A}$
\begin{align*}
(A^*\otimes A)\otimes (B^*\otimes B)\otimes (E^*\otimes E)
&\xrightarrow{\mathrm{id}\otimes\beta\otimes\mathrm{id}} (A^*\otimes A)\otimes ((E^*\otimes B^*)\otimes (B\otimes E))\\
&\xrightarrow{\beta\otimes\mathrm{id}} ((E^*\otimes B^*)\otimes A^*)\otimes (A\otimes (B\otimes E))\\
&\xrightarrow{\mathrm{id}\otimes\beta\otimes\mathrm{id}} (A^*\otimes A)\otimes ((E^*\otimes B^*)\otimes (B\otimes E))\\
&\xrightarrow{\mathrm{ev}_A\otimes\mathrm{ev}_{B\otimes E}}\mathbbm{1}.
\end{align*}
But, $E$ is in $\mathcal{Z}_{(2)}(\mathcal{A})$, so that this composite is simply 
\begin{align*}
(A^*\otimes A)\otimes (B^*\otimes B)\otimes (E^*\otimes E)&\xrightarrow{\mathrm{id}\otimes\mathrm{ev}_E} (A^*\otimes A)\otimes (B^*\otimes B)\\
&\xrightarrow{\mathrm{id}\otimes\beta^2\otimes\mathrm{id}} (A^*\otimes A)\otimes (B^*\otimes B)\xrightarrow{\mathrm{ev}_A\otimes\mathrm{ev}_B}\mathbbm{1},
\end{align*}
which induces $\omega_{\mathcal{A}}\otimes\epsilon$ on coends. This shows that $\mathbb{F}_{\mathcal{A}}\mathbb{F}_{\mathcal{E}}^+$ indeed lies in the left kernel of $\omega_{\mathcal{A}}$. The proof that $\mathbb{F}_{\mathcal{A}}\mathbb{F}_{\mathcal{E}}^+$ lies in the right ideal of $\omega_{\mathcal{A}}$ is analogous.
\end{proof}

The quotient of the universal Hopf algebra $\mathbb{F}_{\mathcal{A}}$ by the kernel of the canonical pairing $\omega_{\mathcal{A}}$ was studied in \cite[Section 3]{Lyu:modular}. We give a more precise characterization of both in Corollary \ref{cor:comparisonLyubashenko} below.

\subsection{The Relative Drinfeld Center}

For use in the next section, we will also need to understand the relationship between the relative universal Hopf algebra $\mathbb{F}_{\mathcal{E}/\mathcal{A}}$ and the relative Drinfeld center of $\mathcal{A}$. This notion can be traced back to \cite{DMNO} in the semisimple case.

\begin{Definition}
Let $\mathcal{B}$ be a finite braided tensor category, and let $\mathcal{C}$ be a $\mathcal{B}$-enriched finite tensor category with enrichment provided by the braided tensor functor $F:\mathcal{B}\rightarrow\mathcal{Z}(\mathcal{C})$. The relative Drinfeld center of $\mathcal{C}$, denoted by $\mathcal{Z}(\mathcal{C},\mathcal{B})$, is the full finite tensor subcategory of $\mathcal{Z}(\mathcal{C})$ on those objects $Z$ in $\mathcal{Z}(\mathcal{C})$ such that $\beta^2_{Z,F(B)} = \mathrm{id}$ for every $B$ in $\mathcal{B}$.
\end{Definition}

We begin by the following variant of a well-known identification. We refer the reader to \cite[Proposition 3.34]{Lau} and \cite[Lemma 2.35]{Kin} for related observations.

\begin{Lemma}\label{lem:relativecenter}
There is an equivalence of tensor categories $$\mathrm{End}_{\mathcal{C}\boxtimes_{\mathcal{B}}\mathcal{C}^{\mathrm{mop}}}(\mathcal{C})\simeq\mathcal{Z}(\mathcal{C},\mathcal{B}).$$
\end{Lemma}
\begin{proof}
It is well-known, and follows easily by inspecting the definitions, that $\mathrm{End}_{\mathcal{C}\boxtimes\mathcal{C}^{\mathrm{mop}}}(\mathcal{C})\simeq\mathcal{Z}(\mathcal{C})$ as finite tensor categories (see for instance \cite[Proposition 7.13.8]{EGNO}). The equivalence is given explicitly by sending $Z$ in $\mathcal{Z}(\mathcal{C})$ to the $\mathcal{C}$-$\mathcal{C}$-bimodule functor $C\mapsto C\otimes Z$. Now, by definition, $\mathrm{End}_{\mathcal{C}\boxtimes\mathcal{C}^{\mathrm{mop}}}(\mathcal{C})$ is the full tensor subcategory of $\mathrm{End}_{\mathcal{C}\boxtimes_{\mathcal{B}}\mathcal{C}^{\mathrm{mop}}}(\mathcal{C})$ on those $\mathcal{C}$-$\mathcal{C}$-bimodule functors intertwining the two $\mathcal{B}$-module structures. Under the equivalence $\mathrm{End}_{\mathcal{C}\boxtimes\mathcal{C}^{\mathrm{mop}}}(\mathcal{C})\simeq\mathcal{Z}(\mathcal{C})$, this corresponds precisely to those objects $Z$ in $\mathcal{Z}(\mathcal{C})$ for which the diagram
$$\begin{tikzcd}[sep = small]
F(B)\otimes C\otimes Z \arrow[rr, "{\beta_{F(B),C\otimes Z}}"] \arrow[rd, "{\beta_{F(B),C}\otimes \mathrm{id}}"'] &                     & C\otimes Z\otimes F(B) \arrow[ld, "{\mathrm{id}\otimes\beta_{Z,F(B)}}"] \\
& C\otimes F(B)\otimes Z, &       
\end{tikzcd}$$
commutes for all $C$ in $\mathcal{C}$ and $B$ in $\mathcal{B}$. This selects exactly the full tensor subcategory $\mathcal{Z}(\mathcal{C},\mathcal{B})$ of $\mathcal{Z}(\mathcal{C})$.
\end{proof}

As a consequence, we obtain the following generalization of and improvement upon \cite[Lemma 3.7]{Shi:nondeg}. 

\begin{Proposition}\label{prop:RelativeCenterModule}
Let $\mathcal{A}$ be a faithfully flat $\mathcal{E}$-enriched finite braided tensor category. There is an equivalence of tensor categories $$\mathrm{Mod}_{\mathcal{A}}(\mathbb{F}_{\mathcal{E}/\mathcal{A}})\simeq\mathcal{Z}(\mathcal{A},\mathcal{E}).$$
\end{Proposition}
\begin{proof}
Combining Lemma \ref{lem:relativecenter} with the relative Eilenberg-Watts theorem, see for instance \cite[Theorem 3.3]{DSPS:balanced} or \cite[Lemma 5.7]{BJS}, there is an equivalence of tensor categories $$\mathcal{Z}(\mathcal{A},\mathcal{E})^{\mathrm{mop}}\simeq \mathrm{Bimod}_{\mathrm{Comod}_{\mathcal{A}}(\mathbb{F}_{\mathcal{E}/\mathcal{A}})}(\mathbb{F}_{\mathcal{E}/\mathcal{A}}).$$ In order to conclude the proof, it is therefore enough to identify the right hand-side with the tensor category $\mathrm{LMod}_{\mathcal{A}}(\mathbb{F}_{\mathcal{E}/\mathcal{A}})$, whose monoidal structure is provided by the Hopf algebra structure on $\mathbb{F}_{\mathcal{E}/\mathcal{A}}$. But there is an equivalence of left $\mathrm{Comod}_{\mathcal{A}}(\mathbb{F}_{\mathcal{E}/\mathcal{A}})$-module categories
$$\begin{tabular}{c c c}
$\mathcal{A}$ & $\xrightarrow{\sim}$ & $\mathrm{Mod}_{\mathrm{Comod}_{\mathcal{A}}(\mathbb{F}_{\mathcal{E}/\mathcal{A}})}(\mathbb{F}_{\mathcal{E}/\mathcal{A}})$.\\
$A$ & $\mapsto$ & $A\otimes \mathbb{F}_{\mathcal{E}/\mathcal{A}}$
\end{tabular}$$ This a variant of the fundamental theorem on Hopf modules, which follows by applying the Barr-Beck monadicity theorem to the dominant tensor functor $\mathrm{Comod}_{\mathcal{A}}(\mathbb{F}_{\mathcal{E}/\mathcal{A}})\rightarrow \mathcal{A}$.

Given a left $\mathbb{F}_{\mathcal{E}/\mathcal{A}}$-module in $\mathcal{A}$ with multiplication $n:\mathbb{F}_{\mathcal{E}/\mathcal{A}}\otimes M\rightarrow M$, we can endow $M\otimes \mathbb{F}_{\mathcal{E}/\mathcal{A}}$ in $\mathrm{Mod}_{\mathrm{Comod}_{\mathcal{A}}(\mathbb{F}_{\mathcal{E}/\mathcal{A}})}(\mathbb{F}_{\mathcal{E}/\mathcal{A}})$ with the left $\mathbb{F}_{\mathcal{E}/\mathcal{A}}$-module structure given by $$\mathbb{F}_{\mathcal{E}/\mathcal{A}}\otimes M\otimes \mathbb{F}_{\mathcal{E}/\mathcal{A}}\xrightarrow{(\mathrm{id}\otimes\beta\otimes\mathrm{id})\circ(\Delta\otimes\mathrm{id})}\mathbb{F}_{\mathcal{E}/\mathcal{A}}\otimes M\otimes \mathbb{F}_{\mathcal{E}/\mathcal{A}}\otimes\mathbb{F}_{\mathcal{E}/\mathcal{A}}\xrightarrow{n\otimes m}M\otimes \mathbb{F}_{\mathcal{E}/\mathcal{A}}.$$
It follows by inspection that this defines a tensor functor
$$\begin{tabular}{c c c}
$\mathrm{LMod}_{\mathcal{A}}(\mathbb{F}_{\mathcal{E}/\mathcal{A}})$ & $\xrightarrow{\sim}$ & $\mathrm{Bimod}_{\mathrm{Comod}_{\mathcal{A}}(\mathbb{F}_{\mathcal{E}/\mathcal{A}})}(\mathbb{F}_{\mathcal{E}/\mathcal{A}}),$\\
$M$ & $\mapsto$ & $M\otimes \mathbb{F}_{\mathcal{E}/\mathcal{A}}$
\end{tabular}$$
which is an equivalence by the discussion of the preceding paragraph. But, using the invertibility of the antipode, we have that $$\mathrm{LMod}_{\mathcal{A}}(\mathbb{F}_{\mathcal{E}/\mathcal{A}})\simeq \mathrm{Mod}_{\mathcal{A}}(\mathbb{F}_{\mathcal{E}/\mathcal{A}})^{\mathrm{mop}}$$ as tensor categories, which concludes the proof of the result.
\end{proof}

\begin{Remark}
In Remark \ref{rem:exactsequencecomod} above, we considered the relative exact sequence of tensor categories given by taking comodules for the exact sequence of Hopf algebras of Theorem \ref{thm:exactsequenceHopf}. We also obtain an exact sequence relative to $\mathcal{A}$ by taking modules: $$\begin{tikzcd}[sep = small]
\mathrm{Mod}_{\mathcal{A}}(\mathbb{F}_{\mathcal{E}}) \arrow[d, "\rotatebox{90}{$\simeq$}"'] & \mathrm{Mod}_{\mathcal{A}}(\mathbb{F}_{\mathcal{A}}) \arrow[l] \arrow[d, "\rotatebox{90}{$\simeq$}"'] & \mathrm{Mod}_{\mathcal{A}}(\mathbb{F}_{\mathcal{E}/\mathcal{A}}) \arrow[l] \arrow[d, "\rotatebox{90}{$\simeq$}"'] \\
\mathcal{Z}_{\mathcal{E}}(\mathcal{A}) & \mathcal{Z}(\mathcal{A}) \arrow[l]  & {\mathcal{Z}(\mathcal{A},\mathcal{E}).} \arrow[l]       
\end{tikzcd}$$ Above, the notation $\mathcal{Z}_{\mathcal{E}}(\mathcal{A})$ refers to the tensor category whose objects are pairs consisting of an object of $\mathcal{A}$ equipped with a half-braiding with respect to objects of $\mathcal{E}$. The left vertical identification is supplied by \cite[Lemma 4.5]{Shi:nondeg}.
\end{Remark}

\section{Relative Invertibility and Non-Degeneracy}

As above we work over a fixed finite symmetric tensor category $\mathcal{E}$ over a perfect field $\mathbbm{k}$. We begin by spelling out the invertibility criterion of \cite{BJSS} specialized to the Morita 4-category $\mathrm{Mor_2^{pre}}(\mathbf{Pr}_{\mathcal{E}})$ of $\mathcal{E}$-enriched pre-tensor categories. We go on to study in depth the invertibility of $\mathcal{E}$-enriched finite braided tensor categories. Our main tool is the relative universal Hopf algebra and its pairing, which were introduced above.

\subsection{An Algebraic Criterion for Relative Invertibility}

Let us now fix an $\mathcal{E}$-enriched finite braided tensor category $\mathcal{A}$. In addition, we write $F:\mathcal{E}\rightarrow \mathcal{Z}_{(2)}(\mathcal{A})$ for the symmetric tensor functor supplying its $\mathcal{E}$-enriched structure. We think of $\mathcal{A}$ as an object of the Morita 4-category $\mathrm{Mor_2^{pre}}(\mathbf{Pr}_{\mathcal{E}})$.

There are various pre-tensor categories associated to $\mathcal{A}$. In the previous section, we have mentioned the relative Drinfeld center $\mathcal{Z}(\mathcal{A},\mathcal{E})$, and focused our attention on the relative enveloping algebra $\mathcal{A}\boxtimes_{\mathcal{E}}\mathcal{A}^{\mathrm{mop}}$. Another pre-tensor category of interest is the relative Harish-Chandra category of $\mathcal{A}$ defined by $$\mathrm{HC}_{\mathcal{E}}(\mathcal{A}) := \mathcal{A}\boxtimes_{\mathcal{A}\boxtimes_{\mathcal{E}}\mathcal{A}^{\mathrm{rev}}}\mathcal{A}^{\mathrm{mop}}.$$ In case $\mathcal{E} = \mathrm{Vec}$, we simply omit this symbol from the notations.

\begin{Lemma}\label{lem:HCsymmetriccenter}
There is an equivalence of tensor categories $$\mathrm{End}_{\mathrm{HC}_{\mathcal{E}}(\mathcal{A})}(\mathcal{A})\simeq\mathcal{Z}_{(2)}(\mathcal{A}).$$
\end{Lemma}
\begin{proof}
It was established in \cite[Proposition 3.7]{BJSS} that $\mathrm{End}_{\mathrm{HC}(\mathcal{A})}(\mathcal{A})\simeq\mathcal{Z}_{(2)}(\mathcal{A})$. More precisely, Lemma \ref{lem:relativecenter} with $\mathcal{B} = \mathcal{A}\boxtimes\mathcal{A}^{\mathrm{rev}}$ and $\mathcal{C} = \mathcal{A}$ yields equivalences $$\mathrm{End}_{\mathrm{HC}(\mathcal{A})}(\mathcal{A})\simeq\mathcal{Z}(\mathcal{A},\mathcal{A}\boxtimes\mathcal{A}^{\mathrm{rev}})\simeq\mathcal{Z}_{(2)}(\mathcal{A}),$$ where the second equivalence holds by direct inspection. Setting $\mathcal{B}=\mathcal{A}\boxtimes_{\mathcal{E}}\mathcal{A}^{\mathrm{rev}}$, we similarly find $$\mathrm{End}_{\mathrm{HC}_{\mathcal{E}}(\mathcal{A})}(\mathcal{A})\simeq\mathcal{Z}(\mathcal{A},\mathcal{A}\boxtimes_{\mathcal{E}}\mathcal{A}^{\mathrm{rev}}).$$ But, the braided tensor functor $\mathcal{A}\boxtimes\mathcal{A}^{\mathrm{rev}}\rightarrow\mathcal{Z}(\mathcal{A})$ factors through $\mathcal{A}\boxtimes\mathcal{A}^{\mathrm{rev}}\twoheadrightarrow \mathcal{A}\boxtimes_{\mathcal{E}}\mathcal{A}^{\mathrm{rev}}$, so that their images in $\mathcal{Z}(\mathcal{A})$ coincide. This shows that $\mathcal{Z}(\mathcal{A},\mathcal{A}\boxtimes_{\mathcal{E}}\mathcal{A}^{\mathrm{rev}}) = \mathcal{Z}(\mathcal{A},\mathcal{A}\boxtimes\mathcal{A}^{\mathrm{rev}})$ as tensor categories, and thereby concludes the proof.
\end{proof}

Thanks to Proposition \ref{prop:3dualizability} and Lemmas \ref{lem:relativecenter} and \ref{lem:HCsymmetriccenter}, we find that specializing the invertibility criterion given in \cite[Theorem 2.30]{BJSS} to $\mathrm{Mor_2^{pre}}(\mathbf{Pr}_{\mathcal{E}})$ yields the following statement.

\begin{Theorem}[\cite{BJSS}]\label{thm:relativeinvertibilityBJSS}
An $\mathcal{E}$-enriched finite braided tensor category $\mathcal{A}$ is an invertible object of $\mathrm{Mor_2^{pre}}(\mathbf{Pr}_{\mathcal{E}})$ if and only if the following conditions are satisfied:
\begin{enumerate}
    \item It is $\mathcal{E}$-non-degenerate, i.e.\ the functor $F:\mathcal{E}\rightarrow \mathcal{Z}_{(2)}(\mathcal{A})$ is an equivalence.
    \item It is $\mathcal{E}$-factorizable, i.e.\ the canonical functor $\mathcal{A}\boxtimes_{\mathcal{E}}\mathcal{A}^{\mathrm{rev}}\rightarrow \mathcal{Z}(\mathcal{A},\mathcal{E})$ is an equivalence.
    \item It is $\mathcal{E}$-cofactorizable, i.e.\ the canonical functor $\mathrm{HC}_{\mathcal{E}}(\mathcal{A})\rightarrow\mathrm{End}_{\mathcal{E}}(\mathcal{A})$ is an equivalence.
\end{enumerate}
\end{Theorem}

The main result of this section is the proof that these three conditions are equivalent when $\mathcal{A}$ is an $\mathcal{E}$-enriched finite braided tensor category such that $F:\mathcal{E}\rightarrow\mathcal{Z}_{(2)}(\mathcal{A})$ is faithful. Faithfulness is necessary because the second and third condition only involve the image of the functor $F$. When the monoidal unit of $\mathcal{E}$ is simple, this is automatic.

In the case $\mathcal{E} = \mathrm{Vec}$, the equivalence between the first two conditions was obtained in \cite[Theorem 1.1]{Shi:nondeg}, and the equivalence between the last two conditions was derived in \cite[Theorem 1.6]{BJSS}. These results were proven by comparing the above conditions to the non-degeneracy of the pairing $\omega_{\mathcal{A}}$ on the universal Hopf algebra $\mathbb{F}_{\mathcal{A}}$. In order to establish the $\mathcal{E}$-enriched version, we will proceed analogously using the relative universal Hopf algebra $\mathbb{F}_{\mathcal{E}/\mathcal{A}}$ and its pairing $\omega_{\mathcal{E}/\mathcal{A}}$.

\subsection{Relative Non-Degeneracy and Relative Factorizability}

We begin by considering a generalization of \cite[Theorem 4.2]{Shi:nondeg}, that is, we relate $\mathcal{E}$-non-degeneracy and $\mathcal{E}$-factorizability.

\begin{Theorem}\label{thm:nondegeneratevsfactorizable}
Assume that $\mathbbm{k}$ is algebraically closed, and let $\mathcal{A}$ be a faithfully flat $\mathcal{E}$-enriched finite braided tensor category. Then, $\mathcal{A}$ is $\mathcal{E}$-non-degenerate if and only if $\mathcal{A}$ is $\mathcal{E}$-factorizable.
\end{Theorem}
\begin{proof}
Without loss of generality, we can assume that $\mathcal{A}$ and therefore also $\mathcal{E}$ has simple monoidal unit $\mathbbm{1}$. Given that $\mathbbm{k}$ is algebraically closed, we can make use of the Frobenius-Perron dimension reviewed in \cite[Chapters 3 and 6]{EGNO}. Then, we have $$\mathrm{FPdim}(\mathcal{A}\boxtimes_{\mathcal{E}}\mathcal{A}^{\mathrm{rev}}) = \mathrm{FPdim}(\mathcal{A})^2/\mathrm{FPdim}(\mathcal{E})$$ by fully faithfulness of $\mathcal{E}\hookrightarrow\mathcal{A}$. Namely, we have seen that $\mathcal{A}\boxtimes_{\mathcal{E}}\mathcal{A}^{\mathrm{rev}} = \mathrm{Mod}_{\mathcal{A}\boxtimes\mathcal{A}^{\mathrm{rev}}}(T_{\mathcal{E}}^{R}(\mathbbm{1}))$ in Lemma \ref{lem:relativetensorasmodules}. But, the connected exact commutative algebra $T_{\mathcal{E}}^{R}(\mathbbm{1})$ in $\mathcal{E}\boxtimes\mathcal{E}^{\mathrm{rev}}$ has Frobenius-Perron dimension equal to $\mathrm{FPdim}(\mathcal{E})$ by \cite[Lemma 6.2.4]{EGNO}. The claim follows from the same result but now applied to the connected exact commutative algebra $T_{\mathcal{E}}^{R}(\mathbbm{1})$ in $\mathcal{A}\boxtimes\mathcal{A}^{\mathrm{rev}}$, which has the same Frobenius-Perron dimension.

Let us write $\mathcal{A}\vee\mathcal{A}^{\mathrm{rev}}$ for the image of the tensor functor $\mathcal{A}\boxtimes\mathcal{A}^{\mathrm{rev}}\rightarrow\mathcal{Z}(\mathcal{A})$. It is a finite tensor subcategory as $\mathcal{Z}(\mathcal{A})$ is braided. By \cite[Lemma 4.8]{Shi:nondeg}, we have $$\mathrm{FPdim}(\mathcal{A}\vee \mathcal{A}^{\mathrm{rev}})\cdot\mathrm{FPdim}(\mathcal{A}\cap \mathcal{A}^{\mathrm{rev}}) = \mathrm{FPdim}(\mathcal{A})^2.$$ But, we have $\mathcal{A}\cap \mathcal{A}^{\mathrm{rev}} = \mathcal{Z}_{(2)}(\mathcal{A})$ by inspection. Moreover, we also have that $\mathcal{A}\vee \mathcal{A}^{\mathrm{rev}} = \mathcal{Z}(\mathcal{A}, \mathcal{Z}_{(2)}(\mathcal{A}))$ by \cite[Lemma 2.8]{Mu} and \cite[Theorem 4.9]{Shi:nondeg}. It therefore follows that $$\mathrm{FPdim}(\mathcal{Z}(\mathcal{A}, \mathcal{Z}_{(2)}(\mathcal{A})))\cdot \mathrm{FPdim}(\mathcal{Z}_{(2)}(\mathcal{A})) = \mathrm{FPdim}(\mathcal{A})^2.$$ Now, the dominant tensor functor $\mathcal{A}\boxtimes\mathcal{A}^{\mathrm{rev}}\twoheadrightarrow\mathcal{Z}(\mathcal{A}, \mathcal{Z}_{(2)}(\mathcal{A}))$ factors through $\mathcal{A}\boxtimes_{\mathcal{E}}\mathcal{A}^{\mathrm{rev}}\twoheadrightarrow\mathcal{Z}(\mathcal{A}, \mathcal{Z}_{(2)}(\mathcal{A}))$. We therefore find that $\mathcal{A}$ is $\mathcal{E}$-factorizable if and only if $\mathrm{FPdim}(\mathcal{A}\boxtimes_{\mathcal{E}}\mathcal{A}^{\mathrm{rev}}) = \mathrm{FPdim}(\mathcal{Z}(\mathcal{A}, \mathcal{Z}_{(2)}(\mathcal{A})))$ by \cite[Proposition 6.3.4]{EGNO}. Thanks to the above equalities and \cite[Proposition 6.3.4]{EGNO}, this is the case if and only if $\mathrm{FPdim}(\mathcal{Z}_{(2)}(\mathcal{A})) = \mathrm{FPdim}(\mathcal{E})$. By \cite[Proposition 6.3.3]{EGNO}, this holds if and only if $\mathcal{E}\hookrightarrow\mathcal{Z}_{(2)}(\mathcal{A})$ is an equivalence. This concludes the proof.
\end{proof}

\begin{Remark}\label{rem:nondegeneracyarbitraryfield}
We now discuss how to generalize the proof of our last theorem to the case of an arbitrary perfect field. We emphasize that this is the only one of our result which does not have a complete proof over an arbitrary perfect field $\mathbbm{k}$.
Without loss of generality, we can again assume that $\mathcal{A}$ and therefore also $\mathcal{E}$ has simple monoidal unit $\mathbbm{1}$. Moreover, we can assume that $\mathbbm{k} = \mathrm{End}_{\mathcal{A}}(\mathbbm{1})$. Namely, as $\mathcal{A}$ is a braided tensor category, it must be linear over $\mathrm{End}_{\mathcal{A}}(\mathbbm{1})$. This observation also applies to $\mathcal{E}$. Let us write $\overline{\mathbbm{k}}$ for the algebraic closure of $\mathbbm{k}$. We consider the base changes $\overline{\mathcal{E}}$ and $\overline{\mathcal{A}}$ of $\mathcal{E}$ and $\mathcal{A}$ from $\mathbbm{k}$ to $\overline{\mathbbm{k}}$. These are finite tensor categories over $\overline{\mathbbm{k}}$ that have simple monoidal units as $\mathbbm{k} = \mathrm{End}_{\mathcal{A}}(\mathbbm{1})$. We claim that the symmetric center is stable under base change, i.e.\ we assert that $$\mathcal{Z}_{(2)}(\overline{\mathcal{A}}) \simeq \overline{\mathcal{Z}_{(2)}(\mathcal{A})}.$$ We expect that this can be obtained using a variant of the argument used to prove \cite[Lemma 2.1.4]{DS:compact}. However, so as to make this precise, one would need to generalize the theory of forms on finite semisimple tensor categories from \cite{EG:descent} to all finite tensor categories. As this would bring us too far afield, we leave the details to the interested reader. Finally, as the symmetric center is stable under base change, we find that $F:\mathcal{E}\rightarrow \mathcal{Z}_{(2)}(\mathcal{A})$ is an equivalence if and only if its base change is. Theorem \ref{thm:nondegeneratevsfactorizable} would then conclude the proof.
\end{Remark}

In order to complete the equivalence between the first two points in Theorem \ref{thm:relativeinvertibilityBJSS}, it remains to prove that if $\mathcal{A}$ is $\mathcal{E}$-factorizable, then it must be faithfully flat. In order to do so, it is necessary to make the very mild assumption that $F:\mathcal{E}\rightarrow\mathcal{Z}_{(2)}(\mathcal{A})$ is faithful. Namely, the condition of being $\mathcal{E}$-factorizable only depends on the image of $F$. We note that it is a standard result that $F$ is automatically faithful if $\mathcal{E}$ has simple monoidal unit. This condition is therefore really only needed to exclude pathological examples.

\begin{Proposition}\label{prop:faithfulfactorizablefaithfullyflat}
Let $\mathcal{A}$ be an $\mathcal{E}$-factorizable $\mathcal{E}$-enriched finite braided tensor category. If $F:\mathcal{E}\rightarrow\mathcal{Z}_{(2)}(\mathcal{A})$ is faithful, then $\mathcal{A}$ must be faithfully flat over $\mathcal{E}$.
\end{Proposition}
\begin{proof}
Let us write $\mathcal{F}$ for the image of the symmetric tensor functor $F:\mathcal{E}\rightarrow \mathcal{Z}_{(2)}(\mathcal{A})$. Let us also set $A:=F^R(\mathbbm{1})$, a commutative exact algebra. The canonical symmetric tensor functor $\mathcal{E}\rightarrow\mathcal{F}$ can be identified with the free $A$-module functor $\mathcal{E}\rightarrow\mathbf{Mod}_{\mathcal{E}}(A)$. Let us now consider the following commutative diagram: $$\begin{tikzcd}[sep=3ex]
\mathcal{A}\boxtimes_{\mathcal{E}}\mathcal{A}^{\mathrm{rev}} \arrow[d] \arrow[r, "\simeq"] & {\mathcal{Z}(\mathcal{A},\mathcal{E})}                    \\
\mathcal{A}\boxtimes_{\mathcal{F}}\mathcal{A}^{\mathrm{rev}} \arrow[r]                              & {\mathcal{Z}(\mathcal{A},\mathcal{F})}. \arrow[u, "\rotatebox{-90}{$\simeq$}"']
\end{tikzcd}$$ The left vertical arrow is induced by the universal property of the relative Deligne tensor product given that every $\mathcal{F}$-balanced functor out of $\mathcal{A}\boxtimes\mathcal{A}^{\mathrm{rev}}$ is $\mathcal{E}$-balanced. This shows in particular that the diagram is indeed commutative. Then, the top horizontal arrow is an equivalence as $\mathcal{A}$ is $\mathcal{E}$-factorizable. The right vertical arrow is also an equivalence by definition of the relative Drinfeld center. As the diagram commutes, this implies that the bottom composite is an equivalence, so that the functor $\mathcal{A}\boxtimes_{\mathcal{F}}\mathcal{A}^{\mathrm{rev}}\rightarrow\mathcal{Z}(\mathcal{A},\mathcal{F})$ is full. But it is also faithful as it is a tensor functor between finite tensor categories and it restricts to the identity on $\mathcal{F}$. It follows that the functor $\mathcal{A}\boxtimes_{\mathcal{E}}\mathcal{A}^{\mathrm{rev}}\rightarrow \mathcal{A}\boxtimes_{\mathcal{F}}\mathcal{A}^{\mathrm{rev}}$ must be fully faithful.

The following diagram is commutative by construction:
$$\begin{tikzcd}[sep=3ex]
\mathcal{F}\boxtimes_{\mathcal{E}}\mathcal{F}^{\mathrm{rev}} \arrow[r, hook] \arrow[d] & \mathcal{A}\boxtimes_{\mathcal{E}}\mathcal{A}^{\mathrm{rev}} \arrow[d, hook] \\
\mathcal{F} \arrow[r, hook]                                                                       & \mathcal{A}\boxtimes_{\mathcal{F}}\mathcal{A}^{\mathrm{rev}}.
\end{tikzcd}$$ As the two horizontal arrows are fully faithful by definition of $\mathcal{F}$, we find that $\mathcal{F}\boxtimes_{\mathcal{E}}\mathcal{F}^{\mathrm{rev}}\hookrightarrow \mathcal{F}$ is fully faithful. But it follows from the properties of the relative Deligne tensor product \cite{DSPS:balanced}, that this functor can be identified with the left $\mathcal{E}$-module functor $\mathbf{Mod}_{\mathcal{E}}(A\otimes A)\rightarrow \mathbf{Mod}_{\mathcal{E}}(A)$ given by relative tensor product with $A$ viewed as an $A\otimes A$-$A$-bimodule. In particular, there are isomorphisms $$\mathrm{Hom}_{\mathcal{E}}(E,A\otimes A)\cong \mathrm{Hom}_{A\otimes A}(E\otimes A\otimes A,A\otimes A)\cong \mathrm{Hom}_{A}(E\otimes A,A)\cong\mathrm{Hom}_{\mathcal{E}}(E,A)$$ that are natural in $E$. Appealing to the Yoneda lemma, this shows that $A\otimes A\cong A$ as objects of $\mathcal{E}$. But, as $F$ is faithful, $\mathbbm{1}$ must be a subobject of $A$, from which we deduce that $A = \mathbbm{1}$.
\end{proof}

\subsection{Relative Factorizability and Non-Degenerate Pairing}

Let $\mathcal{A}$ be a faithfully flat $\mathcal{E}$-enriched braided tensor category. We will now prove that $\mathcal{A}$ is $\mathcal{E}$-factorizable if and only if the pairing $\omega_{\mathcal{E}/\mathcal{A}}$ on $\mathbb{F}_{\mathcal{A}}$ is non-degenerate in the following sense.

\begin{Definition}
We say that the pairing $\omega_{\mathcal{E}/\mathcal{A}}$ on $\mathbb{F}_{\mathcal{E}/\mathcal{A}}$ is non-degenerate if there exists a map $\chi:\mathbbm{1}\rightarrow \mathbb{F}_{\mathcal{E}/\mathcal{A}}\otimes \mathbb{F}_{\mathcal{E}/\mathcal{A}}$ such that the snake equations hold: $$(\mathrm{id}\otimes\omega_{\mathcal{E}/\mathcal{A}})\circ (\chi\otimes\mathrm{id}) = \mathrm{id} = (\omega_{\mathcal{E}/\mathcal{A}}\otimes\mathrm{id})\circ (\mathrm{id}\otimes\chi).$$
\end{Definition}

We begin by recalling the canonical braided tensor functor appearing in the definition of (relative) factorizability appear. Viewing $\mathcal{A}$ as a plain finite braided tensor category, there is a braided tensor functor
\begin{equation}\label{eq:factorizability}\begin{tabular}{c c c}
$\mathcal{A}\boxtimes\mathcal{A}^{\mathrm{rev}}$ & $\rightarrow$ & $\mathcal{Z}(\mathcal{A}),$\\
$A\boxtimes B\ \ \ $ & $\mapsto$ & $(A\otimes B, \zeta_{A\otimes B})$
\end{tabular}\end{equation}
where $\zeta$ is the half-braiding on $A\otimes B$ given by $\zeta_{A\otimes B,-}=(\beta_{A,-}\otimes B)\circ (A\otimes\beta^{-1}_{-,B})$. It follows from the definitions that this functor lands in the full tensor subcategory $\mathcal{Z}(\mathcal{A},\mathcal{E})$. Moreover, it is straightforward to check that the above assignment is $\mathcal{E}$-balanced. Altogether, we obtain a canonical braided tensor functor $$\mathcal{A}\boxtimes_{\mathcal{E}}\mathcal{A}^{\mathrm{rev}}\rightarrow \mathcal{Z}(\mathcal{A},\mathcal{E}).$$

It was shown in \cite[Theorem 3.3]{Shi:nondeg} that the tensor functor $\mathcal{A}\boxtimes\mathcal{A}^{\mathrm{rev}}\rightarrow\mathcal{Z}(\mathcal{A})$ can be identified with a functor $\omega_{\mathcal{A}}^{\flat}:\mathrm{Comod}_{\mathcal{A}}(\mathbb{F}_{\mathcal{A}})\rightarrow\mathrm{Mod}_{\mathcal{A}}(\mathbb{F}_{\mathcal{A}})$ induced by $\omega_{\mathcal{A}}$. We will prove a relative version of this identification. Towards this goal, we consider a tensor functor $$\omega_{\mathcal{E}/\mathcal{A}}^{\flat}:\mathrm{Comod}_{\mathcal{A}}(\mathbb{F}_{\mathcal{E}/\mathcal{A}})\rightarrow\mathrm{Mod}_{\mathcal{A}}(\mathbb{F}_{\mathcal{E}/\mathcal{A}}).$$
This functor is given on objects by sending the $\mathbb{F}_{\mathcal{E}/\mathcal{A}}$-comodule $M$ with coaction $\rho_M$ to the $\mathbb{F}_{\mathcal{E}/\mathcal{A}}$-module $M$ with action $$M\otimes \mathbb{F}_{\mathcal{E}/\mathcal{A}}\xrightarrow{\rho_M\otimes\mathrm{id}}M\otimes \mathbb{F}_{\mathcal{E}/\mathcal{A}}\otimes \mathbb{F}_{\mathcal{E}/\mathcal{A}}\xrightarrow{\mathrm{id}\otimes\omega_{\mathcal{E}/\mathcal{A}}} M,$$ and it is defined in the obvious way on morphisms. Using that the pairing $\omega_{\mathcal{E}/\mathcal{A}}$ is compatible with the Hopf algebra structure, it follows easily that $\omega_{\mathcal{E}/\mathcal{A}}^{\flat}$ is a tensor functor. We obtain the following generalization of \cite[Theorem 3.3]{Shi:nondeg}. We also note that a similar result was obtained in \cite[Proposition 3.32]{Kin} under different hypotheses. Our proof follows the same outline. 

\begin{Proposition}\label{prop:factorizablevspairing}
The canonical tensor functor $$\mathcal{A}\boxtimes_{\mathcal{E}}\mathcal{A}^{\mathrm{rev}}\rightarrow \mathcal{Z}(\mathcal{A},\mathcal{E})$$ is naturally identified with the tensor functor $$\omega_{\mathcal{E}/\mathcal{A}}^{\flat}:\mathrm{Comod}_{\mathcal{A}}(\mathbb{F}_{\mathcal{E}/\mathcal{A}})\rightarrow\mathrm{Mod}_{\mathcal{A}}(\mathbb{F}_{\mathcal{E}/\mathcal{A}})$$ arising from the pairing $\omega_{\mathcal{E}/\mathcal{A}}$. In particular, it is an equivalence if and only if $\omega_{\mathcal{E}/\mathcal{A}}$ is non-degenerate.
\end{Proposition}
\begin{proof}
Let us consider the diagram:
\begin{equation}\label{eq:relativefactorizability}\begin{tikzcd}[sep=tiny]
\mathcal{A}\boxtimes\mathcal{A}^{\mathrm{rev}} \arrow[rr, two heads] \arrow[dd, "\rotatebox{90}{$\simeq$}"'] \arrow[rrrrrr, bend left = 15]                                   &  & \mathcal{A}\boxtimes_{\mathcal{E}}\mathcal{A}^{\mathrm{rev}} \arrow[rr] \arrow[dd, "\rotatebox{90}{$\simeq$}"']                                           &  & {\mathcal{Z}(\mathcal{A},\mathcal{E})} \arrow[dd, "\rotatebox{-90}{$\simeq$}"] \arrow[rr, hook]          &  & \mathcal{Z}(\mathcal{A}) \arrow[dd, "\rotatebox{-90}{$\simeq$}"]            \\
&  &  &  & &  &  \\
\mathrm{Comod}_{\mathcal{A}}(\mathbb{F}_{\mathcal{A}}) \arrow[rr, two heads] \arrow[rrrrrr, "\omega_{\mathcal{A}}^{\flat}"', bend right = 15] &  & \mathrm{Comod}_{\mathcal{A}}(\mathbb{F}_{\mathcal{E}/\mathcal{A}}) \arrow[rr, "\omega_{\mathcal{E}/\mathcal{A}}^{\flat}"] &  & \mathrm{Mod}_{\mathcal{A}}(\mathbb{F}_{\mathcal{E}/\mathcal{A}}) \arrow[rr, hook] &  & \mathrm{Mod}_{\mathcal{A}}(\mathbb{F}_{\mathcal{A}}).
\end{tikzcd}\end{equation}
We will deduce the commutativity of the middle square from the commutativity of the other squares. Firstly, that the vertical arrows are equivalences follows from Proposition \ref{prop:comonadicreconstruction} as $\mathcal{A}^{\mathrm{rev}}\simeq \mathcal{A}^{\mathrm{mop}}$. Additionally, that the left square commutes follows from the proof of Theorem \ref{thm:exactsequenceHopf}. Secondly, that the right vertical arrows are equivalences was established in Proposition \ref{prop:RelativeCenterModule}. From the proof of that result, one deduces the commutativity of the right square. Thirdly, the top square commutes by definition. Fourthly, the bottom square commutes as the map of Hopf algebras $\mathbb{F}_{\mathcal{A}}\rightarrow\mathbb{F}_{\mathcal{E}/\mathcal{A}}$ is compatible with the pairings by construction (see Proposition \ref{prop:inducedpairing}).

Crucially, the outer square of \eqref{eq:relativefactorizability} commutes as was shown in \cite[Subsection 3.5]{Shi:nondeg}. We now argue that this implies that the middle square commutes as well. To see this, observe that it follows by direct inspection that the natural isomorphism witnessing the commutativity of the outer square is $\mathcal{E}$-balanced, so that it factors via $\mathcal{A}\boxtimes_{\mathcal{E}}\mathcal{A}^{\mathrm{rev}}$. As a consequence, the composite of the middle and the right squares in \eqref{eq:relativefactorizability} does commute. But the right horizontal functors are fully faithful by definition of the relative Drinfeld center. This shows that middle square commutes as desired, and concludes the proof of the first part of the proposition.

We now conclude the proof of the proposition. As $\mathbb{F}_{\mathcal{E}/\mathcal{A}}$ is a compact object of $\mathcal{A}$, it has a dual $\mathbb{F}_{\mathcal{E}/\mathcal{A}}^*$. In particular, $\mathbb{F}_{\mathcal{E}/\mathcal{A}}^*$ inherits a Hopf algebra structure, for which there is an identification $$\mathrm{Comod}_{\mathcal{A}}(\mathbb{F}_{\mathcal{E}/\mathcal{A}})\simeq \mathrm{Mod}_{\mathcal{A}}(\mathbb{F}_{\mathcal{E}/\mathcal{A}}^*)$$ of tensor categories. Moreover, the pairing $\omega_{\mathcal{E}/\mathcal{A}}$ induces a homomorphism of Hopf algebras $f:\mathbb{F}_{\mathcal{E}/\mathcal{A}}^*\rightarrow \mathbb{F}_{\mathcal{E}/\mathcal{A}}$. The tensor functor $\omega_{\mathcal{E}/\mathcal{A}}^{\flat}$ then corresponds to restriction under the homomorphism of Hopf algebras $f$. But the functor $\omega_{\mathcal{E}/\mathcal{A}}^{\flat}$ is an equivalence if and only if $f$ is an isomorphism (see for instance \cite[Lemma 3.4]{Shi:nondeg}). The latter condition holds if and only if $\omega_{\mathcal{E}/\mathcal{A}}$ is non-degenerate. This finishes the proof.
\end{proof}

Let $\mathcal{A}$ be a finite braided tensor category. The kernel $\ker(\omega_{\mathcal{A}})$ of the pairing $\omega_{\mathcal{A}}$ on $\mathbb{F}_{\mathcal{A}}$ was studied in \cite[Section 3]{Lyu:modular}. In particular, it was shown therein that the Hopf algebra $\mathbb{F}_{\mathcal{A}}/\ker(\omega_{\mathcal{A}})$ admits a non-degenerate pairing. Using our results, we can give another description of the kernel $\ker(\omega_{\mathcal{A}})$.

\begin{Corollary}\label{cor:comparisonLyubashenko}
Let $\mathcal{A}$ be a finite braided tensor category. We have that $$\ker(\omega_{\mathcal{A}})=\mathbb{F}_{\mathcal{A}}\mathbb{F}_{\mathcal{Z}_{(2)}(\mathcal{A})}^{+}\subseteq \mathbb{F}_{\mathcal{A}},$$ so that $\mathbb{F}_{\mathcal{A}}/\ker(\omega_{\mathcal{A}})\cong \mathbb{F}_{\mathcal{E}/\mathcal{A}}$ as Hopf algebras.
\end{Corollary}
\begin{proof}
We view $\mathcal{A}$ as a faithfully flat $\mathcal{Z}_{(2)}(\mathcal{A})$-enriched finite braided tensor category. It follows from Proposition \ref{prop:factorizablevspairing} and Theorem \ref{thm:nondegeneratevsfactorizable} that the pairing on $\mathbb{F}_{\mathcal{Z}_{(2)}(\mathcal{A})/\mathcal{A}}$ is non-degenerate. But the canonical map of Hopf algebras $\mathbb{F}_{\mathcal{A}}\rightarrow \mathbb{F}_{\mathcal{Z}_{(2)}(\mathcal{A})/\mathcal{A}}$ is compatible with the pairings by Proposition \ref{prop:inducedpairing}. This concludes the proof.
\end{proof}

\begin{Remark}
It would be interesting to understand whether this characterization of the kernel of the pairing $\omega_{\mathcal{A}}$ holds without any finiteness hypothesis. It is at the very least clear that $\mathbb{F}_{\mathcal{A}}\mathbb{F}_{\mathcal{Z}_{(2)}(\mathcal{A})}^{+}\subseteq\ker(\omega_{\mathcal{A}})$.
\end{Remark}

\subsection{Relative Cofactorizability \& Non-Degenerate Pairing}

Let $\mathcal{A}$ be a finite braided tensor category. We begin by recalling from \cite{BJSS} how the canonical monoidal functor $$\mathrm{HC}(\mathcal{A})=\mathcal{A}\boxtimes_{\mathcal{A}\boxtimes\mathcal{A}^{\mathrm{rev}}}\mathcal{A}^{\mathrm{mop}}\rightarrow \mathrm{End}(\mathcal{A})$$ arises. Observe that there is a monoidal functor
$$\begin{tabular}{c c c}
$\mathcal{A}\boxtimes\mathcal{A}^{\mathrm{mop}}$ & $\rightarrow$ & $\mathrm{End}(\mathcal{A})$.\\
$A\boxtimes B$ & $\mapsto$ & $\{V\mapsto A\otimes V\otimes B\}$
\end{tabular}$$
One checks that this last functor above is $\mathcal{A}\boxtimes\mathcal{A}^{\mathrm{rev}}$-balanced, so that it induces a monoidal functor $\mathrm{HC}(\mathcal{A})\rightarrow \mathrm{End}(\mathcal{A})$. This functor first appeared in \cite{BZBJ2}. Let now $\mathcal{A}$ be a faithfully flat $\mathcal{E}$-enriched finite braided tensor category. The tensor functor $\mathcal{A}\boxtimes\mathcal{A}^{\mathrm{mop}}\rightarrow\mathrm{End}(\mathcal{A})$ can then naturally be upgraded to a monoidal functor $\mathcal{A}\boxtimes\mathcal{A}^{\mathrm{mop}}\rightarrow\mathrm{End}_{\mathcal{E}}(\mathcal{A})$. Moreover, it follows by inspection that this functor is not only $\mathcal{A}\boxtimes\mathcal{A}^{\mathrm{rev}}$-balanced, but that the two $\mathcal{E}$-balanced structure coincide, so that we obtain a monoidal functor $$\mathrm{HC}_{\mathcal{E}}(\mathcal{A})=\mathcal{A}\boxtimes_{\mathcal{A}\boxtimes_{\mathcal{E}}\mathcal{A}^{\mathrm{rev}}}\mathcal{A}^{\mathrm{mop}}\rightarrow \mathrm{End}_{\mathcal{E}}(\mathcal{A}).$$

It was shown in \cite[Proposition 3.16]{BJSS}, see also \cite[Section 4]{BZBJ2}, that there is an equivalence $$\mathrm{HC}(\mathcal{A})\simeq \mathrm{Mod}_{\mathcal{A}}(\mathbb{F}_{\mathcal{A}})$$ of categories. Succinctly, the monoidal structure on $\mathrm{HC}(\mathcal{A})$ arises by viewing $\mathbb{F}_{\mathcal{A}}$ as a commutative algebra in $\mathcal{Z}(\mathcal{A})$. Namely, thanks to \cite[Corollary 6.7]{SY:MTC}, $T^R_{\mathcal{A}}(\mathbbm{1})$ is a commutative algebra in $\mathcal{A}\boxtimes\mathcal{A}^{\mathrm{rev}}$. Using the canonical braided tensor functor $\mathcal{A}\boxtimes\mathcal{A}^{\mathrm{rev}}\rightarrow\mathcal{Z}(\mathcal{A})$, we obtain the desired commutative algebra structure on $\mathbb{F}_{\mathcal{A}}$ in $\mathcal{Z}(\mathcal{A})$. It then follows from a standard argument that this provides $\mathrm{Mod}_{\mathcal{A}}(\mathbb{F}_{\mathcal{A}})$ with a monoidal structure, which agrees with that of $\mathrm{HC}(\mathcal{A})$ thanks to the universal property of the relative Deligne tensor product. In particular, we emphasize that this monoidal structure does not come from the Hopf algebra structure on $\mathbb{F}_{\mathcal{A}}$. For the purposes of our proof, it is only necessary to observe that the monoidal unit of $\mathrm{Mod}_{\mathcal{A}}(\mathbb{F}_{\mathcal{A}})$ is given by $\mathbb{F}_{\mathcal{A}}$. Now, using the properties of the relative Deligne tensor product, we also have an equivalence of plain categories $$\mathrm{HC}_{\mathcal{E}}(\mathcal{A})\simeq \mathcal{A}\boxtimes_{\mathcal{A}\boxtimes_{\mathcal{E}}\mathcal{A}^{\mathrm{rev}}}\mathrm{Mod}_{\mathcal{A}\boxtimes_{\mathcal{E}}\mathcal{A}^{\mathrm{rev}}}(T^R_{\mathcal{E}/\mathcal{A}}(\mathbbm{1}))\simeq \mathrm{Mod}_{\mathcal{A}}(\mathbb{F}_{\mathcal{E}/\mathcal{A}}).$$ Using the obvious variant of the argument given above in the non-relative case, we find that upon viewing $\mathbb{F}_{\mathcal{E}/\mathcal{A}}$ as a commutative algebra in $\mathcal{Z}(\mathcal{A})$ and endowing $\mathrm{Mod}_{\mathcal{A}}(\mathbb{F}_{\mathcal{E}/\mathcal{A}})$ with the corresponding monoidal structure, the last equivalence above is in fact compatible with the monoidal structures. In particular, the monoidal unit is again given by $\mathbb{F}_{\mathcal{E}/\mathcal{A}}$. As the functor $\mathrm{HC}(\mathcal{A})\twoheadrightarrow\mathrm{HC}_{\mathcal{E}}(\mathcal{A})$ is monoidal and compatible with the left $\mathcal{A}$-module structures by construction, it follows from the relative Eilenberg-Watts theorem, see \cite[Theorem 3.3]{DSPS:balanced} or \cite[Lemma 5.7]{BJS}, that it can be identified with $$(-)\otimes_{\mathbb{F}_{\mathcal{A}}} \mathbb{F}_{\mathcal{E}/\mathcal{A}}:\mathrm{Mod}_{\mathcal{A}}(\mathbb{F}_{\mathcal{A}})\rightarrow\mathrm{Mod}_{\mathcal{A}}(\mathbb{F}_{\mathcal{E}/\mathcal{A}}).$$

There are also equivalences of plain categories $$\mathrm{End}(\mathcal{A})\simeq \mathcal{A}\boxtimes\mathcal{A}^{\mathrm{mop}}\simeq\mathrm{Comod}_{\mathcal{A}}(\mathbb{F}_{\mathcal{A}}).$$ The first one is for instance given by \cite[Proposition 5.3 and Corollary 5.5]{BJS}. In the non-relative case, we can use \cite[Lemma 3.5]{Shi:unimodular} to see that the equivalence above is implemented by the explicit formula \begin{equation}\label{eq:EndtoComodCoend}
F\mapsto \int^{A\in\mathcal{A}^{\mathbf{c}}} F(A)^*\otimes A.
\end{equation}
In the relative case, recall from \cite[Lemma 5.4 and Proposition 5.8]{BJS} that $$\mathrm{End}_{\mathcal{E}}(\mathcal{A})\simeq \mathrm{End}(\mathcal{A})\boxtimes_{\mathcal{E}\boxtimes\mathcal{E}}\mathcal{E}.$$ It follows that the forgetful functor $\mathrm{End}_{\mathcal{E}}(\mathcal{A})\rightarrow\mathrm{End}(\mathcal{A})$ corresponds to the functor $\mathrm{Mod}_{\mathcal{A}\boxtimes\mathcal{A}^{\mathrm{mop}}}(\mathbb{F}_{\mathcal{E}})\rightarrow\mathcal{A}\boxtimes\mathcal{A}^{\mathrm{mop}}$. In fact, we have that the category $\mathrm{End}_{\mathcal{E}}(\mathcal{A})$ can be explicit identified via
\begin{equation}\label{eq:modulefunctor}
\begin{tabular}{c c c}
$\mathrm{End}_{\mathcal{E}}(\mathcal{A})$ & $\xrightarrow{\simeq}$ & $\mathrm{Mod}_{\mathrm{Comod}_{\mathcal{A}}(\mathbb{F}_{\mathcal{A}})}(\mathbb{F}_{\mathcal{E}})$,\\
$F$ & $\mapsto$ & $\int^{A\in\mathcal{A}^{\mathbf{c}}}F(A)^*\otimes A$
\end{tabular}
\end{equation}
where the coend is equipped with the canonical right $\mathbb{F}_{\mathcal{E}}$-module structure induced by the $\mathcal{E}$-module structure on $F$. Thanks to Theorem \ref{thm:exactsequenceHopf}, we therefore have the following commutative diagram of plain categories
$$\begin{tikzcd}[sep=3ex]
\mathrm{End}_{\mathcal{E}}(\mathcal{A}) \arrow[r] \arrow[d,"\rotatebox{90}{$\simeq$}"']& \mathrm{End}(\mathcal{A}) \arrow[d,"\rotatebox{-90}{$\simeq$}"] \\
\mathrm{Mod}_{\mathrm{Comod}_{\mathcal{A}}(\mathbb{F}_{\mathcal{A}})}(\mathbb{F}_{\mathcal{E}}) \arrow[r] \arrow[d,"\rotatebox{90}{$\simeq$}"'] & \mathrm{Comod}_{\mathcal{A}}(\mathbb{F}_{\mathcal{A}}),\\
\mathrm{Comod}_{\mathcal{A}}(\mathbb{F}_{\mathcal{E}/\mathcal{A}}) \arrow[ru, "(-)\Box^{\pi}\mathbb{F}_{\mathcal{A}}"'{xshift=-2pt, yshift=-2pt}] & 
\end{tikzcd}$$ in which the two horizontal arrows are forgetful functors.

The canonical functor $\mathrm{HC}(\mathcal{A})\rightarrow \mathrm{End}(\mathcal{A})$ was identified with a functor $\mathrm{Mod}_{\mathcal{A}}(\mathbb{F}_{\mathcal{A}})\rightarrow \mathrm{Comod}_{\mathcal{A}}(\mathbb{F}_{\mathcal{A}})$ in \cite[Section 3.3]{BJSS}. This can be done effectively using the following result.

\begin{Lemma}[\cite{Kin}]\label{lem:pairing}
Let $A$ be an algebra in $\mathcal{A}$ and $C$ be a coalgebra in $\mathcal{A}$. Pairings between $A$ and $C$ are in bijective correspondence with isomorphism classes of left $\mathcal{A}$-module functors $\mathrm{Mod}_{\mathcal{A}}(A)\rightarrow \mathrm{Comod}_{\mathcal{A}}(C)$ sending free $A$-modules to cofree $C$-modules.
\end{Lemma}
\begin{proof}
Even though are hypotheses are more general, the proof given in \cite[Lemma A.8]{Kin} carries over verbatim. We outline the argument for the reader's convenience. A functor $\mathrm{Mod}_{\mathcal{A}}(A)\rightarrow \mathrm{Comod}_{\mathcal{A}}(C)$ as in the statement of the lemma is completely determined by its action on the multiplication map $m:A\otimes A\rightarrow A$. Namely, such a functor is not only compatible with the left $\mathcal{A}$-module structures, but is also cocontinuous because all of our constructions are carried out within $\mathbf{Pr}$. In particular, this functor is completely determined by the image of $m$, which is a right $C$-comodule map $:A\otimes C\rightarrow C$, or, equivalently, the data of a pairing $A\otimes C\rightarrow\mathbbm{1}$.
\end{proof}

\noindent A variant of the above observation was used in \cite[Proposition 3.19]{BJSS} to show that the canonical functor $\mathrm{HC}(\mathcal{A})\rightarrow\mathrm{End}(\mathcal{A})$ can be identified with the functor $\omega_{\mathcal{A}}^{\sharp}:\mathrm{Mod}_{\mathcal{A}}(\mathbb{F}_{\mathcal{A}})\rightarrow\mathrm{Comod}_{\mathcal{A}}(\mathbb{F}_{\mathcal{A}})$ supplied by the pairing $\omega_{\mathcal{A}}$. We will write $\omega_{\mathcal{E}/\mathcal{A}}^{\sharp}:\mathrm{Mod}_{\mathcal{A}}(\mathbb{F}_{\mathcal{E}/\mathcal{A}})\rightarrow\mathrm{Comod}_{\mathcal{A}}(\mathbb{F}_{\mathcal{E}/\mathcal{A}})$ for the functor obtained by applying Lemma \ref{lem:pairing} to the pairing $\omega_{\mathcal{E}/\mathcal{A}}$.

Our next result generalizes \cite[Proposition 3.19]{BJSS}, although our proof relies crucially on this result. We wish to point out that a similar result was obtained in \cite[Proposition 3.39]{Kin} under different hypotheses.

\begin{Proposition}\label{prop:cofactorizablenondegeneratepairing}
Let $\mathcal{A}$ be a faithfully flat $\mathcal{E}$-enriched finite braided tensor category. The canonical functor $$\mathrm{HC}_{\mathcal{E}}(\mathcal{A})\rightarrow\mathrm{End}_{\mathcal{E}}(\mathcal{A})$$ is naturally identified with the functor $$\omega_{\mathcal{E}/\mathcal{A}}^{\sharp}:\mathrm{Mod}_{\mathcal{A}}(\mathbb{F}_{\mathcal{E}/\mathcal{A}})\rightarrow\mathrm{Comod}_{\mathcal{A}}(\mathbb{F}_{\mathcal{E}/\mathcal{A}})$$ arising from the Hopf pairing $\omega_{\mathcal{E}/\mathcal{A}}$. In particular, it is an equivalence if and only if $\omega_{\mathcal{E}/\mathcal{A}}$ is non-degenerate.
\end{Proposition}
\begin{proof}
We consider the diagram:
\begin{equation}\label{eq:relativecofactorizability}\begin{tikzcd}[sep= 0.88em]
\mathrm{HC}(\mathcal{A})\arrow[rr, two heads] \arrow[dd, "\rotatebox{90}{$\simeq$}"'] \arrow[rrrrrr, bend left = 15]                                   &  & \mathrm{HC}_{\mathcal{E}}(\mathcal{A}) \arrow[rr] \arrow[dd, "\rotatebox{90}{$\simeq$}"']                                           &  & {\mathrm{End}_{\mathcal{E}}(\mathcal{A})} \arrow[dd, "\rotatebox{-90}{$\simeq$}"] \arrow[rr, hook]          &  & {\mathrm{End}(\mathcal{A})} \arrow[dd, "\rotatebox{-90}{$\simeq$}"]            \\
&  &  &  & &  &  \\
\mathrm{Mod}_{\mathcal{A}}(\mathbb{F}_{\mathcal{A}}) \arrow[rr, two heads, "(-)\otimes_{\mathbb{F}_{\mathcal{A}}} \mathbb{F}_{\mathcal{E}/\mathcal{A}}" {yshift=4pt}] \arrow[rrrrrr, "\omega_{\mathcal{A}}^{\sharp}"', bend right = 15] &  & \mathrm{Mod}_{\mathcal{A}}(\mathbb{F}_{\mathcal{E}/\mathcal{A}}) \arrow[rr, "\omega_{\mathcal{E}/\mathcal{A}}^{\sharp}"] &  & \mathrm{Comod}_{\mathcal{A}}(\mathbb{F}_{\mathcal{E}/\mathcal{A}}) \arrow[rr, hook, "(-)\Box^{\pi} \mathbb{F}_{\mathcal{A}}" {yshift=4pt}] &  & \mathrm{Comod}_{\mathcal{A}}(\mathbb{F}_{\mathcal{A}}).
\end{tikzcd}\end{equation}
The vertical equivalences as well as the commutativity of the left and right squares were described above. Commutativity of the top and bottom squares holds by construction and direct inspection respectively. The commutativity of the outer square was derived in the proof of \cite[Proposition 3.19]{BJSS}.

We now observe that every free $\mathbb{F}_{\mathcal{E}/\mathcal{A}}$-module is the image under the functor $(-)\otimes_{\mathbb{F}_{\mathcal{A}}} \mathbb{F}_{\mathcal{E}/\mathcal{A}}$ of a free $\mathbb{F}_{\mathcal{A}}$-module. It was also shown in \cite[Section 3.3]{BJSS} that the left $\mathcal{A}$-module functor $\mathrm{HC}(\mathcal{A})\twoheadrightarrow\mathrm{End}(\mathcal{A})$ sends free $\mathbb{F}_{\mathcal{A}}$-modules to cofree $\mathbb{F}_{\mathcal{A}}$-comodules. More precisely, the free $\mathbb{F}_{\mathcal{A}}$-module $A\otimes\mathbb{F}_{\mathcal{A}}$ is sent to the functor $A\otimes (-):\mathcal{A}\rightarrow \mathcal{A}$, which is identified with the cofree $\mathbb{F}_{\mathcal{A}}$-comodule $A\otimes\mathbb{F}_{\mathcal{A}}$ under Equation \eqref{eq:EndtoComodCoend}. As a consequence, we find that the free $\mathbb{F}_{\mathcal{A}}$-module $A\otimes\mathbb{F}_{\mathcal{A}}$ is sent under $\mathrm{HC}(\mathcal{A})\twoheadrightarrow\mathrm{End}_{\mathcal{E}}(\mathcal{A})$ to the $\mathcal{E}$-module functor $A\otimes (-):\mathcal{A}\rightarrow \mathcal{A}$. Under the equivalence of Equation \eqref{eq:modulefunctor}, we find that the $\mathcal{E}$-module functor $A\otimes (-):\mathcal{A}\rightarrow \mathcal{A}$ corresponds to the cofree $\mathbb{F}_{\mathcal{A}}$-comodule $A\otimes\mathbb{F}_{\mathcal{A}}$ with its canonical $\mathbb{F}_{\mathcal{E}}$-module structure. But, the latter corresponds to the cofree $\mathbb{F}_{\mathcal{E}/\mathcal{A}}$-comodule $A\otimes\mathbb{F}_{\mathcal{E}/\mathcal{A}}$, so that we have established that the left $\mathcal{A}$-module functor $\mathrm{HC}(\mathcal{A})\twoheadrightarrow\mathrm{End}_{\mathcal{E}}(\mathcal{A})$ sends free $\mathbb{F}_{\mathcal{A}}$-modules to cofree $\mathbb{F}_{\mathcal{E}/\mathcal{A}}$-modules. As every free $\mathbb{F}_{\mathcal{E}/\mathcal{A}}$-module is the image of a free $\mathbb{F}_{\mathcal{A}}$-module, it follows immediately that the left $\mathcal{A}$-module functor $\mathrm{HC}_{\mathcal{E}}(\mathcal{A})\twoheadrightarrow\mathrm{End}_{\mathcal{E}}(\mathcal{A})$ sends free $\mathbb{F}_{\mathcal{E}/\mathcal{A}}$-modules to cofree $\mathbb{F}_{\mathcal{E}/\mathcal{A}}$-modules. Finally, the bottom square of the diagram of Equation \eqref{eq:relativecofactorizability} commutes, so we find that $\mathrm{HC}_{\mathcal{E}}(\mathcal{A})\twoheadrightarrow\mathrm{End}_{\mathcal{E}}(\mathcal{A})$ can be identified with $\omega_{\mathcal{E}/\mathcal{A}}^{\sharp}$ by appealing to Lemma \ref{lem:pairing}.

We now conclude the proof of the proposition. Let us write $\mathbb{F}_{\mathcal{E}/\mathcal{A}}^*$ for the dual of $\mathbb{F}_{\mathcal{E}/\mathcal{A}}$. The coalgebra structure on $\mathbb{F}_{\mathcal{E}/\mathcal{A}}$ induces an algebra structure on $\mathbb{F}_{\mathcal{E}/\mathcal{A}}^*$. We may therefore consider the following equivalence of plain categories $$\mathrm{Comod}_{\mathcal{A}}(\mathbb{F}_{\mathcal{E}/\mathcal{A}})\simeq \mathrm{Mod}_{\mathcal{A}}(\mathbb{F}_{\mathcal{E}/\mathcal{A}}^*).$$ The pairing $\omega_{\mathcal{E}/\mathcal{A}}$ induces a map of algebras $f:\mathbb{F}_{\mathcal{E}/\mathcal{A}}\rightarrow \mathbb{F}_{\mathcal{E}/\mathcal{A}}^*$. Under the above identification the functor $\omega_{\mathcal{E}/\mathcal{A}}^{\sharp}$ corresponds to the functor $$(-)\otimes_{\mathbb{F}_{\mathcal{E}/\mathcal{A}}}\mathbb{F}_{\mathcal{E}/\mathcal{A}}^*:\mathrm{Mod}_{\mathcal{A}}(\mathbb{F}_{\mathcal{E}/\mathcal{A}}^*)\rightarrow\mathrm{Mod}_{\mathcal{A}}(\mathbb{F}_{\mathcal{E}/\mathcal{A}})$$ induced by $f$. In particular, its right adjoint is the functor of restriction under $f$. We therefore find that $(-)\otimes_{\mathbb{F}_{\mathcal{E}/\mathcal{A}}}\mathbb{F}_{\mathcal{E}/\mathcal{A}}^*$ is an equivalence if and only if its right adjoint is an equivalence. Thanks to \cite[Lemma 3.4]{Shi:nondeg}, this second condition holds if and only if $f$ is an isomorphism, which itself holds if and only if $\omega_{\mathcal{E}/\mathcal{A}}$ is non-degenerate.
\end{proof}

It remains to show that if we only assume that $F:\mathcal{E}\rightarrow\mathcal{Z}_{(2)}(\mathcal{A})$ is faithful, then $\mathcal{E}$-cofactorizability implies that $F$ is fully faithful.

\begin{Lemma}
Let $\mathcal{A}$ be an $\mathcal{E}$-cofactorizable $\mathcal{E}$-enriched finite braided tensor category such that that $F:\mathcal{E}\rightarrow\mathcal{Z}_{(2)}(\mathcal{A})$ is faithful. Then, we have that $\mathcal{A}$ is faithfully flat, i.e.\ $F$ is fully faithful.
\end{Lemma}
\begin{proof}
Let us write $\mathcal{F}$ for the finite symmetric tensor category given by the image of $F:\mathcal{E}\rightarrow\mathcal{Z}_{(2)}(\mathcal{A})$. Let us now consider the following commutative diagram: $$\begin{tikzcd}[sep=3ex]
\mathrm{HC}_{\mathcal{E}}(\mathcal{A}) \arrow[d] \arrow[r, "\simeq"] & {\mathrm{End}_{\mathcal{E}}(\mathcal{A})}                    \\
\mathrm{HC}_{\mathcal{F}}(\mathcal{A}) \arrow[r]                              & {\mathrm{End}_{\mathcal{F}}(\mathcal{A})}. \arrow[u]
\end{tikzcd}$$ 
The top horizontal arrow is an equivalence by hypothesis. In particular, this implies that the functor $\mathrm{End}_{\mathcal{F}}(\mathcal{A})\rightarrow \mathrm{End}_{\mathcal{E}}(\mathcal{A})$ is essentially surjective and full. But this functor is also faithful by inspection. Namely, a module structure on a natural transformation is a property and not additional data. It follows that $\mathrm{End}_{\mathcal{F}}(\mathcal{A})\rightarrow \mathrm{End}_{\mathcal{E}}(\mathcal{A})$ is an equivalence. This shows that the functor $\mathcal{A}\boxtimes_{\mathcal{F}}\mathcal{A}^{\mathrm{mop}}\rightarrow\mathcal{A}\boxtimes_{\mathcal{E}}\mathcal{A}^{\mathrm{mop}}$ is an equivalence. But this is the right adjoint to the canonical functor $$\mathcal{A}\boxtimes_{\mathcal{E}}\mathcal{A}^{\mathrm{mop}}\rightarrow\mathcal{A}\boxtimes_{\mathcal{F}}\mathcal{A}^{\mathrm{mop}},$$ induced by the universal property of the relative Deligne tensor product, so that this latter functor is also an equivalence. Mimicking the second half of the proof of Proposition \ref{prop:faithfulfactorizablefaithfullyflat} yields the desired result.
\end{proof}

\subsection{Relative Invertibility}

Combining Theorem \ref{thm:relativeinvertibilityBJSS} with Theorem \ref{thm:nondegeneratevsfactorizable} and Propositions \ref{prop:factorizablevspairing} and \ref{prop:cofactorizablenondegeneratepairing}, we obtain the following result generalizing \cite{BJSS}.

\begin{Theorem}\label{thm:relativeinvertibility}
Let $\mathbbm{k}$ be algebraically closed, and let $\mathcal{E}$ be a finite symmetric tensor category over $\mathbbm{k}$. An $\mathcal{E}$-enriched finite braided tensor category $\mathcal{A}$ is invertible as an object of the Morita 4-category $\mathrm{Mor}^{\mathrm{pre}}_2(\mathbf{Pr}_{\mathcal{E}})$ if and only if $F:\mathcal{E}\rightarrow\mathcal{Z}_{(2)}(\mathcal{A})$ is an equivalence.
\end{Theorem}

\noindent Heuristically, the above theorem shows that any finite braided tensor category is invertible relative to its symmetric center. Let us also mention that, in order to remove the hypothesis on the base field $\mathbbm{k}$, it would suffice to formalize the argument sketched out in Remark \ref{rem:nondegeneracyarbitraryfield}.

\begin{Example}\label{ex:smallquantumgroup}
Set $\mathbbm{k} = \mathbb{C}$ to be the field of complex numbers. Let also $G$ be a simple algebraic group over $\mathbb{C}$, and $q$ be a root of unity. We write $\mathrm{Rep}(G_q)$ for the braided tensor category of representations of Lusztig's divided power quantum group for $G$ at root of unity $q$. Its symmetric center was computed in \cite[Section 9]{Neg:arbitrary}. It follows from \cite{Del:fiber} that there is an essentially unique symmetric tensor functor $F:\mathcal{Z}_{(2)}(\mathrm{Rep}(G_q))\rightarrow\mathrm{sVec}$, to the category of super vector spaces. As shown in \cite[Section 10]{Neg:arbitrary}, the image of this functor can be either $\mathrm{Vec}$ or $\mathrm{sVec}$. We write $\mathrm{Tann}_q\subseteq \mathcal{Z}_{(2)}(\mathrm{Rep}(G_q))$ for the symmetric tensor category that is the preimage of $\mathrm{Vec}\subset\mathrm{sVec}$ under $F$. It was established in \cite[Theorem 13.1]{Neg:arbitrary} that $$\mathrm{Rep}(G_q)\boxtimes_{\mathrm{Tann}_q}\!\!\mathrm{Vec}$$ is a finite braided tensor category. By construction, its symmetric center is either $\mathrm{Vec}$ if $\mathrm{Tann}_q = \mathcal{Z}_{(2)}(\mathrm{Rep}(G_q))$ or $\mathrm{sVec}$ otherwise. In the former case, we obtain invertible objects of $\mathrm{Mor}^{\mathrm{pre}}_2(\mathbf{Pr})$ via \cite{BJSS,Shi:nondeg}. In the latter, we get invertible objects of $\mathrm{Mor}^{\mathrm{pre}}_2(\mathbf{Pr}_{\mathrm{sVec}})$ thanks to Theorem \ref{thm:relativeinvertibility}.
\end{Example}

\begin{Remark}
Under the hypotheses of the previous example, it was shown in \cite[Theorem 4.3]{Kin} that when $\mathrm{Tann}_q=\mathcal{Z}_{(2)}(\mathrm{Rep}(G_q))$, that is, the symmetric center of $\mathrm{Rep}(G_q)$ is Tannakian, then $\mathrm{Rep}(G_q)$ is an invertible object of the Morita 4-category of $\mathcal{Z}_{(2)}(\mathrm{Rep}(G_q))$-enriched braided pre-tensor categories. The proof makes heavy use of the fact that the finite braided tensor category $\mathrm{Rep}(G_q)\boxtimes_{\mathrm{Tann}_q}\!\!\mathrm{Vec}$ is non-degenerate, and therefore invertible in the Morita 4-category of braided pre-tensor categories by \cite{BJSS}. Thanks to Theorem \ref{thm:relativeinvertibility}, we expect that the assumption that $\mathrm{Tann}_q=\mathcal{Z}_{(2)}(\mathrm{Rep}(G_q))$ can be removed, so that the result of \cite{Kin} holds without any restriction on the pair of $G$ and $q$.
\end{Remark}

\begin{Example}\label{ex:mixedVerlinde}
Let $\mathbbm{k}$ be an algebraically closed field of characteristic $p>0$. Let also $n$ be a positive integer, and $\zeta$ a root of unity in $\mathbbm{k}$. The mixed Verlinde category $\mathrm{Ver}_{p^{(n)}}^{\zeta}$, introduced in \cite{STWZ}, are finite braided tensor categories obtained as the abelian envelope of quotients of the category of tilting modules for $\mathrm{SL}_2$ at root of unity $\zeta$. With $\zeta = + 1$, these are precisely the symmetric Verlinde categories $\mathrm{Ver}_{p^{n}}$ defined in \cite{BEO,C:monoidal}. For any root of unity $\zeta$, the symmetric center of $\mathrm{Ver}_{p^{(n)}}^{\zeta}$ was identified in \cite{D:Verlinde}. For instance, if $\zeta$ has odd order, the symmetric center of $\mathrm{Ver}_{p^{(n)}}^{\zeta}$ is $\mathrm{Ver}_{p^{n-1}}$. For such roots of unity, it therefore follows from Theorem \ref{thm:relativeinvertibility} that $\mathrm{Ver}_{p^{(n)}}^{\zeta}$ yields an invertible object in the Morita 4-category of $\mathrm{Ver}_{p^{n-1}}$-enriched braided pre-tensor categories.
\end{Example}

\section{Applications}

We discuss some applications of Theorem \ref{thm:relativeinvertibility}. We assume throughout that our base field $\mathbbm{k}$ is algebraically closed, although we suspect that all our result hold over any perfect field. We begin by giving a higher categorical construction of the relative Witt groups of separable braided tensor categories introduced in \cite{DNO}. We then consider non-separable versions of the Witt group and its relative variants. Finally, we generalize the full dualizability result of \cite{BJS}.

\subsection{Witt Groups of Separable Braided Tensor Categories}

Let also $\mathcal{E}$ be a separable symmetric tensor category over an algebraically closed field. We may therefore consider the symmetric monoidal sub-4-category $\mathrm{Mor}^{\mathrm{sep}}_2(\mathbf{Pr}_{\mathcal{E}})$ of $\mathrm{Mor}^{\mathrm{pre}}_2(\mathbf{Pr}_{\mathcal{E}})$. Its Picard group $\pi_0(\mathrm{Mor}^{\mathrm{sep}}_2(\mathbf{Pr}_{\mathcal{E}})^{\times})$, that is, the group of invertible objects of $\mathrm{Mor}^{\mathrm{sep}}_2(\mathbf{Pr}_{\mathcal{E}})$ up to equivalence, can be identified with a well-known group in the theory of separable braided tensor categories.

In the case $\mathcal{E}=\mathrm{Vec}$, it was shown in \cite{BJSS} that the group $\pi_0(\mathrm{Mor}^{\mathrm{sep}}_2(\mathbf{Pr})^{\times})$ coincides with the (quantum) Witt group $\mathcal{W}itt^{\mathrm{sep}}(\mathrm{Vec})$ introduced in \cite{DMNO}. In essence, the group $\mathcal{W}itt^{\mathrm{sep}}(\mathrm{Vec})$ is the quotient of the monoid of non-degenerate separable braided tensor categories by the submonoid consisting of the Drinfeld centers of separable tensor categories. This abelian group is of great interest as it offers an approachable alternative to the essentially inaccessible problem of classifying all non-degenerate separable braided tensor categories.

Now, if $\mathcal{E}$ is a separable symmetric tensor category, a relative (quantum) Witt group $\mathcal{W}itt^{\mathrm{sep}}(\mathcal{E})$ of $\mathcal{E}$-non-degenerate separable braided tensor categories was introduced in \cite{DNO}. Succinctly, the group $\mathcal{W}itt^{\mathrm{sep}}(\mathcal{E})$ is defined as the quotient of the monoid $\mathcal{E}$-non-degenerate separable braided tensor categories by the submonoid of relative Drinfeld centers of faithfully flat $\mathcal{E}$-enriched separable tensor categories. Generalizing \cite[Theorem 4.2]{BJSS}, we obtain the following result.

\begin{Proposition}
Let $\mathbbm{k}$ be an algebraically closed field, and let $\mathcal{E}$ be a separable symmetric tensor category. Then, the relative Witt group $\mathcal{W}itt^{\mathrm{sep}}(\mathcal{E})$ is isomorphic to $\pi_0(\mathrm{Mor}^{\mathrm{sep}}_2(\mathbf{Pr}_{\mathcal{E}})^{\times})$, the Picard group of $\mathrm{Mor}^{\mathrm{sep}}_2(\mathbf{Pr}_{\mathcal{E}})$.
\end{Proposition}
\begin{proof}
It follows from Theorem \ref{thm:relativeinvertibility} that the invertible objects of $\mathrm{Mor}^{\mathrm{sep}}_2(\mathbf{Pr}_{\mathcal{E}})$ are precisely faithfully flat $\mathcal{E}$-enriched separable braided tensor categories. It is therefore enough to show that given a faithfully flat $\mathcal{E}$-enriched separable braided tensor category $\mathcal{A}$ and an $\mathcal{E}$-enriched $\mathcal{A}$-central separable tensor category $\mathcal{C}$, the corresponding 1-morphism $\mathcal{C}:\mathcal{E}\nrightarrow \mathcal{A}$ is an equivalence if and only if the canonical braided tensor functor $\mathcal{A}\rightarrow \mathcal{Z}(\mathcal{C},\mathcal{E})$ is an equivalence in $\mathrm{Mor}^{\mathrm{sep}}_2(\mathbf{Pr}_{\mathcal{E}})$. This is \cite[Theorem 2.51]{DHJFNPPRY}. More precisely, that result is stated in characteristic zero, but the proof carries over to positive characteristic under our additional separability assumption.
\end{proof}

\subsection{Witt Groups of Finite Braided Tensor Categories}

Let us write $\mathrm{Mor}_2^{\mathrm{fin}}(\mathbf{Pr})$ for the symmetric monoidal sub-4-category of $\mathrm{Mor}_2^{\mathrm{pre}}(\mathbf{Pr})$ whose objects are finite braided tensor categories, 1-morphisms are finite central pre-tensor categories, 2-morphisms are finite centered bimodule categories, and 3- and 4-morphisms are unchanged. The group $\pi_0(\mathrm{Mor}_2^{\mathrm{fin}}(\mathbf{Pr})^{\times})$ is a finite but non-separable version of the Witt group of non-degenerate separable braided tensor categories $\mathcal{W}itt^{\mathrm{sep}}(\mathrm{Vec})\cong \pi_0(\mathrm{Mor}_2^{\mathrm{sep}}(\mathbf{Pr})^{\times})$. Another such group $\mathcal{W}itt^{\mathrm{fin}}(\mathrm{Vec})$ was introduced in \cite[Section 7]{SY:MTC}. More precisely, they define the group $\mathcal{W}itt^{\mathrm{fin}}(\mathrm{Vec})$ as the quotient of the monoid of non-degenerate finite braided tensor categories by the submonoid of Drinfeld centers. We expect that these groups are isomorphic although we can only prove that there is a surjective group homomorphism between them. In order to do so, we provide an affirmative answer to \cite[Question 4.7]{BJSS}.

\begin{Proposition}
Let $\mathcal{C}$ be a finite tensor category with simple monoidal unit. Then, the class of the non-degenerate finite braided tensor category $\mathcal{Z}(\mathcal{C})$ viewed as an object of $\mathrm{Mor}^{\mathrm{fin}}_2(\mathbf{Pr})$ is trivial.
\end{Proposition}
\begin{proof}
Let us write $\mathcal{A}:=\mathcal{Z}(\mathcal{C})$. We show that the 1-morphism $\mathcal{C}:\mathcal{A}\nrightarrow\mathrm{Vec}$ is invertible. Applying \cite[Theorem 2.26]{BJSS}, it follows that we must show that the following four conditions are satisfied:
\begin{enumerate}
    \item The canonical tensor functor $\mathcal{C}\boxtimes\mathcal{C}^{\mathrm{mop}}\rightarrow \mathrm{End}_{\mathcal{A}}(\mathcal{C})$ is an equivalence.
    \item The canonical braided tensor functor $\mathcal{A}\rightarrow\mathcal{Z}(\mathcal{C})$ is an equivalence.
    \item The inclusion of the unit $\mathrm{Vec}\rightarrow \mathrm{End}_{\mathcal{C}\boxtimes_{\mathcal{A}}\mathcal{C}^{\mathrm{mop}}}(\mathcal{C})$ is an equivalence.
    \item The canonical monoidal functor $\mathcal{C}\boxtimes_{\mathcal{A}}\mathcal{C}^{\mathrm{mop}}\rightarrow \mathrm{End}(\mathcal{C})$ is an equivalence.
\end{enumerate}
The second condition holds by assumption. To show that the third condition holds, we may appeal to Lemma \ref{lem:relativecenter} to rewrite $\mathrm{End}_{\mathcal{C}\boxtimes_{\mathcal{A}}\mathcal{C}^{\mathrm{mop}}}(\mathcal{C})$ as the symmetric center of $\mathcal{A}$, which is trivial by \cite[Proposition 8.6.3]{EGNO} and \cite[Theorem 1.1]{Shi:nondeg}. Furthermore, the first condition follows from the second via \cite[Theorem 7.12.11]{EGNO}. It therefore only remains to argue that the fourth condition is satisfied.

Let us consider the following diagram of plain categories, whose horizontal arrows are the canonical functors appearing in conditions 1 and 4 above:
$$\begin{tikzcd}[sep=small]
\mathcal{C}\boxtimes\mathcal{C}^{\mathrm{mop}} \arrow[r] & \mathrm{End}_{\mathcal{A}}(\mathcal{C}) \arrow[d, "\rotatebox{-90}{$\simeq$}"]                      \\
\mathrm{End}(\mathcal{C}) \arrow[u, "\rotatebox{90}{$\simeq$}"]                      & \mathcal{C}\boxtimes_{\mathcal{A}}\mathcal{C}^{\mathrm{mop}}. \arrow[l]
\end{tikzcd}$$
The vertical equivalences are supplied for instance by \cite[Proposition 5.3 and Corollary
5.5]{BJS}. It follows by direct inspection that the two horizontal functors are adjoints. Given that the top horizontal arrow is an equivalence, so is the bottom one, whence concluding the proof.
\end{proof}

\begin{Corollary}\label{cor:Wittquotient}
There is a well-defined surjective group homomorphism $$\mathcal{W}itt^{\mathrm{fin}}(\mathrm{Vec})\twoheadrightarrow \pi_0(\mathrm{Mor}^{\mathrm{fin}}_2(\mathbf{Pr})^{\times}).$$
\end{Corollary}

\begin{Remark}
There are group homomorphisms
$$\mathcal{W}itt^{\mathrm{sep}}(\mathrm{Vec})\hookrightarrow\mathcal{W}itt^{\mathrm{fin}}(\mathrm{Vec})\twoheadrightarrow \pi_0(\mathrm{Mor}_2^{\mathrm{fin}}(\mathbf{Pr})^{\times})\rightarrow\pi_0(\mathrm{Mor}_2^{\mathrm{pre}}(\mathbf{Pr})^{\times}).$$
The left homomorphism can easily be seen to be injective. As already noted in \cite[Question 4.5]{BJSS} and \cite[Question 7.11]{SY:MTC}, it is very natural to wonder whether this morphism is also surjective. That the middle homomorphism is surjective follows from the definitions. The potential failure of this group homomorphism to be injective comes from the fact that, a priori, there may exist non-degenerate finite braided tensor categories that can be expressed as the Drinfeld center of a finite pre-tensor category, but not of a finite tensor category. This is exactly the content of \cite[Question 4.4]{BJSS}. The right homomorphism is perhaps the most mysterious. It is unclear whether this map is either injective or surjective, see \cite[Question 4.6]{BJSS}.
\end{Remark}

Let now $\mathcal{E}$ be a finite symmetric tensor category. We consider $\mathrm{Mor}_2^{\mathrm{fin}}(\mathbf{Pr}_{\mathcal{E}})$, the symmetric monoidal sub-4-category of $\mathrm{Mor}_2^{\mathrm{pre}}(\mathbf{Pr}_{\mathcal{E}})$ whose objects are faithfully flat $\mathcal{E}$-enriched finite braided tensor categories, 1-morphisms are $\mathcal{E}$-enriched finite central pre-tensor categories, 2-morphisms are $\mathcal{E}$-enriched finite centered bimodule categories, and 3- and 4-morphisms are unchanged. That this indeed a symmetric monoidal subcategory follows from the argument used in the proof of Proposition \ref{prop:relativefinite}. We propose the Picard group $\pi_0(\mathrm{Mor}_2^{\mathrm{fin}}(\mathbf{Pr}_{\mathcal{E}})^{\times})$ as a finite but non-separable generalization of the relative Witt groups of \cite{DNO}. In particular, we view the study of this group as an alternative to the virtually inaccessible problem of classifying all finite braided tensor categories.

\begin{Conjecture}
The group $\pi_0(\mathrm{Mor}_2^{\mathrm{fin}}(\mathbf{Pr}_{\mathcal{E}})^{\times})$ is isomorphic to the quotient of the monoid of $\mathcal{E}$-non-degenerate finite braided tensor categories by the submonoid of relative Drinfeld centers of faithfully flat $\mathcal{E}$-enriched finite tensor categories.
\end{Conjecture}

\noindent More precisely, Corollary \ref{cor:Wittquotient} admits an $\mathcal{E}$-enriched version, so that the difficulty is once again proving that this homomorphism is injective.

\subsection{Full Dualizability of Finite Braided Tensor Categories}

We obtain the following simultaneous generalization of \cite[Theorem 1.10]{BJS} and \cite[Theorem 1.1]{BJSS}.

\begin{Theorem}\label{thm:separablesymmetriccenter4dualizable}
Over an algebraically closed field, every finite braided tensor category whose symmetric center is separable is a fully dualizable object of the Morita 4-category $\mathrm{Mor}^{\mathrm{pre}}_2(\mathbf{Pr})$.
\end{Theorem}
\begin{proof}
We have already recalled above in Proposition \ref{prop:3dualizability} that finite braided pre-tensor categories are 3-dualizable objects of the Morita 4-category $\mathrm{Mor}^{\mathrm{pre}}_2(\mathbf{Pr})$. That they are 2-dualizable is a general property of objects in a higher Morita category of $E_2$-algebras \cite{GS}. More precisely, the dual to the finite braided pre-tensor category $\mathcal{A}$ is $\mathcal{A}^{\mathrm{rev}}$ with evaluation and coevaluation 1-morphisms being supplied by the canonical $\mathcal{A}\boxtimes\mathcal{A}^{\mathrm{rev}}$-central finite pre-tensor category $\mathcal{A}$, i.e.\ $\mathcal{A}$ equipped with the braided tensor functor $\mathcal{A}\boxtimes\mathcal{A}^{\mathrm{rev}}\rightarrow\mathcal{Z}(\mathcal{A})$. We will show that if $\mathcal{A}$ is a finite braided tensor category whose symmetric center is separable, then these evaluation and coevaluation 1-morphisms have all possible adjoints. This proves that such a braided tensor category is a fully dualizable object of $\mathrm{Mor}^{\mathrm{pre}}_2(\mathbf{Pr})$.

Without loss of generality, we will focus on the coevaluation 1-morphism, the evaluation 1-morphism can be dealt with entirely analogously. The coevaluation 1-morphism is $\mathcal{A}:\mathcal{A}\boxtimes\mathcal{A}^{\mathrm{rev}}\nrightarrow \mathrm{Vec}$ corresponding to the canonical braided tensor functor $\mathcal{A}\boxtimes\mathcal{A}^{\mathrm{rev}}\rightarrow\mathcal{Z}(\mathcal{A})$. Let us write $\mathcal{E} := \mathcal{Z}_{(2)}(\mathcal{A})$ for the symmetric center of $\mathcal{A}$. The above coevaluation 1-morphism can be factored as follows: $$\mathcal{A}\boxtimes\mathcal{A}^{\mathrm{rev}}\nrightarrow\mathcal{A}\boxtimes_{\mathcal{E}}\mathcal{A}^{\mathrm{rev}}\nrightarrow\mathcal{E}\nrightarrow\mathrm{Vec}.$$ It is the composite of the 1-morphisms $\mathcal{A}\boxtimes_{\mathcal{E}}\mathcal{A}^{\mathrm{rev}}:\mathcal{A}\boxtimes\mathcal{A}^{\mathrm{rev}}\nrightarrow\mathcal{A}\boxtimes_{\mathcal{E}}\mathcal{A}^{\mathrm{rev}}$, $\mathcal{A}:\mathcal{A}\boxtimes_{\mathcal{E}}\mathcal{A}^{\mathrm{rev}}\nrightarrow\mathcal{E}$, and $\mathcal{E}:\mathcal{E}\nrightarrow\mathrm{Vec}$. It will be enough to show that all three of these 1-morphisms have all adjoints. For the 1-morphism $\mathcal{E}:\mathcal{E}\nrightarrow\mathrm{Vec}$, this follows immediately from \cite[Theorem 5.16]{BJS} as $\mathcal{E}$ is separable by assumption. For use below, let us also recall that separability of $\mathcal{E}$ is equivalent to separability of the algebra $T^R_{\mathcal{E}}$ in $\mathcal{E}\boxtimes\mathcal{E}$ thanks to \cite[Corollary 2.5.9]{DSPS:dualizable}.

We now check that the 1-morphism $\mathcal{A}:\mathcal{A}\boxtimes_{\mathcal{E}}\mathcal{A}^{\mathrm{rev}}\nrightarrow\mathcal{E}$ has all possible adjoints. To see this, note that it follows from \cite[Lemma 2.22]{Kin}, which builds on \cite{Hau:remarks}, that restriction along the inclusion $\mathrm{Vec}\hookrightarrow\mathcal{E}$ defines a (non-monoidal) functor $$\mathrm{Mor}_2(\mathbf{Pr}_{\mathcal{E}})\rightarrow\mathrm{Mor}_2(\mathbf{Pr})$$ between $4$-categories. But, the 1-morphism $\mathcal{A}\boxtimes_{\mathcal{E}}\mathcal{A}^{\mathrm{rev}}\nrightarrow\mathcal{E}$ is in the image of this functor as it corresponds to the coevaluation 1-morphisms for $\mathcal{A}$ viewed as an $\mathcal{E}$-enriched finite braided tensor category, that is, as an object of $\mathrm{Mor}_2(\mathbf{Pr}_{\mathcal{E}})$. In particular, the 1-morphism $\mathcal{A}:\mathcal{A}\boxtimes_{\mathcal{E}}\mathcal{A}^{\mathrm{rev}}\nrightarrow\mathcal{E}$ in $\mathrm{Mor}_2(\mathbf{Pr}_{\mathcal{E}})$ is invertible by Theorem \ref{thm:relativeinvertibility}. It is therefore also invertible in $\mathrm{Mor}_2(\mathbf{Pr})$, whence it has all possible adjoints.

For the sake of brevity, we will use the following shortened notations
$$\mathcal{A}^{\mathrm{e}} := \mathcal{A}\boxtimes\mathcal{A}^{\mathrm{rev}},\quad \mathcal{A}^{\mathrm{e}}_{\mathcal{E}} := \mathcal{A}\boxtimes_{\mathcal{E}}\mathcal{A}^{\mathrm{rev}}.$$
It only remains to argue that the 1-morphism $\mathcal{A}^{\mathrm{e}}_{\mathcal{E}}:\mathcal{A}^{\mathrm{e}}\nrightarrow\mathcal{A}^{\mathrm{e}}_{\mathcal{E}}$ has all possible adjoints. In order to do so, we will use the fact that this 1-morphism arises from the canonical braided tensor functor $\mathcal{A}^{\mathrm{e}}\twoheadrightarrow \mathcal{A}^{\mathrm{e}}_{\mathcal{E}}$. In \cite[Section 3.2.2]{H}, it was shown that both the left and the right adjoint to $\mathcal{A}^{\mathrm{e}}_{\mathcal{E}}:\mathcal{A}^{\mathrm{e}}\nrightarrow\mathcal{A}^{\mathrm{e}}_{\mathcal{E}}$ are given by $\mathcal{A}^{\mathrm{e}}_{\mathcal{E}}:\mathcal{A}^{\mathrm{e}}_{\mathcal{E}}\nrightarrow\mathcal{A}^{\mathrm{e}}$. Furthermore, the unit and counit of these adjunctions are the 2-morphisms given by the $\mathcal{A}^{\mathrm{e}}_{\mathcal{E}}$-$\mathcal{A}^{\mathrm{e}}_{\mathcal{E}}$-central bimodules
$$\mathcal{A}^{\mathrm{e}}_{\mathcal{E}}:\mathcal{A}^{\mathrm{e}}_{\mathcal{E}}\boxtimes_{\mathcal{A}^{\mathrm{e}}}\!(\mathcal{A}^{\mathrm{e}}_{\mathcal{E}})^{\mathrm{mop}}\nRightarrow\mathcal{A}^{\mathrm{e}}_{\mathcal{E}},\quad \mathcal{A}^{\mathrm{e}}_{\mathcal{E}}:\mathcal{A}^{\mathrm{e}}_{\mathcal{E}}\nRightarrow\mathcal{A}^{\mathrm{e}}_{\mathcal{E}}\boxtimes_{\mathcal{A}^{\mathrm{e}}}\!(\mathcal{A}^{\mathrm{e}}_{\mathcal{E}})^{\mathrm{mop}},$$
induced by the braided tensor functor $\mathcal{A}^{\mathrm{e}}_{\mathcal{E}}\boxtimes_{\mathcal{A}^{\mathrm{e}}}(\mathcal{A}^{\mathrm{e}}_{\mathcal{E}})^{\mathrm{mop}}\rightarrow\mathcal{A}^{\mathrm{e}}_{\mathcal{E}}$, as well as the $\mathcal{A}^{\mathrm{e}}$-$\mathcal{A}^{\mathrm{e}}$-central bimodules
$$\mathcal{A}^{\mathrm{e}}_{\mathcal{E}}:\mathcal{A}^{\mathrm{e}}\nRightarrow\mathcal{A}^{\mathrm{e}}_{\mathcal{E}},\quad \mathcal{A}^{\mathrm{e}}_{\mathcal{E}}:\mathcal{A}^{\mathrm{e}}_{\mathcal{E}}\nRightarrow\mathcal{A}^{\mathrm{e}},$$ induced by $\mathcal{A}^{\mathrm{e}}\rightarrow \mathcal{A}^{\mathrm{e}}_{\mathcal{E}}$. By \cite[Corollary 4.3]{BJS}, it will suffice to show that the underlying bimodule categories have all adjoints. Thanks to Lemma \ref{lem:separablebimodule} below, it is enough to show that the two tensor functors $\mathcal{A}^{\mathrm{e}}_{\mathcal{E}}\boxtimes_{\mathcal{A}^{\mathrm{e}}}\!(\mathcal{A}^{\mathrm{e}}_{\mathcal{E}})^{\mathrm{mop}}\rightarrow\mathcal{A}^{\mathrm{e}}_{\mathcal{E}}$ and $\mathcal{A}^{\mathrm{e}}\rightarrow \mathcal{A}^{\mathrm{e}}_{\mathcal{E}}$ can be identified with the free module functor over a separable algebra. In the case of the latter functor, this was already done in Lemma \ref{lem:relativetensorasmodules}. More precisely, we identified $\mathcal{A}^{\mathrm{e}}\rightarrow \mathcal{A}^{\mathrm{e}}_{\mathcal{E}}$ with the free $T^R_{\mathcal{E}}(\mathbbm{1})$-module functor, but $T^R_{\mathcal{E}}(\mathbbm{1})$ is separable as $\mathcal{E}$ is separable by assumption.

We claim that the tensor functor $\mathcal{A}^{\mathrm{e}}_{\mathcal{E}}\boxtimes_{\mathcal{A}^{\mathrm{e}}}\!(\mathcal{A}^{\mathrm{e}}_{\mathcal{E}})^{\mathrm{mop}}\rightarrow\mathcal{A}^{\mathrm{e}}_{\mathcal{E}}$ can be identified with the free module functor $$\mathrm{Mod}_{\mathcal{A}^{\mathrm{e}}}(T^R_{\mathcal{E}}(\mathbbm{1})\otimes T^R_{\mathcal{E}}(\mathbbm{1}))\rightarrow \mathrm{Mod}_{\mathcal{A}^{\mathrm{e}}}(T^R_{\mathcal{E}}(\mathbbm{1}))$$ induced by the map of commutative algebras $T^R_{\mathcal{E}}(\mathbbm{1})\otimes T^R_{\mathcal{E}}(\mathbbm{1})\rightarrow T^R_{\mathcal{E}}(\mathbbm{1})$. To see this, we will use that $\mathcal{A}^{\mathrm{e}}_{\mathcal{E}}\simeq \mathrm{Mod}_{\mathcal{A}^{\mathrm{e}}_{\mathcal{E}}}(T^R_{\mathcal{E}}(\mathbbm{1}))$ as tensor categories. It then follows from the universal property of the relative Deligne tensor product that 
$$\mathcal{A}^{\mathrm{e}}_{\mathcal{E}}\boxtimes_{\mathcal{A}^{\mathrm{e}}}\mathcal{A}^{\mathrm{e}}_{\mathcal{E}}\simeq \mathrm{LMod}_{\mathcal{A}^{\mathrm{e}}}(T^R_{\mathcal{E}}(\mathbbm{1}))\boxtimes_{\mathcal{A}^{\mathrm{e}}}\mathrm{Mod}_{\mathcal{A}^{\mathrm{e}}}(T^R_{\mathcal{E}}(\mathbbm{1}))\simeq\mathrm{Mod}_{\mathcal{A}^{\mathrm{e}}}(T^R_{\mathcal{E}}(\mathbbm{1})\otimes T^R_{\mathcal{E}}(\mathbbm{1}))$$
as tensor categories using that $T^R_{\mathcal{E}}(\mathbbm{1})$ is a commutative algebra in $\mathcal{E}\boxtimes\mathcal{E} = \mathcal{Z}_{(2)}(\mathcal{A}^{\mathrm{e}})$. In order to identify the tensor functor $\mathcal{A}^{\mathrm{e}}_{\mathcal{E}}\boxtimes_{\mathcal{A}^{\mathrm{e}}}\mathcal{A}^{\mathrm{e}}_{\mathcal{E}}\rightarrow \mathcal{A}^{\mathrm{e}}_{\mathcal{E}}$, observe that $T^R_{\mathcal{E}}(\mathbbm{1})$ is the monoidal unit of $\mathrm{Mod}_{\mathcal{A}^{\mathrm{e}}}(T^R_{\mathcal{E}}(\mathbbm{1}))\simeq \mathcal{A}^{\mathrm{e}}_{\mathcal{E}}$. In particular, we have that $T^R_{\mathcal{E}}(\mathbbm{1})\otimes T^R_{\mathcal{E}}(\mathbbm{1})\mapsto T^R_{\mathcal{E}}(\mathbbm{1})$ functorially under $\mathcal{A}^{\mathrm{e}}_{\mathcal{E}}\boxtimes_{\mathcal{A}^{\mathrm{e}}}\mathcal{A}^{\mathrm{e}}_{\mathcal{E}}\rightarrow \mathcal{A}^{\mathrm{e}}_{\mathcal{E}}$. This establishes the desired identification. Finally, it only remains to argue that the commutative algebra $T^R_{\mathcal{E}}(\mathbbm{1})$ is separable in $\mathrm{Mod}_{\mathcal{A}^{\mathrm{e}}}(T^R_{\mathcal{E}}(\mathbbm{1})\otimes T^R_{\mathcal{E}}(\mathbbm{1}))$. But, the canonical map $T^R_{\mathcal{E}}(\mathbbm{1})\otimes T^R_{\mathcal{E}}(\mathbbm{1})\rightarrow T^R_{\mathcal{E}}(\mathbbm{1})$ admits a section as a map of bimodules by separability. This shows that $T^R_{\mathcal{E}}(\mathbbm{1})$ is a direct summand of $T^R_{\mathcal{E}}(\mathbbm{1})\otimes T^R_{\mathcal{E}}(\mathbbm{1})$. In particular, the map $T^R_{\mathcal{E}}(\mathbbm{1})\otimes T^R_{\mathcal{E}}(\mathbbm{1})\rightarrow T^R_{\mathcal{E}}(\mathbbm{1})$ is the projection onto a summand. It follows that $T^R_{\mathcal{E}}(\mathbbm{1})$ is separable as an algebra in $\mathrm{Mod}_{\mathcal{A}^{\mathrm{e}}}(T^R_{\mathcal{E}}(\mathbbm{1})\otimes T^R_{\mathcal{E}}(\mathbbm{1}))$, thereby concluding the proof.
\end{proof}

We have made use of the following technical result in the course of the proof of the theorem above.

\begin{Lemma}\label{lem:separablebimodule}
Let $\mathcal{A}$ be a finite braided tensor category, and let $A$ be a commutative separable algebra in $\mathcal{A}$. Let us write $F:\mathcal{A}\twoheadrightarrow \mathrm{Mod}_{\mathcal{A}}(A)=:\mathcal{B}$ for the free $A$-module tensor functor. Then, the 2-morphism $\mathcal{B}:\mathcal{A}\nRightarrow\mathcal{B}$ in $\mathrm{Mor}_2^{\mathrm{pre}}(\mathbf{Pr})$ has all possible adjoints.
\end{Lemma}
\begin{proof}
That this bimodule category has left and right adjoints was proven in \cite[Theorem 5.13]{BJS}, see also \cite[Proposition 3.2.1]{DSPS:dualizable}. The left adjoint is given by $\mathrm{LMod}_{\mathcal{A}}(A):\mathcal{B}\nrightarrow \mathcal{A}$ whereas the right adjoint is $\mathrm{Mod}_{\mathcal{A}}(A):\mathcal{B}\nrightarrow \mathcal{A}$ by \cite[Theorem 5.8]{BJS}. The units and counits of these two adjunctions are given by the bimodule functors
$$\mathrm{Mod}_{\mathcal{A}}(A)\boxtimes_{\mathcal{B}}\mathrm{LMod}_{\mathcal{A}}(A)\simeq\mathcal{B}\rightarrow \mathcal{A},\, \mathrm{Mod}_{\mathcal{A}}(A)\boxtimes_{\mathcal{A}}\mathrm{Mod}_{\mathcal{A}}(A)\simeq\mathrm{Mod}_{\mathcal{A}}(A\otimes A)\rightarrow \mathcal{B},$$
given by the forgetful functor and the functor induced by the $A\otimes A$-$A$-bimodule $A$, as well as
$$\mathcal{B}\rightarrow\mathrm{LMod}_{\mathcal{A}}(A)\boxtimes_{\mathcal{A}}\mathrm{Mod}_{\mathcal{A}}(A)\simeq\mathrm{Bimod}_{\mathcal{A}}(A),\, \mathcal{A}\rightarrow\mathrm{Mod}_{\mathcal{A}}(A)\boxtimes_{\mathcal{B}}\mathrm{Mod}_{\mathcal{A}}(A)\simeq\mathcal{B}$$
given by endowing an $A$-module with the natural $A$-$A$-bimodule structure, and the free $A$-module functor. All of the four functors above are compact preserving functors between finite categories. In particular, they have left and right adjoints as plain functors if and only if they are exact. Moreover, as both $\mathcal{A}$ and $\mathcal{B}$ have duals for compact objects, it follows from a standard argument recorded for instance in \cite[Corollary 2.13]{DSPS:balanced} that, provided that these adjoints exist, they are automatically compatible with the bimodule structures.

We therefore only have left to show that the four functor above are exact. For the forgetful functor $\mathcal{B}\rightarrow \mathcal{A}$, this is immediate. In fact, the forgetful functor reflects exactness. Using that the forgetful functor $\mathrm{Bimod}_{\mathcal{A}}(A)\rightarrow\mathcal{A}$ also reflects exactness, we find that $\mathcal{B}\rightarrow \mathrm{Bimod}_{\mathcal{A}}(A)$ is also exact. Then, exactness of the free $A$-module functor $\mathcal{A}\rightarrow\mathcal{B}$ holds by exactness of the tensor product of $\mathcal{A}$. Finally, note that that $A$ is projective, whence flat, as a left $A\otimes A$-module because $A$ is separable. It follows that the functor $\mathrm{Mod}_{\mathcal{A}}(A\otimes A)\rightarrow \mathcal{B}$ is exact.
\end{proof}

\begin{Example}
We give concrete examples of finite braided tensor categories for which Theorem \ref{thm:separablesymmetriccenter4dualizable} applies, but \cite[Theorem 1.10]{BJS} and \cite[Theorem 1.1]{BJSS} do not. Setting $\mathbbm{k} = \mathbb{C}$, and letting $G$ be a simple algebraic group and $q$ a root of unity, we have recalled in Example \ref{ex:smallquantumgroup} above the construction of the finite braided tensor categories $\mathrm{Rep}(G_q)\boxtimes_{\mathrm{Tann}_q}\!\!\mathrm{Vec}$ introduced in \cite{Neg:arbitrary}. Taking $G=\mathrm{SL}_2$ and $q$ a root of unity of odd order, it was shown in \cite[Example 10.2]{Neg:arbitrary} that $$\mathrm{Rep}(G_q)\boxtimes_{\mathrm{Tann}_q}\!\!\mathrm{Vec}\simeq \mathrm{sVec}.$$ It therefore follows from Theorem \ref{thm:separablesymmetriccenter4dualizable} that these finite braided tensor categories are fully dualizable objects of $\mathrm{Mor}_2^{\mathrm{pre}}(\mathbf{Pr})$. More examples are given in \cite[Section 10.2]{Neg:arbitrary}.
\end{Example}

In characteristic zero, the converse of Theorem \ref{thm:separablesymmetriccenter4dualizable} holds. In fact, we have the following observation in any characteristic.

\begin{Proposition}\label{prop:symmetriccentersemisimple}
Over an algebraically closed field, a finite braided tensor category is a fully dualizable object of $\mathrm{Mor}_2^{\mathrm{pre}}(\mathbf{Pr})$ only if its symmetric center is finite semisimple.
\end{Proposition}
\begin{proof}
Let $\mathcal{A}$ be a finite braided tensor category that is 4-dualizable as an object of $\mathrm{Mor}_2^{\mathrm{pre}}(\mathbf{Pr})$. Then, the underlying category of $\mathcal{Z}_{(2)}(\mathcal{A})$ is not only finite, but also a 2-dualizable object of $\mathbf{Pr}$ by \cite[Proposition 3.7]{BJSS}. Thanks to \cite[Theorem A.22]{BDSPV}, which builds on \cite[Theorem 2.5]{Til}, it follows that $\mathcal{Z}_{(2)}(\mathcal{A})$ is finite semisimple as claimed.
\end{proof}

\noindent Recalling that, over an algebraically closed field of characteristic zero, every finite semisimple tensor category is separable thanks to \cite{ENO1}, we obtain:

\begin{Corollary}
Over an algebraically closed field of characteristic zero, a finite braided tensor category is a fully dualizable object of $\mathrm{Mor}_2^{\mathrm{pre}}(\mathbf{Pr})$ if and only if its symmetric center is finite semisimple.
\end{Corollary}

We conjecture that the converse of Theorem \ref{thm:separablesymmetriccenter4dualizable} holds in any characteristic, and, more generally, over any perfect field.

\begin{Conjecture}
Over any algebraically closed field, a finite braided tensor category is a fully dualizable object of the Morita 4-category $\mathrm{Mor}_2^{\mathrm{pre}}(\mathbf{Pr})$ of braided pre-tensor categories only if its symmetric center is separable.
\end{Conjecture}

In positive characteristic, the separability hypothesis on the symmetric center cannot be dropped.

\begin{Proposition}\label{prop:fulldualizbailityseparabilitysymmetric}
Over an algebraically closed field, a finite symmetric tensor category $\mathcal{E}$ is a fully dualizable object of $\mathrm{Mor}_2^{\mathrm{pre}}(\mathbf{Pr})$ if and only if it is separable.
\end{Proposition}
\begin{proof}
The backward direction is \cite[Theorem 1.10]{BJS} and the comment thereafter. For the forward direction, we already have seen that $\mathcal{E}$ is finite semisimple in Proposition \ref{prop:symmetriccentersemisimple}. Then, observe that the Harish-Chandra category $\mathrm{HC}(\mathcal{E})$ yields a 1-morphisms $\mathrm{HC}(\mathcal{E}):\mathrm{Vec}\nrightarrow\mathrm{Vec}$ that has all adjoints. Said differently, the finite pre-tensor category $\mathrm{HC}(\mathcal{E})$ is a fully dualizable object of the symmetric monoidal Morita 3-category $\mathrm{Mor}_1^{\mathrm{pre}}(\mathbf{Pr})$. In particular, it follows that the underlying category of its Drinfeld center $\mathcal{Z}(\mathrm{HC}(\mathcal{E}))$ is a fully dualizable object of $\mathbf{Pr}$. We therefore find from \cite[Theorem A.22]{BDSPV} that $\mathcal{Z}(\mathrm{HC}(\mathcal{E}))$ is finite semisimple. But, as $\mathcal{E}$ is symmetric monoidal, so is $\mathrm{HC}(\mathcal{E})$. As a consequence, there is a fully faithful braided tensor functor $\mathrm{HC}(\mathcal{E})\hookrightarrow \mathcal{Z}(\mathrm{HC}(\mathcal{E}))$, from which we deduce that $\mathrm{HC}(\mathcal{E})$ is finite semisimple. But there are equivalences of plain categories $$\mathrm{HC}(\mathcal{E})\simeq \mathrm{Mod}_{\mathcal{E}}(\mathbb{F}_{\mathcal{E}})\simeq \mathcal{Z}(\mathcal{E}),$$ so that $\mathcal{Z}(\mathcal{E})$ is also finite semisimple, concluding the proof.
\end{proof}

\addcontentsline{toc}{section}{Bibliography}
\bibliography{bibliography.bib}

\end{document}